\documentclass[a4paper, 12pt]{amsart}
\usepackage{amssymb, amsmath, amsthm, enumerate, color, subcaption, tipa, upgreek, xcolor}
\usepackage{amstext, tikz, appendix, mathrsfs, mathabx, mathtools, pgfplots, stmaryrd}
\usepackage{csquotes}
\usepackage[colorlinks=true, breaklinks=true, linkcolor=black, citecolor=blue, urlcolor=red]{hyperref} 
\usepackage[english]{babel}
\usepackage{forest, adjustbox}
\usepackage{tikz-cd}

\numberwithin{equation}{section}
\newtheorem{letterthm}{Theorem}

\newtheorem{lettercor}[letterthm]{Corollary}

\newtheorem{theorem}{Theorem}[section]

\newtheorem{lemma}[theorem]{Lemma}
\newtheorem{corollary}[theorem]{Corollary}
\newtheorem{proposition}[theorem]{Proposition}

\newtheorem{observation}[theorem]{Observation}

\theoremstyle{definition} 
\newtheorem{definition}[theorem]{Definition}
\newtheorem{notation}[theorem]{Notation}
\newtheorem{remark}[theorem]{Remark}
\newtheorem{example}[theorem]{Example}
\newtheorem{construction}[theorem]{Construction}

\newcommand{\tb}{\textcolor{black}}

\newcommand{\act}{\curvearrowright}
\newcommand{\at}{atomic}

\newcommand{\bluebullet}{\textcolor{blue}{\bullet}}
\newcommand{\greenbullet}{\textcolor{green}{\bullet}}
\newcommand{\redbullet}{\textcolor{red}{\bullet}}
\newcommand{\cC}{\mathcal C}

\newcommand{\C}{\mathbf C}

\DeclareMathOperator{\Comm}{Comm}

\newcommand{\fF}{\mathfrak F}
\DeclareMathOperator{\Fix}{Fix}

\newcommand{\fH}{\mathfrak H}

\newcommand{\scrH}{\mathscr H}

\DeclareMathOperator{\id}{id}
\DeclareMathOperator{\Ind}{Ind}

\newcommand{\fK}{\mathfrak K}
\newcommand{\scrK}{\mathscr K}

\DeclareMathOperator{\Leaf}{Leaf}
\DeclareMathOperator{\length}{length}
\newcommand{\limw}{\lim_{w\in W}}

\DeclareMathOperator{\Mon}{Mon}
\newcommand{\N}{\mathbf{N}}

\newcommand{\NInd}{\text{Ind-mixing}}

\newcommand{\cO}{\mathcal O}

\newcommand{\ot}{\otimes}
\newcommand{\ov}{\overline}
\newcommand{\cP}{\mathcal P}

\newcommand{\cQ}{\mathcal{Q}}
\newcommand{\R}{\mathbf{R}}

\DeclareMathOperator{\Ran}{Ran}

\DeclareMathOperator{\Root}{Root}
\newcommand{\bS}{\mathbf S^1}
\DeclareMathOperator{\Stab}{Stab}

\newcommand{\cspan}{\overline{\textrm{span}}}

\DeclareMathOperator{\supp}{supp}
\newcommand{\fT}{\mathfrak T}
\DeclareMathOperator{\tar}{tar}
\newcommand{\ti}{\tilde}

\newcommand{\cU}{\mathcal U}

\newcommand{\fU}{\mathfrak U}

\newcommand{\fV}{\mathfrak V}
\newcommand{\varep}{\varepsilon}
\DeclareMathOperator{\Ver}{Ver}
\newcommand{\fW}{\mathfrak W}
\newcommand{\fX}{\mathfrak X}
\newcommand{\scrX}{\mathscr X}

\newcommand{\fY}{\mathfrak Y}

\newcommand{\Y}{\wedge}
\newcommand{\Z}{\mathbf{Z}}
\newcommand{\fZ}{\mathfrak Z}

\DeclarePairedDelimiterX{\norm}[1]{\lVert}{\rVert}{#1}

\providecommand{\keywords}[1]{\textbf{\textit{Index terms---}} #1}

\begin{document}
	
\title[Atomic Pythagorean representations]{Decomposition of Pythagorean representations of R.~Thompson's groups}
\thanks{
	AB is supported by the Australian Research Council Grant DP200100067.\\
	DW is supported by an Australian Government Research Training Program (RTP) Scholarship.}
\author{Arnaud Brothier and Dilshan Wijesena}
\address{Arnaud Brothier, Dilshan Wijesena\\ School of Mathematics and Statistics, University of New South Wales, Sydney NSW 2052, Australia}
\email{arnaud.brothier@gmail.com\endgraf
	\url{https://sites.google.com/site/arnaudbrothier/}}

\begin{abstract}
We continue to study Pythagorean unitary representation of Richard Thompson's groups $F,T,$ and $V$ that are built from a single isometry from a Hilbert space to its double.
By developing powerful diagrammatically based techniques we show that each such representation splits into a diffuse and an atomic parts.
We previously proved that the diffuse part is Ind-mixing: it does not contain induced representations of finite-dimensional ones.
We fully decompose the atomic part: the building blocks are monomial representations arising from a precise family of parabolic subgroups of $F$.
\end{abstract}

\maketitle





\section*{Introduction}
Thompson's groups $F,T,V$ are some of the most fascinating groups which naturally appear in various branches of mathematics, see the expository article \cite{Cannon-Floyd-Parry96}.
Understanding groups is done by constructing and studying their actions that are manifestations of the symmetries they encode. 
Thompson groups are complicated groups for which constructing actions is difficult which explains our lack of understanding.
Although, they are fraction groups of very simple categories of diagrams that are made of binary trees: a point of view first appearing in the work of Brown (but most certainly goes back to Thompson) that we shall deeply exploit in this article \cite{Brown87}.
Indeed, the category giving $F$ is generated by the single tree with two leaves. 
Jones discovered that a single morphism in any category yields an action of Thompson group $F$ which may extend to $T$ and $V$ under additional assumptions \cite{Jones17,Jones18}.

This is a particular case of Jones' technology that has been already successfully applied for constructing actions on operator algebras, actions on groups, and unitary actions which will be our main use in this article \cite{Jones18-Hamiltonian, Brothier-Stottmeister19,Brothier22,Brothier21,Jones21,Brothier-Jones19-HnonT}. 
Beyond producing actions of the Thompson groups we may use this technology to produce among other new knot invariants, obtaining natural subgroups of the Thompson groups, and for studying certain non-commutative probabilities \cite{Jones18,Grymski-Peters22,Golan-Sapir17,Aiello-Nagnibeda21,Kostler-Krishnan-Wills20,Kostler-Krishnan22}.
Finally, this technology is useful for studying other groups built from categories such as various Thompson-like groups \cite{Brothier23,Brothier22-FSI,Brothier22-FSII}.
We invite the interested to consult the following three recent surveys for more information \cite{Jones19-survey, Brothier19-survey, Aiello22-survey}.

If $\fH$ is a complex Hilbert space and $R:\fH\to\fH\oplus\fH$ is an linear isometry, then we can construct a {\it Pythagorean} unitary representation (in short P-representation) of the three Richard Thompson groups $F,T,$ and $V$ \cite{Brothier-Jones19-Pyth}.
The isometry $R$ decomposes into the direct sum of two operators $A,B\in B(\fH)$ satisfying the {\it Pythagorean identity}:
\begin{equation}\label{eq:Pyth}
A^*A+B^*B=\id_\fH.
\end{equation}
The {\it Pythagorean algebra} is the universal C*-algebra generated by two operators satisfying relation \eqref{eq:Pyth} \cite{Brothier-Jones19-Pyth}.
By definition, the representations of $P$ are in bijection with the P-representations of $F$ and a canonical process extends these representations to $T$ and $V$.

Further, the celebrated Cuntz algebra $\cO$ is an obvious quotient of $P$ (additionally require that $A$ and $B$ are surjective partial isometries) and thus any representation of $\cO$ provides a representation of $P$ \cite{Cuntz77}.
What is very surprising is that a canonical lifting construction permits to construct a representation of $\cO$ from one of $P$ and moreover all representations of $\cO$ arise in that way \cite[Theorem 7.1]{Brothier-Jones19-Pyth}.
Hence, $P$ is not only useful for studying the Thompson groups but also for the Cuntz algebra.
This strengthens in particular a previous remarkable connection between the Thompson groups and the Cuntz algebra \tb{due independently by Birget and Nekrashevych \cite{Birget04,Nekrashevych04}.}

In most practical examples $\fH=\C^d$ and $A,B$ are $d$ by $d$ matrices. Now, the associated Hilbert space $\scrH$ on which $V$ and $\cO$ act contains $\fH$ but is larger: it is always infinite dimensional (it can thought as being equal to all trees whose leaves are decorated by vectors of $\fH$). 
The force of this construction is to be able to easily construct representations of $F,T,V,$ and $\cO$ using finite-dimensional linear algebra. 
Moreover, properties of the infinite-dimensional representations of $F,T,V,\cO$ can be read from the initial data $(A,B,\fH)$ which is much easier to apprehend.

{\bf Brief outline of the article.}
In this article we continue our systematic study of P-representations initiated in \cite{Brothier-Wijesena22}.
We focus on representations of the group $F$ (the smallest Thompson group) but our study can easily be translated (and in fact simplified) to treat the $T$ and $V$-cases.
We previously introduce the key notion of a {\it diffuse} Pythagorean pair $(A,B)$ (i.e.~the words in $A,B$ tends to zero for the strong operator topology, e.g.~$A,B$ have norms strictly smaller than 1) and proved that under this assumption the associated P-representation $\sigma_{A,B}$ of $F$ is {\it Ind-mixing} (i.e.~does not contain the induced representation $\Ind_H^F\theta$ for $\theta$ finite-dimensional representation of subgroup $H\subset F$) \cite{Brothier-Wijesena22}.
In that case we say that $\sigma_{A,B}$ is a {\it diffuse} P-representation.
In this present article we define a negation of being diffuse called {\it atomic} for a Pythagorean pair (in short P-pair) of operators $(A,B)$ and for its associated P-representation.
The first main result of this article shows that given any P-representation $\sigma$ there is a direct decomposition $\sigma=\sigma_d\oplus \sigma_a$ where $\sigma_d,\sigma_a$ are themselves P-representations and are the diffuse and atomic parts of $\sigma$, respectively. We say that $\sigma_a$ is an atomic P-representation just like $\sigma_d$ is a diffuse P-representation.
The second main result of this article consists in describing precisely the atomic part: we show that it is built from a precise list of monomial representations associated to parabolic subgroups and certain one-dimensional representations.
When the initial Hilbert space is of finite dimension $d$ we obtain that $\sigma_a$ is a finite direct sum of irreducible ones and provides the complete list of admissible representations that may occur.

{\bf Detailed content of the article and main results.}
In this article we consider P-representations $(\sigma,\scrH)$ of Thompson's group $F$ constructed from P-pairs $(A,B)$ acting on some Hilbert space $\fH$. 
If $p$ is an infinite binary sequence (called a {\it ray}) and $p_n$ the sub-word of the first $n$ letters, then $p_n$ defines an operator of $\fH$ (using the letters $A$ and $B$).
We define $\fU\subset\fH$ the subset of vectors $\xi$ satisfying that $\lim_n\| p_n \xi\|=0$ for all ray $p$.
This is a Hilbert subspace that is closed under the action of $A$ and $B$ and thus provides a sub-representation $(\sigma_\fU,\scrH_\fU)$ of $(\sigma,\scrH)$.
The pair of restrictions $A|_\fU,B|_\fU$ provides a {\it diffuse} P-pair and thus $\tb{\sigma_d:=} \sigma_\fU$ is Ind-mixing as proved in our previous article \cite{Brothier-Wijesena22}.
\tb{We additionally construct a subspace $\fV \oplus \fW \subset \fH$ which is also closed under $A$ and $B$ (but not necessarily topologically closed), and is orthogonal to $\fU$. As above, we consider the sub-representation $\sigma_a := \sigma_{\fV\oplus \fW}$ that we call the {\it atomic} part of $\sigma$. Furthermore, we define the subspace $\fZ := (\fU \oplus \fV \oplus \fW)^\perp$ which is not necessarily closed under $A$ and $B$.
This yields the following (canonical) orthogonal decomposition:
$$\fH=\fU\oplus\ov\fV\oplus\ov \fW\oplus \fZ.$$}
First, we prove that \tb{when $\fZ$ is finite-dimensional we have the decomposition $\sigma = \sigma_d \oplus \sigma_a$. This shows that the subspace $\fZ\subset \fH$ is not "viewed" by the P-representation $\sigma:F\act\scrH$}.
Second, we completely characterise the building blocks of \tb{$\sigma_a$} by providing a precise list of {\it monomial} representations.

More precisely, consider the usual action of $F$ on the Cantor space equal to the set of infinite binary strings $\{0,1\}^\N$ (i.e.~the set of all rays in the rooted infinite regular binary tree) that we view over $[0,1]$ via the classical continuous $F$-equivariant surjection 
$$(x_n)_n\mapsto \sum_{n\geq 1} 2^{-n}x_n.$$
(We use $x$ rather than $p$ as $x_n$ denotes the $n$-th letter while $p_n$ denotes the first $n$ letters.)
The previous map permits to identify $\{0,1\}^\N$ with $[0,1]$ plus an extra copy of the dyadic rational in the open interval $(0,1)$.
This allows us to define derivations $g'(p)$ for $g\in F$ and $p$ a ray.
We consider the {\it parabolic subgroup} 
$$F_p:=\{g\in F:\ g(p)=p\}$$ 
and its (normal) subgroup 
$$\widehat F_p=\{g\in F_p:\ g'(p)=1\}.$$
Given any modulus one complex number $\varphi\in S^1$ and ray $p$ we form a one-dimensional representation
$$\chi_\varphi^p:F_p\to S^1, g\mapsto \varphi^{\log_2(g'(p))}.$$
We associate to it the induced representation $\Ind_{F_p}^F\chi_\varphi^p$ of $F$ acting on $\ell^2(F/F_p)$.
In particular, when $\varphi=1$, then $\Ind_{F_p}^F\chi_\varphi^p=\lambda_{F/F_p}$ is the quasi-regular representation associated to $F_p\subset F$ and when $p$ is an endpoint (i.e.~$p=0$ or $1$), then $F_p=F$ and $\chi_\varphi^p:=\Ind_{F_p}^F\chi_\varphi^p$ is a one-dimensional representation.
We have the following general theorem.

\begin{letterthm}\label{theo:A}
Let $(A,B)$ be a Pythagorean pair acting on $\fH$ with associated group representation $\sigma:F\act \scrH$. 
The following assertions are true.
\begin{enumerate}
\item \tb{We have the following decomposition of $\fH$ into certain orthogonal topologically closed subspaces:
$$\fH=\fU\oplus \overline{\fV} \oplus \overline{\fW} \oplus \fZ$$
(see Definition \ref{subspace of fH defintiion} for details).}
\tb{\begin{enumerate}
\item The subspace $\fU$ (resp.~$\fV\oplus\fW$) is closed under the action of $A$ and $B$.
\item The representation $\sigma_d:F\act \scrH_\fU$ associated to $\fU$ is diffuse.
\item The representation $\sigma_a:F\act \scrH_{\fV\oplus\fW}$ associated to $\fV\oplus\fW$ is atomic.
\item \label{subitem:sigma-decomp} When $\dim(\fZ)<\infty$, then $\sigma$ decomposes as $\sigma_d\oplus\sigma_a$.
\end{enumerate}}
\item When $\fH$ is separable, then $\sigma_a$ is a direct integral of representations appearing in:
\begin{equation}\label{eq:list}\{ \Ind_{F_p}^F\chi_\varphi^p:\ p \text{ ray }, \varphi\in S^1 \} \cup \{ \lambda_{F/\widehat F_p}:\ p \text{ rational ray } \}.
\end{equation}
\item When $\dim(\fH)<\infty$, then $\fW=\{0\}$ and $\sigma_a$ is a finite direct sum of representations appearing in:
\begin{equation}\label{eq:list-FD}\{ \Ind_{F_p}^F\chi_\varphi^p:\ p \text{ rational ray}, \varphi\in S^1 \}.\end{equation}
\end{enumerate}
\end{letterthm}

\tb{In practice, we are primarily interested in when $\fH$ is finite-dimensional where the theory vastly simplifies and yields elegant results (see \cite{Brothier-Wijesena23} where we provide a complete classification of Pythagorean representations when $\fH$ is finite-dimensional). In this case, $\sigma=\sigma_\fU\oplus\sigma_\fV$ decomposes into a diffuse part and a finite direct sum of monomial representations.} 

Note that the sets of representations provided in items 2 and 3 are optimal: for any representation $\pi$ appearing in \eqref{eq:list} (resp.~\eqref{eq:list-FD}) there exists a P-pair $(A,B)$ acting on a separable (resp.~finite-dimensional) Hilbert space $\fH$ with associated representation $(\sigma,\scrH)$ satisfying that $\pi\subset \sigma$.
Better, we can realise any representation of the list above as a P-representation except few of them (that we completely characterise) which must come in pair (e.g.~$1_F\oplus \lambda_{F/F_{p}} $ with $p=\dots 001$ is isomorphic to the P-representation given by $A=1,B=0$ acting on $\fH:=\C$ but $1_F$ and $\lambda_{F/F_{p}}$ are not isomorphic to a P-representation), see Theorems \ref{sigma_p decompose finite theorem}, \ref{sigma_p decompose infinite theorem} and \ref{sigma_fW decompose theorem}.
Moreover, we can provide a more precise statement for a fixed dimension of $\fH$, see Theorem \ref{sigma_p decompose finite theorem}.

In particular, the class of P-representations is not closed under taking sub-representations.
Although, it is closed under taking (infinite) direct sums and even direct integrals. Indeed, simply take the direct sum or direct integral of the underlying P-pairs. 
The subclass of atomic ones is closed under taking infinite direct sums but is not under taking direct integrals, see Example \ref{ex:int-atomic}.

Using the Mackey-Shoda criterion we observe that all the representations of \eqref{eq:list} except the quasi-regular $\lambda_{F/\widehat F_p}$ are irreducible since $F_p\subset F$ is a self-commensurated subgroup. Moreover, it is not hard to classify them up to unitary conjugacy.
Hence, this provides a complete understanding of all atomic P-representations.
From this theorem we can conclude that the only irreducible finite-dimensional sub-representations of a Pythagorean one are the $\chi_\varphi^p:g\mapsto \varphi^{\log_2g'(p))}$ for $\varphi\in S^1$ and $p$ being $0$ or $1$.
\tb{This corroborates the classical result saying that all finite-dimensional irreducible representations of $F$ are one-dimensional, see \cite{Dudko-Medynets-14} for a proof}.
Furthermore, we have the complete list of all the induced representations $\Ind_H^F\theta$ with $H\subset F, \theta:H\act \fK$ finite-dimensional that are contained in P-representations.
There are very few of them if we consider for instance all quasi-regular representations of subgroups of $F$, see the discussion in Section 3 of \cite{Brothier-Wijesena22}.

\tb{Observe item (1)(d) does not hold when $\fZ$ is infinite-dimensional, see Remark \ref{rem:fH-decomp} for a counter-example. However, in this case we are still able to construct sub-representations $\sigma_d$, $\sigma_a$ which decompose $\sigma$ into diffuse and atomic parts: $\sigma = \sigma_d \oplus \sigma_a$, see Remark \ref{rem:sigma_fW-decompose}. The diffuse part $\sigma_d$ contains $\sigma_\fU$ and is a diffuse Pythagorean representation while the atomic part $\sigma_a$ contains $\sigma_{\fV\oplus\fW}$ and is a Pythagorean representation that is a direct integral of representations appearing in (\ref{eq:list}). Hence, this provides a generalisation of item (1)(d) for when $\fZ$ is infinite-dimensional.}

We previously proved that a diffuse P-representation is Ind-mixing (i.e.~does not contain $\Ind_H^F\theta$ with $\dim(\theta)<\infty$) and similarly provided a criteria for weakly mixing representations (i.e.~does not contain a non-zero finite dimensional sub-representation) \cite{Brothier-Wijesena22}. The above implies the converse yielding the following two characterisations.

\begin{lettercor}\label{cor:B}
Consider a Pythagorean pair $(A,B)$ acting on $\fH$ and the associated P-representation $\sigma:F\act \scrH$.
The following assertions are true.
\begin{enumerate}
\item The representation $\sigma$ is weakly mixing if and only if $\lim_{n}A^n\xi=\lim_n B^n\xi=0$ for all $\xi\in\fH$.
\item The representation $\sigma$ is Ind-mixing if and only if $\lim_n p_n\xi=0$ for all $\xi\in\fH$ and ray $p$.
\end{enumerate}
\end{lettercor}

{\bf Well-definedness of diffuse and atomic P-representations.}
Note, from initial definitions in the current article and the authors' previous article \cite{Brothier-Wijesena22}, it is not immediately clear (nor trivial) that the notion of diffuse and atomic representations are well-defined among their equivalence classes. That is, if $\sigma$ is a diffuse P-representation induced by a diffuse P-pair, it could be possible that there exists an atomic P-pair which induces a (atomic) P-representation equivalent to $\sigma$. However, the above Corollary shows that this cannot occur. Indeed, a representation is diffuse if and only if it is $\NInd$ which is a property preserved by the class of representations. Furthermore, by item $(1)$ in Theorem \ref{theo:A}, we can also conclude that the same holds for atomic representations.

{\bf Plan of the article.}
Section \ref{sec:preliminaries} is a preliminary section on unitary representations, direct integrals, strong operator topology, Thompson's groups, parabolic subgroups of $F$, and Pythagorean representations. 
Moreover, we recall an important class of partial isometries introduced in our previous paper \cite{Brothier-Wijesena22}. 

In Section \ref{decompose pythag rep section} we develop deeply our diagrammatic machinery for decomposing Pythagorean representations. 
This is the longest and most technical part of the paper.
Recall, that a Pythagorean pair $(A,B)$ acting on $\fH$ defines a unitary representation $\sigma:F\act \scrH$ (note that $\fH$ and $\scrH$ are different and $\scrH$ is ``larger'' than $\fH$).
The pair $(A,B)$ can be interpreted as an action of the free monoid $\Mon(a,b)\act \fH$ in two generators $a,b$.
The main idea is to use this (possibly finite-dimensional) monoid action for decomposing the infinite-dimensional representation $\sigma:F\act\scrH$. 

This is done by observing asymptotic behaviours of actions of words.
A ray $p$ in the rooted infinite binary tree is nothing else than an infinite binary sequence giving a sequence of elements $p_n\in \Mon(a,b)$ with $n$ letters.
By considering $\lim_n\|p_n\xi\|$ for rays $p$ and vectors $\xi\in\fH$ we define subspaces of $\fH$ and obtain a direct sum decomposition
$$\fH=\fU\oplus\overline{\fV}\oplus\overline{\fW}\oplus \fZ.$$
Moreover, \tb{when $\fZ$ is finite-dimensional} we may discard the remaining part $\fZ$ without affecting the associated representation $\sigma:F\act\scrH$. 
The space $\fU$ is defined to be the set of vectors $\xi\in\fH$ satisfying that $\lim_n p_n\xi=0$ for all ray $p$.
The spaces $\fV$ and $\fW$ are weak negations of $\fU$ where the sequence of operators $(p_n)_n$ does not converge to zero for certain rays $p$, see Definition \ref{subspace of fH defintiion} for details.
Moreover, we can sub-divide $\fV$ as a direct sum $\oplus_p \fV^p$ where $p$ ranges over all rays (in fact rays up to changing finite prefixes). 
The space $\fV^p$ contains $\{\xi\in\fH:\ \lim_n\|p_n\xi\|=\|\xi\|\}$ but is in general larger.
The space $\fW$ only appears when $\fH$ is infinite-dimensional; it is more complicated to define and to apprehend and cannot be sub-divided over the rays

Vectors of $\scrH$ can be thought as finite binary trees $t$ whose leaves are decorated by vectors of $\fH$.
If $\fK\subset\fH$ is a subspace closed under $A$ and $B$ (hence a $\Mon(a,b)$-subspace), then we obtain a subspace $\scrH_\fK\subset\scrH$ where now we only consider trees decorated by elements in $\fK$ (rather than in all $\fH$). 
This defines a sub-representation of $F$ which is nothing else than the P-representation associated to the restrictions of $A$ and $B$ to $\fK$.
In particular, $\fU,\fV^p$, and $\fW$ define P-representations of $F$ and moreover \tb{when $\fZ$ is finite-dimensional then} $\sigma:F\act\scrH$ is the direct sum of those.
From there we deduce the first item of Theorem \ref{theo:A} where the diffuse part corresponds to $\fU$ and the atomic part corresponds to $\fV\oplus\fW$ (see Theorem \ref{complete scrH decomposition theorem}).
We end this section by studying very precisely a few key examples which is a preparation to the final section. 

In Section \ref{rep from rays section}, we completely decompose and describe the representations of $F$ associated to $\fV^p$ and $\fW$.
This gives items 2 and 3 of Theorem \ref{theo:A}.
When $\fH$ is finite-dimensional, then $\fW$ is trivial and $\fV$ decomposes a finite direct sum of some $\fV^p$.
Moreover, $p$ must be eventually periodic or equivalently is sent to a rational number via the usual surjection from binary sequences to $[0,1]$.
Further, the length of a period of $p$ must be smaller than the dimension of $\fV$.
Now, the representation of the group $F$ associated to $\fV^p$ is a finite direct sum of induced representations of the form $\Ind_{F_p}^F\chi_\varphi^p$ described earlier. 
Similarly, we can perform a description of the representation of $F$ associated to $\fV^p$ and $\fW$ when $\fH$ is separable which uses direct integral of representations rather than direct sums.
\tb{Additionally, we explain how to define the diffuse and atomic parts of $\sigma$ when $\fZ$ is infinite-dimensional.}
We finish the article by deducing Corollary \ref{cor:B}. It is an automatic consequence of the main theorem of our previous article (when diffuse is shown to imply Ind-mixing) and Theorem \ref{theo:A}.

\tb{
\textbf{Comparison to the atomic representations of Dutkay, Haussermann, and Jorgensen.} 
Dutkay, Haussermann, and Jorgensen introduced the family of purely atomic representations of the Cuntz algebra $\cO$ in \cite{DHJ-atomic15} which are in fact precisely the extension of the atomic representations to $\cO$ considered in the current article. Similarly to the current article, in \cite{DHJ-atomic15} the authors classify the irreducible classes of purely atomic representations of $\cO$. However, despite sharing some common features, these two studies are different in nature.
\\
Notably, a key difference is that in \cite{DHJ-atomic15} the classification of purely atomic representations is accomplished by studying directly the larger infinite-dimensional Hilbert space $\scrH$ via certain projections associated to singleton sets (these are analogous to the projections $\rho_p$ for a ray $p$ defined in the present article). In contrast, our study is primarily focused on classifying atomic representations by only studying the smaller Hilbert space $\fH$ (which is often taken to be finite-dimensional in practice). 
\\
Furthermore, we have been able to provide a precise description of the atomic representations as a direct sum (or direct integral) of monomial representations associated to certain subgroups of the Thompson groups. The derivation of this description takes up the majority of the proofs of the main theorems, where else irreducibility and equivalence follows rather easily from classical results. 
\\
Lastly, we note that the equivalence classes of atomic representations of $F$ differs to its extension to $\cO$. For example, consider the following two pairs of Pythagorean pairs $(A_1, B_1)$ and $(A_2, B_2)$:
\[A_1 = \begin{pmatrix} 0 & 1\\0 & 0 \end{pmatrix},\
B_1 = \begin{pmatrix} 0 & 0\\1 & 0 \end{pmatrix} \textrm{ and }
A_2 = \begin{pmatrix} 0 & -1\\0 & 0 \end{pmatrix},\
B_2 = \begin{pmatrix} 0 & 0\\-1 & 0 \end{pmatrix}.
\]
As representations of $F$ we have $\sigma_{A_1,B_1}$ and $\sigma_{A_2,B_2}$ are equivalent; however, their extension to $\cO$ are not equivalent. 
This will be further investigate in our next article \cite{Brothier-Wijesena23}.
}

\tb{Finally, we would like to emphasise that our approach permits to recover a great number of classification results concerning the three Thompson groups and the Cuntz algebra. Notably, we extend results from the following articles \cite{Aita-Bergmann-Conti97, Araujo-Pinto22, barata2019representations, Bergmann-Conti03, Bratteli-Jorgensen-19, DHJ-atomic15, Gar12, Guimaraes-Pinto22, Jones21, Kawamura-05, MSW-07, olesen2016thompson}. This will be extensively explained in our next paper \cite{Brothier-Wijesena23}.}

\section{Preliminaries}\label{sec:preliminaries}
We take this opportunity to fix notations and recall some standard definitions that will be required. We will follow the convention from the previous paper \cite{Brothier-Wijesena22}, hence we shall be brief and refer the reader to that paper for further details.

{\bf Convention.} 
Through out the paper we take the convention that all groups are discrete, all Hilbert spaces are over the complex field $\C$ and linear in the first variable, and all representations are {\it unitary}.

\subsection{Monomial representations and the Mackey-Shoda criterion}
We present a class of representations for which there exists a criteria deciding whether they are irreducible and if two of them are unitary equivalent or not.

{\bf Induced representations.}
If $H$ is a subgroup of a group $G$ and $\sigma$ is a representation of $H$, then $\Ind_H^G \sigma$ denotes the \textit{induced representation} of $\sigma$ associated to $H$. If $\sigma$ is the trivial representation of $H$ then the induced representation is the \textit{quasi-regular} representation associated to $H$ which we denote by $\lambda_{G/H}$.

{\bf Monomial representations.}
A \textit{monomial representation} of $G$ is an induced representation $\Ind_H^G\chi$ associated to a subgroup $H\subset G$ and a one-dimensional representation $\chi:H\to \bS$ (where $\bS$ stands for the unit circle of $\C$ identified with $\cU(\C)$).

{\bf Normaliser and commensurator.}
Let $H\subset G$ be a subgroup.
\begin{itemize}
\item The \textit{normaliser} of $H\subset G$ is the subgroup $N_G(H)\subset G$ of $g\in G$ satisfying $gHg^{-1}=H$.
\item The \textit{commensurator} of $H\subset G$ is the subgroup $\Comm_G(H)\subset G$ of $g\in G$ satisfying that $H\cap g^{-1}Hg$ has finite index in both $H$ and $g^{-1}Hg$.
\end{itemize}
Note that $$H\subset N_G(H)\subset \Comm_G(H)\subset G.$$
The subgroup $H\subset G$ is \textit{normal} (resp.~\textit{self-commensurating}) if $G=N_G(H)$ (resp.~$H=\Comm_G(H)$).

We are now ready to enunciate a useful result known as the Mackey-Shoda criterion which first appeared in \cite{Mackey51}, see \cite{BH} (Theorem 1.F.11, Theorem 1.F.16 and Corollary 1.F.18) for a more recent proof.

\begin{theorem} \label{Mackey-Shoda criteria}
Let $H\subset G$ be a subgroup and $\chi:H\to\bS$ a one-dimensional representation. 
Moreover, consider two subgroups $H_i\subset G, i=1,2$ and two one-dimensional representations $\chi_i:H_i\to\bS, i=1,2.$
The following assertions are true.
\begin{enumerate}
\item The induced representation $\Ind_H^G(\chi)$ is irreducible if and only if for every $g\in\Comm_G(H)\setminus H$, the restrictions of $\chi:H\to \bS$ and $\chi^g:g^{-1}Hg\ni s\mapsto \chi(gsg^{-1})$ to the subgroup $H\cap g^{-1} H g$ do not coincide.
\item The induced representation $\Ind_{H_1}^G\chi_1$ is unitary equivalent to $\Ind_{H_2}^G\chi_2$ if and only if there exists $g\in G$ such that $H_1 \cap g^{-1}H_2g$ has finite index in both groups $H_1$ and $g^{-1}H_2g$; and moreover the restrictions of $\chi_2^g$ and $\chi_1$ to $H_1 \cap g^{-1}H_2g$ coincide.
\end{enumerate}
\end{theorem}
In particular, if $H\subset G$ is a self-commensurating subgroup, then $\Ind_H^G\chi$ is irreducible for all one-dimensional representation $\chi:H\to\bS.$
Moreover, if $H_1,H_2$ are self-commensurating subgroups, then the induced representations are unitary equivalent if and only if there exists $g\in G$ satisfying $g^{-1} H_2 g = H_1$ (i.e.~the groups are conjugated) and $\chi_2^g=\chi_1.$
Subsequently, non-conjugated pairs of self-commensurated subgroups of $G$ together with \textit{any} one-dimensional representations produce non-equivalent irreducible induced representations.

\subsection{Direct integrals of representations}

Direct integrals provide a general framework for decomposing unitary representations into irreducible components. We shall provide a brief introduction to the theory of direct integrals required for this paper. For a more detailed introduction we refer the reader to Chapter 14 in \cite{kadison1986fundamentals}.\\
Let $(X, \mathcal{B}, \mu)$ be a measure space where $\mu$ is a $\sigma$-finite measure on $(X, \mathcal{B})$ and let $\{\scrH_x\}_{x \in X}$ be a family of separable Hilbert spaces. 
Then write
\[\scrH := \int_X^\oplus \scrH_x d\mu(x)\]
to be the space of functions $\xi$ on $X$ mapping $x \mapsto \xi(x) \in \scrH_x$ where $\scrH$ satisfies:
\begin{enumerate}
	\item $x \mapsto \langle \xi(x), \eta(x) \rangle$ is $\mu$-integrable for all $\xi,\eta \in \scrH$;
	\item if $\xi_x \in \scrH_x$ for all $x \in X$ and $\xi \mapsto \langle \xi_x, \eta(x) \rangle$ is $\mu$-integrable for every $\eta \in\scrH$ then there is a $\xi \in \scrH$ such that $\xi(x) = \xi_x$ for almost all $x \in X$. 
\end{enumerate} 
By property $(1)$, an inner-product on $\scrH$ can be defined:
\[\langle \xi, \eta \rangle = \int_X \langle \xi(x), \eta(x) \rangle d\mu(x) \textrm{ for all $\xi,\eta \in \scrH$}.\]
The above inner-product makes $\scrH$ a separable Hilbert space which is called the \textit{direct integral} of $\{\scrH_x\}_{x\in X}$ over $(X, \mu)$. 
Let $G$ be a discrete group. For each $x \in X$, let $\sigma_x$ be a representation of $G$ on $\scrH_x$ and suppose for each $g \in G$ the function $x \mapsto \langle \sigma_x(g) \xi(x), \eta(x) \rangle$ is $\mu$-integrable for all $ \xi, \eta \in \scrH$. Then $\{\sigma_x\}_{x\in X}$ forms a \textit{measurable field of representations} of $G$ and
\[\sigma(g) := \int_X^\oplus \sigma_x(g) d\mu(x)\]
defines a unitary operator on $\scrH$ such that $(\sigma(g)\xi)(x) = \sigma_x(g)(\xi(x))$ for almost every $x \in X$. Further, the mapping $g \mapsto \sigma(g)$ defines a representation $\sigma$ of $G$ on $\scrH$ which we write by
\[\sigma := \int_X^\oplus \sigma_x d\mu(x)\]
and is called the \textit{direct integral} of $\{\sigma_x\}_{x \in X}$ over $(X, \mu)$.

\begin{example}
	Suppose $X$ is a countable set endowed with the counting measure $\mu$. Then the direct integral of $\{\scrH_x\}_{x\in X}$ over $(X, \mu)$ simply gives the usual direct sum $\oplus_{x \in X} \scrH_x$. Similarly, direct integral of representations correspond to direct sum when $X$ is discrete.	
	Going back to general $(X,\mu)$, if $\scrH_x := \scrK$ is fixed for all $x \in X$, then the direct integral $\scrH$ corresponds to $L^2(X, \mu, \scrK)$: the space of (measurable) square integrable functions valued in $\scrK$ that can also be interpreted as the Hilbert space tensor product $L^2(X,\mu) \otimes \scrK.$
\end{example}

Direct integrals and inductions are compatible notions as illustrated by the following proposition.

\begin{proposition}[\cite{Mackey52} Theorem $10.1$] \label{direct integral induced prop}
	Let $H$ be a subgroup of a group $G$ and $\{\sigma_x\}_{x \in X}$ be a measurable field of representations of $H$. Then
	\[\Ind_H^G \left(\int_X^\oplus \sigma_x d\mu(x)\right) \cong \int_X^\oplus (\Ind_H^G \sigma_x) d \mu(x).\]
\end{proposition}

\subsection{Projections and strong operator topology} \label{lattice projection subsection}
\textbf{Projections.}
A \textit{(orthogonal) projection} $P$ acting on a Hilbert space $\fH$ is a bounded linear operator satisfying $P=P^*=P\circ P$. Taking the range $P\mapsto \Ran(P)$ realises a bijection between projections and closed vector subspaces of $\fH$. We order projections as usual: $P_1\leq P_2$ when $\Ran(P_1)\subset \Ran(P_2)$.
Families of projections $\{P_i:\ i\in I\}$ admit greatest lower bound $\bigwedge_i P_i$ and least upper bound $\bigvee_i P_i$ just like closed vector subspaces do.

\textbf{Strong operator topology.}
The \textit{strong operator topology} (SOT) of $B(\fH)$ is the locally convex topology obtained from the seminorms $B(\fH)\ni T\mapsto \|T(\xi)\|$ indexed by the vectors $\xi\in\fH$.
In particular, a net of operators $(T_i)_{i\in I}$ converges to $T$ for the SOT, denoted $T_i\xrightarrow{s} T$, when $\lim_{i\in I}\|T(\xi)-T_i(\xi)\|=0$ for all $\xi\in\fH$.
Below is an important result regarding the convergence of increasing (and decreasing) nets of projections in the SOT, see \cite[Propositions 2.5.6, 2.5.8 and Corollary 2.5.7]{kadison1983fundamentals} for proofs.

\begin{proposition} \label{MCT projection prop}
	If $I$ is a directed set and $(P_i : i \in I)$ an increasing (resp.~decreasing) net of projections on $\fH$, then $P_i \xrightarrow{s} \bigvee_{i \in I} P_i$ (resp. $P_i \xrightarrow{s} \bigwedge_{i\in I} P_i$). \\
	In particular, if $\{Q_i\}_{i \in I}$ is a family of mutually disjoint projections (i.e.~$Q_iQ_j=0$ when $i\neq j$), then $\sum_{j\in J}Q_j \xrightarrow{s} \bigvee_{i\in I} Q_i$ where $J\subset I$ is a finite subset that we let tend to $I$.
\end{proposition}

\subsection{Richard Thompson's group $F$} \label{subsec:F-def}

Richard Thompson's group $F$ is the group of piece-wise linear increasing homeomorphisms of $[0,1]$ having finitely many non-differentiable points each contained in the dyadic rationals $\Z[1/2]$.
It also acts by homeomorphisms on the Cantor space $\mathcal C:=\{0,1\}^{\N^*}$ by locally changing finite prefixes and keeping the lexicographic order. 
Note that $$S:\mathcal C\to [0,1], x\mapsto \sum_{n\geq 1} \frac{x_n}{2^n}$$ is a surjective continuous $F$-equivariant map so that each dyadic rational of $(0,1)$ have two preimages and the other one preimage.
We consider the usual notion of derivative or slope for an element $g\in F$ at a point $x$ of $\mathcal C$ using cylinder sets or standard dyadic intervals (by a standard dyadic interval, or sdi in short, we mean the set $I_w := \{w \cdot x : x \in \cC\}$ where $w$ is a finite word). 
That is, if there exists two finite words $u,v$ in $0,1$ made of $n$ and $m$ letters respectively so that for any infinite word $x\in\mathcal C$ we have $g(u\cdot x)=v\cdot x$, then 
$$g'(u\cdot x) = 2^{n-m}.$$
This coincides with the derivative of $g$ acting on $[0,1]$ on a point where $g$ is differentiable.
As usual the support of $g\in F$, denoted $\supp(g)$, denotes the closure of the points of $\mathcal C$ that are not fixed by $g$:
$$\supp(g):=\overline{ \{x\in\mathcal C:\ g(x)\neq x \} }.$$

We will be mainly using the description of $F$ using trees that we now precisely define. 
This description occured in the work of Brown \cite{Brown87}.

{\bf Infinite binary rooted tree.}
Consider the infinite binary rooted tree $t_\infty$ that we geometrically identify as a graph in the plane where the root is on top, the root has two adjacent vertices to its bottom left and right, and every other vertex has three adjacent vertices: one above, one to the bottom left, and one to the bottom right. The bottom two vertices are called \textit{immediate children} of the vertex.
Moreover, a pair of edges that have a common vertex is called a \textit{caret} which we denote by the symbol $\Y$.

{\bf Decoration of vertices with binary sequences.}
We will identify vertices of $t_\infty$ with finite binary sequences (also referred to as finite words). 
Indeed, write $\Ver$ for the set of vertices of $t_\infty$
Then associate the root with the empty sequence which we denote by $\varnothing$.
Moreover, if $\nu\in \Ver$ which is associated to the finite sequence $s$ and $\nu_0,\nu_1$ are the left and right immediate children of $\nu$, then associate $\nu_i$ with $s \cdot i$ for $i = 0,1$.

{\bf Decoration of vertices with sdi.}
By the above identification we can decorate each vertex with the sdi $I_\nu$. This gives a bijection between $\Ver$ and the set of sdi's.
We say two vertices are \textit{disjoint} if their associated sdi's (as subsets of $\cC$) are disjoint.

{\bf Tree.}
A \textit{tree} is any non-empty rooted finite subtree of $t_\infty$ so that each vertex has either two immediate children or none. 
A vertex with no immediate children is called a \textit{leaf} and we denote the set of leaves of a tree $t$ by $\Leaf(t)$.
The leaves of a tree are indexed increasingly from left to right starting at 1. 
Moreover, we write $\fT$ for the collection of all trees and set $\tar(t) := \vert \Leaf(t) \vert$ where $\tar$ stands for ``target''. 
By associating each leaf with its sdi we obtain a bijection between $\fT$ and the set of all standard dyadic partitions (a standard dyadic partition, in short sdp, is a finite partition of $\cC$ consisting of sdi's).

Denote the trivial tree, the tree having only one leaf, by $e$ or $I$, and write $t_n$ for the regular tree with $2^n$ leaves all at distance $n$ from the root for the usual tree-metric.

{\bf Forest.}
We call a \textit{forest} a finite union of trees that we represent as finitely many trees placed next to each other ordered from left to right.
Write $\fF$ for the collection of all forests and $\Root(f),\Leaf(f)$ for the sets of roots and leaves of a forest $f$, respectively.

The \textit{tensor product} $f\ot g$ of two forests $f,g$ is defined to be the forest obtained by concatenating horizontally $f$ to the left of $g$ which forms an associative binary operation on $\fF$.
An \textit{elementary forest} $f_{k,n}$ where $1\leq k\leq n$ is the unique forest that has $n$ roots, $n+1$ leaves, and all the trees in the forest are trivial except the $k$th tree which is a caret.
Thus $f_{k,n}= I^{\ot k-1} \ot \Y \ot I^{\ot n-k}.$
We may often write $f_{k}$ rather than $f_{k,n}$ if the value of $n$ is clear from context.
Observe that the set of elementary forests generate $\fF$ as every forest can be expressed as a finite composition of elementary forests. 

{\bf Composition of forests.}
When $f,g$ are forests so that the number of leaves of $f$ is equal to the number of roots of $g$ we can geometrically define the \textit{composition} of $f$ with $g$, written $g\circ f$, as the forest obtained by vertically stacking  $f$ on top of $g$ and by lining up the $j$th leaf of $f$ with the $j$th root of $g$.
This provides a partially defined associative binary operation on $\fF$.

We can use this operation to equip $\fT$ with a directed poset structure $\preceq$ by declaring that $t\preceq s$ if and only if there exists $f \in \fF$ such that $s = f\circ t$. 

\textbf{Tree-diagrams.}
A \textit{tree-diagram} is an ordered pair of trees $(t,s) \in \fT \times \fT$ with the same number of leaves.
From the bijection between trees and set of sdp's we deduce that any element $g\in F$ can be described as a tree-diagram.
However, this is not a one to one correspondence because it is clear $(f\circ t,f\circ s)$ and $(t,s)$ correspond to the same Thompson's group element.
Two leaves $\nu \in \Leaf(t), \omega \in \Leaf(s)$ are said to \textit{correspond} to each other if they have the same numbered position. Subsequently, the element in $F$ associated with $(t, s)$ maps the sdi $I_\omega$ to $I_\nu$.

{\bf Description of $F$ as tree-diagrams \cite{Brown87}:}
Denote $\cQ$ to be the set of all tree-diagrams $(t,s)$. 
Define $\sim$ to be the equivalence relation generated by $(f\circ t,f\circ s)\sim (t,s)$ where $f$ is any forest such that composition with $t,s$ is well-defined.
We then consider the quotient space $\cQ/\sim$ and let $[t,s]$ be the class of $(t,s)$.
Define the binary operation: 
$$[t,s]\circ [s, r]:=[t, r].$$
This binary operation extends to all of $\cQ$ since $(\fT, \preceq)$ is a directed poset. Further, it is easy to verify that the operation $\circ$ is well-defined and yields a group $(\cQ/\sim, \circ)$ such that $[t,s]^{-1}=[s,t]$ and $[t,t]$ is the identity for each tree $t$.
Moreover, this group is precisely Thompson's group $F$.

{\bf We will mostly working with this description of $F$ as tree-diagrams and thus will view elements of Thompson's group as an ordered pair of trees. }

{\bf Left/right sides, and centre of $t_\infty$.}
We partition $\Ver$ into four sets on which $F$ acts transitively on. 
Indeed, denote the \textit{left (resp. right) side} of $t_\infty$ to be the set of vertices whose binary sequence consists only of zeroes (resp. ones). 
The root node neither lies on the left or right side of $t_\infty$. 
The \textit{centre} of $t_\infty$ is the set of vertices whose binary sequence consists of a zero and an one. We define the following convenient notation that will be referred to several times in the sequel.

\begin{notation} \label{match vertices}
	Define $D$ to be the set of pairs of vertices $(\nu, \omega) \in \Ver \times \Ver$ such that both $\nu$ and $\omega$ are either on the left side, right side or in the centre of $t_\infty$.
	
	For all $(\nu, \omega) \in D$, there exist $[t, s] \in F$ such that $\nu$ and $\omega$ are corresponding leaves in the trees $t$ and $s$, respectively.
\end{notation}

\subsection{Identification of Rays in $t_\infty$}
We call a \textit{ray} a path in $t_\infty$ of infinite length which begins from the root node and has no repeating edges, i.e.~an infinite geodesic path starting at the root node. Define $\Ver_p=\{\mu_n^p:\ n\geq 0\} \subset \Ver$ to be the set of vertices a ray $p$ passes through so that $\mu_n^p$ is at distance $n$ from the root node.

\subsubsection{Rays as sequences of binary digits}
We associate each ray $p$ with a sequence of binary digits $(x_k : k \geq 1)$ given by
\begin{equation*}
	x_k =
	\begin{cases*}
		0, \quad& if the $k$th edge of $p$ is a left-edge\\
		1, \quad& if the $k$th edge of $p$ is a right-edge.
	\end{cases*}
\end{equation*}
The first $n$ digits of $p$ forms the vertex $\mu_n^p$ for $n \geq 0$. 
This identification gives an obvious bijection between the space of rays in $t_\infty$ and the Cantor space. For the remainder of the paper we shall freely identify elements in the Cantor space with rays in $t_\infty$.

\subsubsection{Rays as sequences of $a,b$}
We provide in conjunction another convenient identification of rays.
Write $\Mon(a,b)$ for the free monoid in two generators $a$ and $b$: all the finite words in these two letters with empty word $e$ equal to the unit. 

\begin{definition} \label{subwords of ray definition}
	A ray $p$ is an infinite binary sequence in $a,b$ that we read from \textit{right to left} so that the $j$th letter from the right is $a$ if the $j$th edge of $p$ is a left-edge and $b$ otherwise for each $j\geq 1$. \\
	For $n \in \N$ define the following.
	\begin{enumerate}[i]
		\item Let $p_n$ be the finite subword in $\Mon(a,b)$ of $p$ made of the first $n$ letters from the right of $p$ where take $p_0$ to be the empty word $e$. Hence, $p_{n+1}=x\cdot p_n$ with $x$ the $(n+1)$th letter of $p$ and $p=q\cdot p_n$ for some ray $q$. Identify $p_n$ with the subray of $p$ which is the path from $\varnothing$ to the vertex in $p$ of length $n$.
		\item Let $\ti p_{n+1}$ be equal to $x \cdot p_n$ where $x = b$ if the $(n+1)$th letter of $p$ is $a$ otherwise $x = a$. 
		\item Let ${}_kp$ be the ray formed by removing the first $k$ letters from the right of $p$ where take $_0p = p$. Thus ${}_kp$ is isomorphic to the subray contained in $p$ beginning on the vertex in $p$ of length $n$ and satisfies the identity $p = {}_kp\cdot p_k$.
	\end{enumerate}
\end{definition}

\textbf{Period and equivalence relation on rays.}
Here are some useful definitions concerning rays.

\begin{definition} \label{straight ray defintiion}
	\begin{enumerate}[i]
		\item If $p$ is an eventually periodic ray then there exist words $c, w$ such that $c$ is prime and $p = c^\infty \cdot w$ (a word is prime if it cannot be expressed in the form $d^n$ for any word $d$ and $n > 1$). We say $c$ is a period of $p$.
		\item We say that two rays are \textit{equivalent}, denoted $p\sim q$, when up to removing finite prefixes they are equal and write $[p]$ for the equivalence class of a ray $p$ with respect to $\sim$.
		The set of all equivalence classes of rays is denoted $\cP$.
		\item A ray is called a \textit{straight line} if it is one of the following two sequences:
		\[\ell := \dots aaa \textrm{ or } r := \dots bbb\]
		and can be geometrically realised as a straight line with only left or right turns.
		Alternatively, $\ell$ (resp. $r$) is the ray that goes down the left (resp. right) side of $t_\infty$. 
		\item A ray is called \textit{eventually straight} if it belongs in either the equivalence class $[\ell]$ or $[r]$. 
		Equivalently, a ray is eventually straight if the ray is constant after some finite number of turns.
		We shall also describe the above two equivalence classes as being eventually straight.
	\end{enumerate}	
\end{definition}

\begin{remark}
	If $c$ is a period of a ray $p$, then any cyclic permutation of the letters of $c$ is also a period of $p$. For instance, the ray $\dots ababab$ admits a period $ab$ but a period $ba$.
	Moreover, note that the equivalence class of $c^\infty$ where $c$ is a prime word is the set of rays that are eventually periodic with $c$ for a period.
\end{remark}

{\bf Identifications.}
To summarise the different identifications of rays we appeal to in this paper, rays in $t_\infty$ can be identified as infinite sequences of $0,1$ that are read from \textit{left to right} and which correspond to elements in the Cantor space. Rays can also be identified as infinite sequences of $a,b$ which are read from \textit{right to left}. We can move between these two identifications by interchanging $0$'s with $a$'s, $1$'s with $b$'s, and by reversing the order of the sequence. 
Additionally, this procedure identifies vertices $\nu$ in $t_\infty$ with finite words in $\Mon(a,b)$. In particular, for a ray $p$ the vertex $\mu_n^p$ is identified with the word $p_n$. 
We shall use these two identifications of rays interchangeably without specification unless it is not clear from context.

\subsection{Certain subgroups of $F$}
\subsubsection{Stabiliser and fixed point subgroups}
Using the action of $F$ on the Cantor space we define some particular subgroups of $F$ that arise in our analysis. 
Then using the presentation of $F$ as tree-diagrams, we shall provide a more useful description of these subgroups as sets of tree-diagrams.

{\bf Stabiliser and fixed point subgroups associated to $F\act \cC$.}
Identify $F$ with its description as a subgroup of the homeomorphisms of the Cantor space $\cC$.
For any subset $A\subset \cC$, write
\[\Stab_F(A)=\{ g\in F:\ g(A)=A\} \text{ and } \Fix_F(A)=\{g\in F:\ g(a)=a, \forall a\in A\}\]
as the \textit{stabiliser subgroup} and \textit{fixed-point subgroup} of $A$, respectively. 
We call $$F_p := \{g\in F:\ g(p)=p\} = \Fix_F(\{p\})$$ a \textit{parabolic subgroup} where $p$ is either an element of the Cantor space or an element of $[0,1].$
Except when $p=0$ or $1$ this subgroup is proper.

We introduce another important subgroup that will be required in the main results.

\begin{definition} \label{subgroup of Fp definition}
	For any ray $p$ we define the subgroup $\widehat{F_p}$ of $F$ by
	\[\widehat{F_p} := \{g \in F : g(p) = p, g'(p) = 1\} \subset F.\]
\end{definition}

\begin{remark} \label{subgroup of Fp remark}	
	\begin{enumerate}[i]
		\item Note that $\widehat{F_p}$ is a normal subgroup of $F_p$ and $F_p/\widehat{F_p}\cong \Z$ when $p$ is rational. Moreover, $\widehat{F_p}$ can be defined as the group of $g\in F$ acting like the identity on a neighbourhood of $p$.
		\item It can be seen that when $p$ is not eventually periodic we have $F_p=\widehat{F_p}$. This will be shown below by Equation \ref{irrational stab group equation}.
		\item If $S:\cC\to[0,1]$ is the classical $F$-equivariant map we have $F_{S(p)}=F_p$. Although, we must take better care when dealing with the groups $\widehat F_p$. Indeed, observe that $p=\dots 0111$ and $q=\dots 1000$ are sent by $S$ to the same number $1/2$ but $\widehat{F_p}\neq\widehat{F_q}.$
	\end{enumerate}
\end{remark}

\subsubsection{Description of parabolic subgroups using tree-diagrams} \label{parabolic desc subsection}
Let $p$ be a ray given by a sequence of binary digits $(x_n : n \geq 1)$ and recall ${}_kp$ is the ray formed by removing the first $k$ digits of $p$.
Given $g=[t,s]\in F$ we have that there exists a unique leaf $\nu$ of $t$ and $\omega$ of $s$ so that $\nu,\omega$ lie in the ray $p$ (equivalently $p$ is in the sdi $I_\nu$ and $I_\omega$ associated to the vertices $\nu$ and $\omega$).
Thus, using the notation defined from the previous subsection, $\nu = \mu^p_m$ and $\omega = \mu^p_n$ where $m = \length(\nu)$ and $n = \length(\omega)$. It can be then verified that $g \in F_p$ if and only if $\mu^p_m$ and $\mu^p_n$ are corresponding leaves of $g$ and ${}_mp = {}_np$. 

Hence, suppose $p=w\cdot c^\infty$ is eventually periodic and thus can be expressed as a minimal finite prefix $w$ (minimal in the length of $w$) followed by a periodic sequence $c^\infty$ where $c$ is a period of $p$. Let $a = \length(w)$ and $b = \length(c)$. Then from the preceding paragraph we have
\begin{equation}\label{parabolic subgroup condition eqn}
	\begin{cases*}
		m = n, &\quad \text{ if } $m,n \leq a$ \\
		n-m \in b\Z, &\quad \text{ if } $m,n > a$.
	\end{cases*}
\end{equation} 

In the case when $p$ is {\it not eventually periodic} we can obtain an even stronger description of $F_p$. Indeed, since now $w$ will be an infinite sequence, Equation \ref{parabolic subgroup condition eqn} shows $g \in F_p$ if and only if $m = n$. Hence $\nu = \omega$ which implies that $g'(p) = 1$ and thus
\begin{equation} \label{irrational stab group equation}
	F_p = \widehat{F_p} = \{g \in F: g(p) = p, g'(p) = 1\} \text{ for all not eventually periodic $p$.}
\end{equation}

More generally, we have an easy description of the subgroups $\widehat{F_p}$ in terms of tree-diagrams for any ray $p$. An element $g = [t,s] \in F$ belongs in $\widehat{F_p}$ if and only if there exists two corresponding vertices of $t,s$ that coincide and the ray $p$ is contained in the sdi associated to that vertex.

\subsubsection{Monomial representations associated to parabolic subgroups} \label{general monomial rep section}
A large part of the analysis in the sequel will involve monomial representations associated to the parabolic subgroups of $F$. From the Mackey-Shoda criterion we are able to gleam information about these monomial representations by instead studying the commensurator of the parabolic subgroups. 
It is well-known and easy to prove that proper parabolic subgroups of $F$ are self-commensurating. This implies the following result.

\begin{lemma} \label{monomial parabolic irrep lemma}
Let $p$ be a ray which is not a straight line and $\chi$ be a one-dimensional representation of $F_p$. Then the monomial representation $\Ind_{F_p}^F \chi$ associated to $F_p$ is irreducible.
\end{lemma}

However, the above result does not extend to the subgroups $\widehat{F_p}$.
When $p$ is not a straight line then $\Comm_F(\widehat{F_p}) = F_p$ and when $p$ is a straight line then $\Comm_F(\widehat{F_p}) = F$. Hence, by Equation \ref{irrational stab group equation}, $\widehat{F_p}$ is not self-commensurating if and only if $p$ is eventually periodic. Thus, Theorem \ref{Mackey-Shoda criteria} shows that $\lambda_{F/\widehat{F_p}}$ is \textit{reducible} when $p$ is eventually periodic unlike the monomial representations associated to $F_p$. 

\subsection{Definition of Pythagorean representations} \label{sec:def-pyth}
We will now explain the specific class of Jones' representations that will be the focus of this paper. As before, we shall be brief and for further details we recommend the reader our previous article \cite{Brothier-Wijesena22} and the article where these representations were initially constructed \cite{Brothier-Jones19-Pyth}.

\subsubsection{Pythagorean pairs of operators}

Consider a pair of bounded linear operators $A, B \in B(\fH)$ acting on a Hilbert space $\fH$. We call $(A,B)$ a Pythagorean pair of operators if it satisfies the so called Pythagorean identity:
$$A^*A + B^* B = \id_\fH$$
where $\id_\fH$ is the identify operator of $\fH$ and $A^*$ is the adjoint of $A$. Based on the above identity we define the following universal $C^*$-algebra.

\begin{definition} \label{universal pythag algebra definition}
The Pythagorean algebra $P=P_2$ is the universal $C^*$-algebra generated by $a,b$ according to the relation
\[a^*a + b^*b = 1.\]
\end{definition}

\subsubsection{Construction of a Hilbert space}
Fix a Pythagorean pair $(A,B)$ over $\fH$. 
For each tree $t\in\fT$ we define $\fH_t$ to be the $n$th direct sum $\fH^n:=\fH^{\oplus n}$ where $n$ is the number of leaves of $t$.
For a vector $\xi = (\xi_1,\cdots,\xi_n) \in \fH_t$, to emphasise the tree $t$ and to avoid confusions we may instead write $(t, \xi)$ or $(t, \xi_1,\cdots,\xi_n)$.
A convenient and diagrammatic way to view a vector $\xi\in\fH_t$ is to consider the tree $t$ such that the leaf $\ell$ is decorated by $\xi_\ell \in \fH$. 
Hence, diagrammatically $\fH_t$ is the Hilbert space formed by decorating the leaves of $t$ with elements in $\fH$.

Now we aim to place an equivalence relation on the family $(\fH_t:\ t\in\fT)$. First define the map: 
\begin{align*}
	\Phi(f_{k,n})=\Phi_{A,B}(f_{k,n}):& \fH^n\to \fH^{n+1},\\
	& (\xi_1,\cdots,\xi_n)\mapsto (\xi_1,\cdots,\xi_{k-1}, A(\xi_k), B(\xi_k), \xi_{k+1},\cdots,\xi_n)
\end{align*}
where recall $f_{k,n}$ is an elementary forest. The above map is an isometry because $(A,B)$ is a Pythagorean pair. Since any forest $f$ is a finite composition of elementary forests (which may not necessarily be unique), we can define $\Phi(f)$ via compositions which will also be a well-defined isometry. 

Then for all trees $t \in \fT$ set $(t,\xi)\sim (f \circ t,\Phi(f)\xi)$ where $(t, \xi) \in \fH_t$ and $f$ is any forest where the composition with $t$ is compatible. We then consider the smallest equivalence relation $\sim$ generated by it. The quotient of the disjoint union $\sqcup_{t \in \fT} \fH_t$ by $\sim$ is a pre-Hilbert space $\scrK = \scrK_{A,B}$ which we complete into a Hilbert space $\scrH = \scrH_{A,B}$.

Hence, elements in the dense subspace $\scrK$ have representatives in the form $(t,\xi)$ which we shall write $[t,\xi]$ to denote the class associated to it.
Importantly, for all trees $t$ the space $\scrK$ contains a copy of $\fH_t$ via the natural inclusion $\fH_t \ni \xi \mapsto [t, \xi] \in \scrK$. In particular, we will commonly view $\fH$ embedded inside $\scrK$ and identify $\xi \in \fH$ with its copy $[e, \xi]$ inside $\scrK$. 

\begin{remark} \label{scrH dimension remark}
	As explained in \cite{Brothier-Wijesena22}, the Hilbert space $\scrH$ will always be infinite-dimensional regardless of the dimension of the initial Hilbert space. 
\end{remark}

\subsubsection{The Jones representation associated to a Pythagorean pair}
We can now construct a unitary representation $\sigma = \sigma_{A,B}$ of $F$ acting on the Hilbert space $\scrH$. Indeed, for an element $g = [t, s] \in F$ and vector $[s, \xi] \in \scrH$ we set
\[\sigma([t, s])[s, \xi] := [t, \xi].\]
The map $\sigma([t, s])$ extends to all of $\scrH$ since $(\fT, \preceq)$ is a directed poset and by using the equivalence relation $\sim$ on $(\fH_t : t \in \fT)$. It can be shown that this forms a unitary operator and is well-defined on the equivalence class $[t, s]$.

An example of the action $\sigma$ is shown below. In this case, $\sigma$ only changes the tree while retaining the original decoration.

\begin{center}
	\begin{tikzpicture}[baseline=0cm]
		\draw (0,0)--(-.5, -.5);
		\draw (0,0)--(.5, -.5);
		\draw (-.5, -.5)--(-.9, -1);
		\draw (-.5, -.5)--(-.1, -1);
		
		\node[label={\normalsize $\sigma($}] at (-1.1, -1) {};
		\node[label={\normalsize $,$}] at (.65, -1) {};
	\end{tikzpicture}%
	\begin{tikzpicture}[baseline=0cm]
		\draw (0,0)--(-.5, -.5);
		\draw (0,0)--(.5, -.5);
		\draw (.5, -.5)--(.1, -1);
		\draw (.5, -.5)--(.9, -1);
		
		\node[label={\normalsize $)$}] at (1.1, -1) {};
		\node[label={\normalsize $\cdot$}] at (1.35, -.9) {};
	\end{tikzpicture}%
	\begin{tikzpicture}[baseline=0cm]
		\draw (0,0)--(-.5, -.5);
		\draw (0,0)--(.5, -.5);
		\draw (.5, -.5)--(.1, -1);
		\draw (.5, -.5)--(.9, -1);
		
		\node[label={[yshift=-22pt] \normalsize $\xi_1$}] at (-.5, -.5) {};
		\node[label={[yshift=-22pt] \normalsize $\xi_2$}] at (.1, -1) {};
		\node[label={[yshift=-22pt] \normalsize $\xi_3$}] at (.9, -1) {};	
		
		\node[label={\normalsize $=$}] at (1.35, -1) {};
	\end{tikzpicture}%
	\begin{tikzpicture}[baseline=0cm]
		\draw (0,0)--(-.5, -.5);
		\draw (0,0)--(.5, -.5);
		\draw (-.5, -.5)--(-.9, -1);
		\draw (-.5, -.5)--(-.1, -1);
		
		\node[label={[yshift=-22pt] \normalsize $\xi_1$}] at (-.9, -1) {};
		\node[label={[yshift=-22pt] \normalsize $\xi_2$}] at (-.1, -1) {};
		\node[label={[yshift=-22pt] \normalsize $\xi_3$}] at (.5, -.5) {};		
	\end{tikzpicture}%
\end{center}

\begin{definition}We call $\sigma_{A,B}$ the \textit{Jones representation} or \textit{Pythagorean representation} associated to the Pythagorean pair $(A,B)$.
\end{definition}

\subsubsection{Partial Isometries on $\scrH$} \label{subsec:partial-isom}
A powerful tool in our analysis of Pythagorean representations is the family of partial isometries $(\tau_\nu : \nu \in \Ver) \subset B(\scrH)$ which were initially defined in \cite{Brothier-Wijesena22} (see Subsection 2.1).	 
Informally, the partial isometry $\tau_\nu$ can be defined in the following manner.

Fix $\nu \in \Ver$ and consider $[t, \xi] \in \scrK$. Up to taking representatives of $[t, \xi]$ we can assume that $\nu$ is a vertex of $t$. Define $t_\nu$ to be the sub-tree of $t$ with root $\nu$ and whose leaves are the leaves of $t$ which are children of $\nu$. Then we set $\tau_\nu([t, \xi]) := [t_\nu, \eta]$ where $\eta$ is the decoration of the leaves of $t$ in $[t, \xi]$ that are children of $\nu$. It can be shown that $\tau_\nu$ is well-defined and extends to a surjective partial isometry from $\scrH$ onto itself. 
It follows the adjoint $\tau_\nu^*$ is an isometry and is a right-inverse of $\tau_\nu$. 

\begin{example}
	Suppose we wish to compute $\tau_{01}([\wedge, (\xi_1, \xi_2)])$. The vertex $01$ is not contained in the caret $\wedge$, hence we first consider another representative of $[\wedge, (\xi_1, \xi_2)]$ by attaching a caret to the first leaf of $\wedge$ and decorating the two new leaves with $A\xi, B\xi$ from left to right. Since $01$ is now a leaf in the new tree, the operator $\tau_{01}$ simply takes the decoration of the leaf $01$.
	\begin{center}
		\begin{tikzpicture}[baseline=0cm, scale = 1]
			\draw (0,0)--(-.5, -.5);
			\draw (0,0)--(.5, -.5);
			
			\node[label={[yshift=-22pt] \normalsize $\xi_1$}] at (-.5, -.5) {};
			\node[label={[yshift=-22pt] \normalsize $\xi_2$}] at (.5, -.5) {};
			
			
			\node[label={[yshift= -3pt] \normalsize $\Phi_{A,B}(f_1)$}] at (2, -.7) {$\sim$};
		\end{tikzpicture}%
		\begin{tikzpicture}[baseline=0cm, scale = 1]
			\draw (0,0)--(-.5, -.5);
			\draw (0,0)--(.5, -.5);
			\draw (-.5, -.5)--(-.9, -1);
			\draw (-.5, -.5)--(-.1, -1);
			
			\node[label={[yshift=-22pt] \normalsize $A\xi_1$}] at (-.9, -1) {};
			\node[label={[yshift=-22pt] \normalsize $B\xi_1$}] at (-.1, -1) {};
			\node[label={[yshift=-22pt] \normalsize $\xi_2$}] at (.5, -.5) {};
			
			\node[label={[yshift= -3pt] \normalsize $\tau_{01}$}] at (1.4, -.7) {$\longmapsto$};
		\end{tikzpicture}%
		\begin{tikzpicture}[baseline=0cm, scale = 1]
			\node[label={[yshift=-25pt] \normalsize $B\xi_1$}] at (0, -.5) {$\bullet$};
		\end{tikzpicture}%
	\end{center}
\end{example}

Diagrammatically, $\tau_\nu$ acts on $(t, \xi) \in \fH_t \subset \scrK$ by first ``growing'' the tree to be large enough and then ``snipping'' the tree at $\nu$ while retaining the components of $\xi$ which are children of $\nu$. Furthermore, $\tau^*_\nu$ acts on $(t, \xi)$ by ``lifting'' the tree $t$, along with its components $\xi$, and attaching it to the vertex $\nu$ while setting all other components to zero. More generally, $\tau^*_\nu\tau_\omega$ is the partial isometry which ``snips'' the tree at $\omega$ and attaches the resulting subtree, along with its components, at the vertex $\nu$ while setting all other components to $0$. An example is show below.

\begin{center}
	\begin{tikzpicture}[baseline=0cm, scale = 1]
		\draw (0,0)--(-.7, -.5);
		\draw (0,0)--(.7, -.5);
		\draw (-.7, -.5)--(-1.1, -1);
		\draw (-.7, -.5)--(-.3, -1);
		\draw[thick] (.7, -.5)--(.3, -1);
		\draw[thick] (.7, -.5)--(1.1, -1);
		
		\node[label={[yshift=-22pt] \normalsize $\xi_1$}] at (-1.1, -1) {};
		\node[label={[yshift=-22pt] \normalsize $\xi_2$}] at (-.3, -1) {};
		\node[label={[yshift=-22pt] \normalsize $\xi_3$}] at (.3, -1) {};
		\node[label={[yshift=-22pt] \normalsize $\xi_4$}] at (1.1, -1) {};
		
		\node[label={[yshift= -3pt] \normalsize $\tau_{1}$}] at (1.6, -.7) {$\longmapsto$};
	\end{tikzpicture}%
	\begin{tikzpicture}[baseline=0cm, scale = 1]
		\draw[thick] (0,-.3)--(-.5, -.8);
		\draw[thick] (0,-.3)--(.5, -.8);
		
		\node[label={[yshift=-22pt] \normalsize $\xi_3$}] at (-.5, -.8) {};
		\node[label={[yshift=-22pt] \normalsize $\xi_4$}] at (.5, -.8) {};
		
		\node[label={[yshift= -3pt] \normalsize $\tau^*_{0}$}] at (1.1, -.7) {$\longmapsto$};
	\end{tikzpicture}%
	\begin{tikzpicture}[baseline=0cm, scale = 1]
		\draw (0,0)--(-.7, -.5);
		\draw (0,0)--(.7, -.5);
		\draw[thick] (-.7, -.5)--(-1.1, -1);
		\draw[thick] (-.7, -.5)--(-.3, -1);
		
		\node[label={[yshift=-22pt] \normalsize $\xi_3$}] at (-1.1, -1) {};
		\node[label={[yshift=-22pt] \normalsize $\xi_4$}] at (-.3, -1) {};
		\node[label={[yshift=-22pt] \normalsize $0$}] at (.7, -.5) {};		
	\end{tikzpicture}%
\end{center} 

In particular, we define the orthogonal projection:
$$\rho_\nu := \tau_\nu^*\tau_\nu$$ 
which sets the components to zero for all leaves which are children of vertices that are disjoint from $\nu$. Observe that if $\{\nu_i\}_{i=1}^n$ is a sdp then $\sum_{i=1}^n \rho_{\nu_i} = \id$.
	
The Pythagorean representation is closely related to the above maps.
Indeed, let $z \in \scrH$ and $g := [t, s] \in F$ where $\Leaf(t) := \{\nu_i\}_{i = 1}^n$ and $\Leaf(s) := \{\omega_i\}_{i = 1}^n$. Then the action $\sigma(g)$ can be interpreted as ``snipping'' $s$ at the vertex $\omega_i$ and then attaches the subtree to the vertex $\nu_i$, along with the components of $\xi$ that are children of $\omega_i$ (this corresponds to how $g$ maps the sdi $I_\omega$ to $I_\nu$). This is expressed in the following identity which will be frequently used in the sequel:
\begin{equation} \label{action rearrange equation}
	\sigma(g) = \sum_{i = 1}^n \tau^*_{\nu_i} \tau_{\omega_i}.
\end{equation}

\textbf{Projections associated to a ray.}
We can also associate each ray to a projection on $\scrH$. This construction will require notation and results introduced in Subsection \ref{lattice projection subsection}.

\begin{construction}
	Let $p$ be a ray and recall $(\mu_n^p : n \geq 0)$ is the sequence of vertices that $p$ passes through. This in turn induces a sequence of projections $(\rho_{\mu_n^p} : n \geq 0)$. 
	By definition, this sequence is monotonically decreasing and thus, by Proposition \ref{MCT projection prop}, converges in the SOT to the projection 
	$$\rho_p:=\bigwedge_{n \in \N} \rho_{\mu_{n}^p}.$$
\end{construction}

\begin{notation} \label{tau notation}
	For ease of notation, hereby we shall write $\rho_{p,n}$ for $\rho_{\mu_n^p}$ which gives $\rho_{p,n} \xrightarrow{s}\rho_p$. Similarly, we shall write $\tau_{p,n}, \tau^*_{p,n}$ for $\tau_{\mu^p_n}, \tau^*_{\mu^p_n}$, respectively.
\end{notation}

\begin{remark}The range of $\rho_p$ is in some sense the space of vectors which are entirely supported by the ray $p$ (this will be described in more detail in Subsection \ref{class vectors subsection}). 
	These vectors will play an important role in the decomposition of Pythagorean representations. However, note $\rho_p$ will be trivial in many cases. For example, if $\norm{A}, \norm{B} < 1$ then it can be shown that $\rho_p$ will simply be the zero operator for all rays $p$.
\end{remark}


\section{Canonical Decomposition of Pythagorean Representations} \label{decompose pythag rep section}
In this section we perform a powerful decomposition of an arbitrary Pythagorean representation $(\sigma,\scrH)$ by considering the actions of $A$ and $B$ on the smaller Hilbert space $\fH$.

\subsection{Monoid generated by $(A, B)$} \label{monoid section}

Fix a Pythagorean pair $(A, B)$ acting on $\fH$ and let $W := W_{A,B} \subset B(\fH)$ be the monoid generated by $A$ and $B$: the smallest subset of $B(\fH)$ containing $A,B,\id_\fH$ and closed under composition. 
Recall $\Mon(a,b)$ is the free monoid in two generators $a$ and $b$ which includes the empty word $e$. We equip $\Mon(a,b)$ with the usual word-length denoted $|\cdot|$.
We have a canonical monoid morphism 
$$\phi_{A,B}:\Mon(a,b)\to W_{A,B},\ a\mapsto A, b\mapsto B$$
which provides an action of $\Mon(a,b)$ on $\fH$.

If $w\in \Mon(a,b)$ and $\xi\in\fH$, then we write $w(\xi)$ or $w\xi$ for $\phi(w)(\xi)\in\fH$. 
With this convenient notation we obtain for instance that
$$\Phi(t_n)\xi = \oplus_{w: |w|=n} w\xi\in \fH^{2^n}$$
where $t_n$ is the regular binary trees with all leaves of length $n$ and $\Phi$ is the functor induced by $(A,B)$.

\subsection{Convergence of words arising from $\phi$.}
Using the monoid action $\phi$, we can define two different notions of convergence of words in $W$.

{\bf Convergence of words with respect to the Fr\'echet filter.}
We consider point-wise and uniform limits of words $w\in W$ to zero using the Fr\'echet filter of $\Mon(a,b).$
For $\xi, \eta \in \fH$, say $w\xi$ \textit{converges to $0$ for the Fr\'echet filter}, denoted 
$$\lim_{w\in W} w\xi=0,$$ 
if for all $\varep>0$ we have $\| w'\xi\|<\varep$ for all but finitely many $w'\in \Mon(a,b).$
Similarly, say that $w$ \textit{converges to $0$ for the Fr\'echet filter}, denoted 
$$\lim_{w\in W} w =0,$$ 
if for all $\varep>0$ we have $\|w'\|<\varep$ for all but finitely many $w'\in \Mon(a,b).$

\begin{remark}
\begin{enumerate}
\item We only consider convergence to zero because no other limits can appear when consider words of $A,B$ when $(A,B)$ is a Pythagorean pair.	
\item	In the remainder of the paper, we will prove a variety of results that may only hold when $\fH$ is finite-dimensional and thus where the strong operator topology coincides with the operator (or any other) norm topology of $B(\fH)$.
	However, we will also prove general statements that are valid for \textit{any} Pythagorean pairs including cases where $\fH$ is infinite-dimensional and thus where it makes sense to carefully define the topology considered.
	\end{enumerate}
\end{remark}

{\bf Convergence of words with respect to a ray.}
Here we use the identification of a ray $p$ as an infinite sequence in $a,b$. Hence, $p_n$ is a finite word in $a,b$ of length $n$ that we read from right to left.

\begin{definition}We say that \textit{$w\xi$ converges to $\eta$ for the ray $p$} if $\lim_{n\to\infty}p_n\xi=\eta$ for vectors $\xi, \eta \in \fH$.
Similarly, we say that \textit{$w$ converges to $C$ for the ray $p$} if $\lim_{n\to\infty} p_n=C$ for an operator $C \in B(\fH)$.\end{definition}
We now prove the non-trivial fact that converging (point-wise) to all rays to $0$ is equivalent to converging (point-wise) for the Fr\'echet filter to $0$.

\begin{proposition}\label{prop:Frechet}
	Consider a Pythagorean pair $(A,B)$ over $\fH$ with monoid of words $W$.
	For any $\xi\in\fH$ we have that $w\xi$ converges to $0$ for the Fr\'echet filter if and only if $w\xi$ converges to $0$ for any ray, i.e.$$\lim_{w\in W} w\xi=0 \text{ if and only if } \lim_{n\to \infty}p_n\xi=0 \text{ for all rays } p.$$
\end{proposition}

\begin{proof}
	Fix $A,B,\fH,W$ as above and a vector $\xi\in\fH$ of norm one.
	The proof is deduced from two observations:
	\begin{enumerate}
		\item for any $n\geq 1$ there are only finitely many words of length smaller than $n$;
		\item for any $n\geq 1$ we have that $\xi\mapsto \oplus_{w: |w|=n} w\xi$ is an isometry (from $\fH$ to $\fH^{2^n}$).
	\end{enumerate}
	The first observation gives that $\lim_{w\in W}w\xi=0$ if and only if $\sup\{\|w\xi\|:\ w\in W, |w|=n\}$ tends to $0$ in $n$.
	Since $\sup\{ \|w\xi\|:\ w\in W, |w|=n\}\geq \|p_n\xi\|$ for any ray $p$ we deduce one direction (the easiest one) of the assertion.
	
	Let us prove the converse.
	Assume $\lim_{n\to\infty}p_n\xi=0$ for all rays $p$ and fix $\varep>0$.
	We want to show that for $n$ large enough $\sup\{\|w\xi\| :\ w\in W, |w|=n\}<\varep.$
	Choose $N\geq 1$ satisfying $\frac{1}{N}<\varep^2.$
	For each $n$ we define the set of words $$Q_n:=\{w\in \Mon(a,b):\ |w|=n, \|w\xi\|\geq\varep\}.$$
	We have $|Q_n|\leq N$ by the second observation and the assumption $\|\xi\|^2=1.$
	Observe now that since $A,B$ are contractions we have that if $w\in Q_{n+1}$, then necessarily $w$ is of the form $x\cdot w_n$ where $x=a$ or $b$ and $w_n\in Q_n$. 
	It means that elements of $Q_{n+1}$ are immediate children of elements of $Q_n$.
	This implies that $\cup_{n\geq 1} Q_n$ is contained in the vertex set of at most $N$ rays, i.e.~$\cup_{n\geq 1} Q_n\subset \cup_{i\in I} \Ver_{p^{(i)}}$ where $\{p^{(i)}:\ i\in I\}$ is a set of rays with at most $N$ elements.
	By assumption $w\xi$ converges to $0$ with respect to each of the ray $p^{(i)}$.
	This implies that for $n$ large enough $Q_n$ is empty.
\end{proof}

We immediately deduce the following fact on the projections $\rho_p$ defined in Subsection \ref{subsec:partial-isom}.

\begin{corollary} \label{zero projection corollary}
	If $\limw w\xi = 0$ for all $\xi \in \fH$, then the projections $\rho_p$ acting on $\scrH$ are equal to the zero operator for all rays $p$.
\end{corollary}

When $\fH$ is finite-dimensional we deduce an interesting dichotomy regarding convergence of words and rays. The proof uses the compactness of the closed unit ball of $\fH$. 
We will see in Remark \ref{dichotomy counter prop} that this statement does not generalise in the infinite dimensional case.
\begin{proposition} \label{dichotomy convergence prop}
	Let $\fH$ be a {\bf finite-dimensional} Hilbert space, $(A,B)$ a Pythagorean pair over $\fH$, and $W$ the monoid generated by $A$ and $B$.
	We have the following dichotomy:
	\begin{center}$\lim_{w\in W}w=0$ or there exists $\xi\in\fH$ non-zero and a ray $p$ satisfying $\lim_{n\to\infty}\|p_n\xi\|=\|\xi\|.$\end{center}

\end{proposition}

\begin{proof}
	
Consider a finite-dimensional Hilbert space $\fH$ and a Pythagorean pair $(A,B)$ over $\fH$.
	Assume that $\lim_{w\in W} w\neq 0$. 
	Hence, there exists a unit vector $\xi$ satisfying $\lim_{w\in W}\| w\xi\|\neq 0$.
	By Proposition \ref{prop:Frechet}, there exists a ray $p$ satisfying $\lim_{n\to\infty}\|p_n\xi\|\neq 0.$
	Since both $A,B$ have norm smaller that 1 we necessarily have that the sequence $(\|p_n\xi\|)_{n\geq 1}$ is decreasing. 
	Since this sequence is bounded below it must converge to a certain $\ell>0$.
	Note that $(p_n\xi)_{n\geq 1}$ is a sequence in the closed unit ball of $\fH$. Since $\fH$ is finite-dimensional we have that this ball is compact by the Riesz Theorem.
	Therefore, this sequence admits at least one accumulation point $\eta$.
	We necessarily have that $\|\eta\|=\ell$ and thus $\eta$ is non-zero.
	Our strategy is to construct a new ray $\hat p$ satisfying that $\|\hat p_n\eta\|=\ell$ for all $n\geq 1$.
	
	For any $\varep>0$ define $I_\varep$ the set of indices $n\geq 1$ satisfying that $\| \eta - p_n\xi\|<\varep.$
	By definition $I_\varep$ is infinite for any choice of $\varep>0$ (and obviously nested in $\varep$).
	If $n<m$ we write $p_n^m$ for the word in $a,b$ satisfying $p_m = p_n^m\cdot p_n$.
	Observe that if $n,m\in I_\varep$ and $n<m$, then 
	\begin{align*}
		\|\eta-p_n^m\eta\| & \leq \| \eta - p_m\xi\| + \|p_m\xi - p_n^m\eta\| \\
		& \leq \| \eta - p_m\xi\| + \|p_n^m(p_n\xi - \eta)\| \\
		& \leq \| \eta - p_m\xi\| + \|p_n^m\| \cdot \|(p_n\xi - \eta)\| \\
		& \leq 2\varep.
	\end{align*}
	This implies that 
	$$\|p_n^m\eta\| \geq \|\eta\| - \|p_n^m\eta -\eta\| \geq \|\eta\| - 2 \varep.$$
	Write $x_k$ the $k$th letter of $p$.
	The previous inequality implies that 
	$$\|x_{n+1} \eta\| \geq \|\eta\| - 2\varep \text{ for all } n\in I_\varep.$$
	From now on we choose $\varep>0$ small enough with respect to $\ell=\|\eta\|$ so that $\|\eta\| - 2\varep> \|\eta\|/\sqrt 2.$
	This implies that $I_\varep\ni n\mapsto x_{n+1}$ is constant.
	Indeed, if both $a,b$ would appear as $x_{n+1}$ and $x_{m+1}$ for $n,m\in I_\varep$ we would contradict the Pythagorean equation of $(A,B)$ since we would have:
	$$\|\eta\|^2 = \|A\eta\|^2+\|B\eta\|^2 = \|x_{n+1}\eta\|^2+\|x_{m+1}\eta\|^2> \|\eta\|^2,$$
	a contradiction.
	Let $y_1\in\{a,b\}$ be the letter equal to $x_{n+1}$ for $n\in I_\varep.$
	Observe that $$\|y_1\eta\|\geq \|\eta\|-2\varep' \text{ for all } 0<\varep'\leq\varep.$$
	We deduce that $\|y_1\eta\|=\|\eta\|$.
	A similar reasoning can be applied to the second letter of $p_n^m$ (for $m\geq n+2$ and $n,m\in I_\varep$) and the vector $y_1\eta$ to prove that $n\ni I_\varep\mapsto x_{n+2}$ is constant equal to a certain $y_2$ and moreover $\|y_2y_1\eta\|=\|\eta\|$.
	By induction we obtain a sequence $(y_n)_{n\geq 1}$ and thus a new ray $\hat p$ from it satisfying that 
	$$\|\hat p_k\eta\|=\|y_k\dots y_1 \eta\| = \|\eta\| \text{ for all } k\geq 1.$$
	In particular, note that the ray $\hat p$ is periodic and is equal to ${}_kp$ for some $k \in \N$.
	We have proven that if $\lim_{w\in W}w\neq 0$, then there exists a unit vector (here $\eta/\ell$) and a ray (here $\hat p$) satisfying $\lim_{n\to\infty}\|\hat p_n \eta/\ell\|=1.$
	The converse is obvious.	
\end{proof}

By identifying $\xi$ with its copy inside $\scrH$, we observe that $\lim_{n\to\infty}\norm{p_n\xi} = 1$ if and only if $\rho_p(\xi) = \xi$. 
Hence, either $\limw w=0$ or there exists a ray $p$ satisfying $\Ran(p)\cap \fH\neq \{0\}$.

\subsection{Jones sub-representations and direct sums}

The interest in studying $\phi$-invariant subspaces of $\fH$ is that they can be naturally associated to sub-representations of $\sigma$ in the following manner.

\begin{construction} \label{subspace of scrH from fH construction}
	Let $\fX \subset \fH$ be a $\phi$-invariant subspace ($\fX$ may not necessarily be closed). For a tree $t \in \fT$, write $\fX_t$ to be the subspace of $\fH_t$ that is:
	\[\fX_t := \{(t, \xi): \xi \in \fX^{\tar(t)}\} \subset \fH_t.\]
	Identify $\fX_t$ with its copy in $\scrH$ and then define
	\[\scrK_\fX := \bigcup_{t\in\fT} \fX_t \subset \scrK, \quad \scrH_\fX := \overline{\scrK_\fX} \subset \scrH.\]
\end{construction}
Informally, $\scrH_\fX$ is the closure of the set of all trees whose leaves are decorated with elements in $\fX$. We easily deduce the following fact.

\begin{proposition}\label{subrep of sigma from fH proposition}
	The subspace $\scrH_\fX$ is invariant under the action of $\sigma$ providing a sub-representation
	\[\sigma_\fX := \sigma\restriction_{\scrH_\fX}.\]
	Additionally, $\sigma_\fX$ is equivalent to the Pythagorean representation $\sigma_{A\restriction_{\overline{\fX}}, B\restriction_{\overline{\fX}}}$.
\end{proposition}

In the sequel we shall freely identify the above two representations.

Conversely, if $(A,B)$ and $(A',B')$ are two Pythagorean pairs acting on $\fH$ and $\fH'$ with associated representations $\sigma$ and $\sigma'$, then $(A\oplus A',B\oplus B')$ is a Pythagorean pair acting on $\fH\oplus\fH'$ yielding the representation $\sigma\oplus\sigma'$.
Similarly, we may consider arbitrary direct sums and even direct integrals of Pythagorean pairs and Pythagorean representations.

\subsection{Classes of vectors in $\fH$.} \label{class vectors subsection}
We now introduce some different classes of vectors in $\fH$ that will play an important role in our decomposition of Pythagorean representations.

\begin{definition} \label{class of vectors definition}
	Let $p$ be a ray and $\xi$ be a non-zero vector in $\fH$. Then we say:
	\begin{itemize}
		\item $\xi$ is a vector \textit{contained} in the ray $p$ if $\lim_{n \to \infty}\norm{p_n\xi} = \norm{\xi}$ (or equivalently $\norm{p_n\xi} = \norm{\xi}$ for all $n \in \N$);
		\item $\xi$ is a vector \textit{eventually contained} in the ray $p$ if $\xi$ is contained in some ray belonging to the class $[p]$ (we shall also frequently say $\xi$ is eventually contained in $[p]$);
		\item $\xi$ is a vector \textit{partially contained} in the ray $p$ if $\lim_{n \to\infty} \norm{p_n\xi} = c > 0$ and there does not exist $m \in \N$ such that $\norm{p_m\xi} = c$;
		\item $\xi$ is a vector \textit{annihilated} by all rays if $\lim_{n \to \infty} \norm{q_n\xi} =0$ for all rays $q$ (or equivalently $\limw w\xi = 0$ from Proposition \ref{prop:Frechet}).
	\end{itemize}
\end{definition}

\begin{remark}
	Note, if $\xi$ is partially contained in $p$ then $p_n\xi$ is not contained in any ray for all $n \in \N$. Further, since necessarily $\norm{A}, \norm{B} \leq 1$, then $(\norm{p_n\xi} : n \geq 0)$ forms a decreasing sequence which is bounded below by $0$, and thus the limit always exists. Hence, either $\xi$ is partially contained in some ray, annihilated by all rays or there exists a ray $p$ and $n \in \N$ such that $p_n\xi$ is contained in the ray ${}_np$.
\end{remark}

While the above definitions of the different classes of vectors are analytic in nature and only relies on the monoid action on $\fH$, the definitions are motivated by its diagrammatic implication on the larger Hilbert space $\scrH$. Indeed, let $\xi$ be a vector in $\fH$ identified with its image $[e, \xi]$ inside $\scrH$. 
Moreover, assume that $\xi$ is contained in a fixed ray $p$.
If we consider representatives $(t_n, \ti \xi_n)$ of $\xi$ then the component of $\ti \xi_n$ which lies in the ray $p$ will have its norm conserved as $n$ increased while all other components of $\ti \xi_n$ must be zero due to the equality $\norm{\xi} = \norm{\ti\xi_n}$. \\
Pictorially, this can be visualised by considering the infinite tree $t_\infty$ and labelling each of the vertices $\nu$ with the components $\tau_\nu(\xi)$ of $\xi$. Then if we were to highlight all the vertices with non-zero components with norm equal to $\norm{\xi}$ and connect the edges then we would form the ray $p$ while all the other vertices would be labelled with zeroes. This visualisation motivates the terminology of $\xi$ being ``contained'' in $p$ and shows how $\xi$ ``vanishes'' when it leaves the ray.\\
As an example, consider when $\fH = \C^2, 
A = \begin{pmatrix}
	0 & 0 \\
	1 & 0
\end{pmatrix},
B = \begin{pmatrix}
	0 & 1 \\
	0 & 0
\end{pmatrix}$. 
Then the vector $e_1=\begin{pmatrix}1\\ 0\end{pmatrix}  \in \fH$ is contained in the zig-zag ray $\dots aba$.

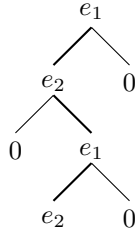
\begin{center}
	\begin{tikzpicture}[baseline=0cm]
		\draw[thick] (0,0)--(-.5, -.5);
		\draw (0,0)--(.5, -.5);
		\draw (-.5, -.9)--(-1, -1.4);
		\draw[thick] (-.5, -.9)--(0, -1.4);
		\draw[thick] (0, -1.8)--(-.5, -2.3);
		\draw (0, -1.8)--(.5, -2.3);
		
		\node[label={[yshift=-5pt] \footnotesize $e_1$}] at (0, 0) {};
		\node[label={[yshift=-18pt] \footnotesize $e_2$}] at (-.5, -.5) {};
		\node[label={[yshift=-18pt] \footnotesize $0$}] at (.5, -.5) {};
		\node[label={[yshift=-18pt] \footnotesize $0$}] at (-1, -1.4) {};
		\node[label={[yshift=-18pt] \footnotesize $e_1$}] at (0, -1.4) {};
		\node[label={[yshift=-18pt] \footnotesize $e_2$}] at (-.5, -2.3) {};
		\node[label={[yshift=-18pt] \footnotesize $0$}] at (.5, -2.3) {};
	\end{tikzpicture}%
\end{center}

Similarly, the above diagram shows that $e_2$ is contained in the ray $\dots bab$. 
The preceding discussion is summarised in the following easy but important observation.

\begin{observation} \label{contained vector norm obs}
	A non-zero vector $\xi \in \fH$ is contained in a ray $p$ if and only $\xi \in \Ran(\rho_p) \cap \fH$. Equivalently, $\norm{\tau_{\nu}(\xi)} = \norm{\rho_{\nu}(\xi)} = \norm{\xi}$ for $\nu \in \Ver_p$ and $\rho_\nu(\xi) = \tau_{\nu}(\xi) = 0$ for $\nu \notin \Ver_{p}$. 
\end{observation}

Again, consider $\xi \in \fH$ with representatives $(t_n, \ti \xi_n) \in [e, \xi]$, and now suppose $\xi$ is a vector annihilated by all rays. 
Then this is equivalent in some sense to requiring all the components of $\ti \xi_n$ tend to $0$ \textit{uniformly} as $n$ tends to $\infty$. 
If $\xi$ is instead partially contained in a ray $p$ then $\xi \notin \ker(\rho_p) \cup \Ran(\rho_p)$.

\begin{example}\label{dichotomy counter prop}
Proposition \ref{dichotomy convergence prop} asserts that if $\fH$ is finite-dimensional then either every vector is annihilated by all rays or there exists a vector contained in some ray. 
We provide a counter-example in the {\it infinite-dimensional case}.
Given a ray $p$ we construct a Pythagorean pair $(A,B)$ acting on an infinite dimensional $\fH$ and satisfying for all ray $p$ that:
	\begin{itemize}
		\item there exists a unit vector $\xi\in\fH$ partially contained in $p$; and
		\item for any ray $q$ and any unit vector $\eta\in\fH$ we have $\lim_{n\to\infty}\| q_n \eta\| <1.$
	\end{itemize}

Observe that if  $1$ is not an eigenvalue of $A^*A$ nor $B^*B$ (i.e.~$A$ and $B$ are strict contractions), then the second item is satisfied.
	Fix a ray $p$ and let $\fH:=\ell^2(\Ver(t_\infty))$ be the Hilbert space of square-summable families of complex numbers indexed by the vertex set of the infinite rooted binary regular tree $t_\infty$.
	As usual $\{\delta_\nu:\ \nu\in \Ver(t_\infty)\}$ denotes the standard orthonormal basis of $\fH$.
	Recall $(\mu_n:\ n\geq 0)$ are the vertices of the ray $p$ so that $\mu_0$ is the root node of $t_\infty$ and $\mu_n$ is at distance $n$ from the root node (note the superscript $p$ has been dropped).
	Choose any sequence $(c_n:\ n\geq 0)$ of elements of $(0,1)$ satisfying that their product converges to $1/2$.
	Finally, for each $n\geq 1$ we write $x_n$ for the $n$th letter (from the right) of the infinite word $p$ (hence $x_n=a$ if the $n$th edge of $p$ is a left-edge and $b$ otherwise).
	If $x_n=a$ we set 
	$$A(\delta_{\mu_{n-1}})=c_{n-1}\delta_{\mu_{n}} \text{ and } B(\delta_{\mu_{n-1}})=\sqrt{1-c_{n-1}^2} \delta_{\mu_{n}}.$$
	If $x_n=b$, then we swap $A$ and $B$ in the equations above.
	If $\nu$ is a vertex of $t_\infty$ that is not a vertex of $p$, then we set 
	$$A(\delta_\nu)=B(\delta_\nu) =1/\sqrt{2} \delta_\nu.$$
	This determines $A$ and $B$ on the standard orthonormal basis and they extend into bounded linear operators on $\fH$.
	Moreover, they satisfy the Pythagorean identity and we have that the spectra of $A$ and $B$ are subsets of 
	$$\{ c_n^2, 1-c_n^2, 1/2:\ n\geq 0\}.$$
	In particular, $1$ is not an eigenvalue of $A^*A$ nor $B^*B$ implying the second item and thus there are no vectors in $\fH$ which are contained in any ray.
	Consider the unit vector $\xi:=\delta_{\mu_\varnothing}$ and observe that $$p_n\xi= \left(\prod_{i=0}^{n-1} c_i \right)\delta_{\mu_{n}}$$ implying that $\lim_{n\to\infty}\|p_n \xi\| = 1/2 > 0.$ 
	Further, there is no $m \in \N$ such that $\norm{p_m\xi} = 1/2$ since $c_n < 1$ for all $n$.
\end{example}

\subsection{Subspaces of $\fH$ associated to classes of vectors} \label{subspace of fH subsection}
To the classes of vectors defined in the previous subsection we shall associate them with $\phi$-invariant subspaces of $\fH$. 
Many of the following definitions and constructions can be more easily understood by studying some simple examples. We suggest the reader to study in tandem the examples provided in Subsection \ref{examples of pythag rep subsection} to help inform the motivation for each of the below constructions.

\begin{definition} \label{subspace of fH defintiion}
	Let $p$ be a ray and define the following subspaces of $\fH$.
	\begin{enumerate}[i]
		\item \label{item fX def} \textit{(subspace of vectors contained in $p$)} 
		Define
		\begin{equation*}
			\fX^p := \{\xi \in \fH : \norm{p_n\xi} = \norm{\xi} \textrm{ for all } n \in \N\}.
		\end{equation*}
		\item \label{item fVp def} \textit{(subspace of vectors quasi-contained in $p$)}
		Define $\fV^p$ to be the set of $\xi\in\fH$ such that there exist $m \geq 1$, $N \in \N$ and $m$ rays $p^{(1)}, \dots, p^{(m)}$ in the equivalence class $[p]$ satisfying
		\begin{equation*} 
			\norm{\xi}^2 = \sum_{k=1}^m \norm{p^{(k)}_n\xi}^2 \textrm{ for all } n \geq N.
		\end{equation*}
		\item \label{item fV def} \textit{(subspace of vectors quasi-contained in rays)}
		Define $\fV$ be the subspace of all vector $\xi \in \fH$ such that there exists $m \geq 1$, $N \in N$ and $m$ rays $p^{(1)}, \dots, p^{(m)}$ such that
		\[\norm{\xi}^2 = \sum_{k=1}^m \norm{p^{(k)}_n\xi}^2 \textrm{ for all } n \geq N.\]
		\item \label{item fW def} \textit{(subspace of vectors partially contained in rays)}
		Define $\fW$ to be the subspace of $\fH$ which consists of all vectors $\xi \in \fH$ such that there exists a set of rays $\{p^{(i)}\}_{i \in I}$ satisfying
		\begin{equation*}
			\phi(v)\xi \in \fV^\perp \textrm{ for all } v \in \Mon(a,b) \textrm{ and } \phi(w)\xi = 0 \textrm{ for all } w \notin \bigcup_{i \in I} \{p^{(i)}_n\}_{n \in \N}
		\end{equation*}
		and $\xi$ is partially contained in each of the rays $p^{(i)}$.
		\item \label{item fU def} \textit{(subspace of vectors annihilated by all rays)} 
		Define
		\begin{equation*}
			\fU = \{\xi \in \fH : \lim_{n\to\infty}\norm{p_n\xi} = 0 \textrm{ for all rays } p\} 
			= \{\xi \in \fH: \limw w\xi = 0\}
		\end{equation*}
		where equality of the two expressions is given by Proposition \ref{prop:Frechet}.
	\end{enumerate}
\end{definition}

\begin{remark} \label{subspace of fH remark}
	\begin{enumerate}[i]
		\item If the ray $p$ starts by $a$, then $\fX^p\subset \ker(B)$ since $\xi\mapsto (A\xi,B\xi)$ is an isometry.		
		\item Note that $\fV^p$ only depends on the class $[p]$ of $p$ and forms a vector subspace of $\fH$ containing $\fX^p$.
		Subsequently, for $[p] \in \cP$ we shall write $\fV^{[p]}$ to mean $\fV^p$ for any ray $p$ in the class $[p]$.
		\item Using notations from item \ref{item fVp def} above, observe that if $\xi\in\fV^p$, then $\norm{p_{N+j}^{(k)}\xi} = \norm{p_N^{(k)}\xi}$ for all $1\leq k\leq m, 1\leq j$ implying that $p_N^{(k)}\xi$ is contained in a truncation of $p^{(k)}$ and thus $\xi$ is eventually contained in finitely many rays in the class of $p$. Hence, we shall say elements of $\fV^p$ are the vectors {\it quasi-contained} in the ray $p$.
		A similar comment also applies to elements of $\fV$.
		\item Clearly $\fW$ is orthogonal to $\fV$ and can be described to be ``strongly'' orthogonal in some sense. This condition is required to ensure $\fW$ is $\phi$-invariant as we will see in Proposition \ref{alt subspace presentation prop}. Interestingly, this condition forces $I$ to be either empty (when $\xi = 0$) or a countably infinite set. Indeed, by a norm argument the set $I$ must be countable. Further, if $I$ is a finite set, then for $i \in I$ there is a large enough $n$ such that $p^{(i)}_n \neq p^{(j)}_n$ for all $j \in I, j \neq i$. Then it follows that $p^{(i)}_n\xi$ is contained in a truncation of the ray $p^{(i)}$ (note $p^{(i)}_n\xi$ is non-zero because $\xi$ is partially contained in $p^{(i)}$) which contradicts the assumption $p^{(i)}_n\xi \in \fV^\perp$. Therefore, $I$ must be countably infinite.
	\end{enumerate}
\end{remark}

The following proposition proves some important properties that will be required later and provides a more convenient presentation of each of the above subspaces by considering them as subspaces of the larger Hilbert space $\scrH$.

\begin{proposition} \label{alt subspace presentation prop}
	Let $p$ be a ray and denote $\mathfrak{S} \subset \fH$ to be the closed subspace formed by the intersection of the subspaces $\phi(v)^{-1}(\fV^\perp)$ for all $v \in \Mon(a,b)$. We then have the following assertions.  
	\begin{enumerate}[i]
		\item \label{item fX ran def} $\fX^p = \Ran(\rho_p) \cap \fH$.
		\item \label{item fX unitary operator} Suppose $\ker(A)$, $\ker(B)$ are both finite-dimensional and $p$ is periodic with period of length $n$. Then $\fX^p$ is finite-dimensional and $\phi(p_n)$ restricts to a unitary operator from $\fX^p$ to itself.
		\item \label{item fVp alt def 1} A vector $\xi \in \fH$ belongs in $\fV^p$ if and only if there exists a finite set of vectors $\{\eta_k\}_{k=1}^m$ in $\fH$ eventually contained in $p$ and a finite set of disjoint vertices $\{\nu_k\}_{k=1}^m$ such that $\xi = \sum_{k=1}^m \tau^*_{\nu_k}(\eta_k)$ (see Subsection \ref{subsec:F-def} for definition of disjoint vertices). 
		\item \label{item fVp alt def 3} The subspace $\fV^p$ is equal to the set of all finite sum of elements belonging in $\cup_{q\in[p]} \Ran(\rho_q)$ such that the sum is contained in $\fH$.
		\item \label{item fV alt def} Items \ref{item fVp alt def 1} and \ref{item fVp alt def 3} hold for $\fV$ when considering all rays rather than rays in the class $[p]$.
		\item \label{item fW alt def 1} A non-zero vector $\xi \in \fH$ belongs in $\fW$ if and only if $\tau_\nu(\xi) \in \fV^\perp$ for all $\nu \in \Ver$ and $\xi$ is an infinite (converging) sum of orthogonal non-zero vectors $\sum_{i \in I} z_i$ where each vector $z_i \in \Ran(\rho_{p^{(i)}})$ for a set of rays $\{p^{(i)}\}_{i \in I}$. In particular, $z_i$ can be taken to be $\rho_{p^{(i)}}(\xi)$.
		\item \label{item fW alt def 2} The subspace $\fW$ is equal to the following expression:
		\[\fW = \biggl(\bigoplus_{p} \Ran(\rho_p)\biggr) \cap \mathfrak{S}\]
		where $p$ ranges over all rays. 
		\item \label{item fU alt def} The space $\fU$ is a closed subspace of $\fH$ and is equal to the following expression:
		\[\fU = \biggl(\bigcap_p \ker(\rho_p) \biggr) \cap \fH = \biggl(\bigoplus_p \Ran(\rho_p) \biggr)^\perp \cap \fH\]
		where $p$ runs across all rays.
		\item \label{item fU contained uncountable ray} If $\xi \in \fU$ is a non-zero vector, then there does not exist a countable set of rays $\{p^{(i)}\}_{i \in I}$ such that $\xi \in \Ran(\bigvee_{i \in I} \rho_{p^{(i)}})$.	
		\item \label{item phi invariant subspace} The subspaces $\fU, \fV^p, \fV, \fW$ are $\phi$-invariant. 
	\end{enumerate}
\end{proposition}

\begin{proof}
	Item \ref{item fX ran def} is trivial; item \ref{item fVp alt def 3} follows from item \ref{item fVp alt def 1}; item \ref{item fV alt def} is immediate once items \ref{item fVp alt def 1} and \ref{item fVp alt def 3} have been proven; item \ref{item fW alt def 2} can be deduced from item \ref{item fW alt def 1} and item iii from Remark \ref{subspace of fH remark}; and item \ref{item fU alt def} is implied by Corollary \ref{zero projection corollary}. Therefore, we only need to prove items \ref{item fX unitary operator}, \ref{item fVp alt def 1}, \ref{item fW alt def 1}, \ref{item fU contained uncountable ray} and \ref{item phi invariant subspace}. 
	
	Proof of \ref{item fX unitary operator}.~ Since $p=c^\infty$ for some word $c$ of length $n$ we deduce that $\phi(p_n)(\fX^p)\subset \fX^p$ and moreover the restriction is norm-preserving. By item i in Remark \ref{subspace of fH remark}, $\fX^p$ is either contained inside $\ker(A)$ or in $\ker(B)$ implying by assumption that $\fX^p$ finite-dimensional. Therefore, $\phi(p_n)$ restricts to a unitary operator from $\fX^p$ to itself. 
	
	Proof of \ref{item fVp alt def 1}.~ Let $\xi \in \fV^p$, then the result follows from item ii in Remark \ref{subspace of fH remark} by taking $\eta_k$ to be $p^{(k)}_N\xi$ and $\nu_k$ to be the vertex contained in $p^{(k)}$ with length $N$.\\
	Then suppose $\xi \in \fH$ is a vector such that $\xi = \sum_{k=1}^m \tau^*_{\nu_k}(\eta_k)$ as described in item \ref{item fVp alt def 1}. Each of the vectors $\tau^*_{\nu_k}(\eta_k)$ belong in the range of $\rho_{p^{(k)}}$ for some ray $p^{(k)}$ in the class of $p$. Take $N$ to be the maximum length of the vertices $\{\nu_k\}_{k=1}^m$. Then it is clear that the vectors $p^{(k)}_N\xi$ are contained in a truncation of $p^{(k)}$ and $w\xi = 0$ if $w \notin \cup_{n\geq N} \{p^{(k)}_n\}_{k=1}^m$. From this we can conclude $\norm{\xi}^2 = \sum_{k=1}^m \norm{p^{(k)}_n\xi}^2$ for all $n \geq N$ which shows that $\xi \in \fV^p$ and proves the result.
		
	Proof of \ref{item fW alt def 1}.~ First, note the condition $\phi(v)\xi \in \fV^\perp$ for all $v \in \Mon(a,b)$ is equivalent to the condition $\tau_\nu(\xi) \in \fV^\perp$ for all $\nu \in \Ver$. Suppose $\xi$ is a non-zero vector in $\fW$ and take $\{p^{(i)}\}_{i \in I}$ to be the set of rays as described in the definition of $\fW$ in Definition \ref{subspace of fH defintiion}. The set $I$ is countably infinite by item iii in Remark \ref{subspace of fH remark}. It is clear $\{\rho_{p^{(i)}}\}$ forms an orthogonal set of projections. Hence, by Proposition \ref{MCT projection prop} the sum $\sum_i \rho_{p^{(i)}}$ converges in the strong operator topology to the projection $\rho_I$ associated to the closed subspace $\bigvee_i \Ran(\rho_{p^{(i)}})$. By definition of $\fW$, for all $n \in \N$ we have
	\[\xi \in \bigvee_{i \in I} \Ran(\rho_{p^{(i)}, n})\]
	where recall $\rho_{p^{(i)}, n} = \rho_{\nu}$ and $\nu$ is the vertex in $p^{(i)}$ with length $n$. Since the above holds for all $n \in \N$ and $\rho_{p^{(i)}} = \bigwedge_n \rho_{p^{(i)},n}$ it follows that
	\[\xi \in \bigvee_{i \in I} \Ran(\rho_{p^{(i)}}) = \Ran(\rho_I).\]
	Thus we can write
	\[\xi = \rho_I(\xi) = \sum_{i \in I} \rho_{p^{(i)}}(\xi).\]
	Since $\xi$ is partially contained in each of the rays $p^{(i)}$, the terms $\rho_{p^{(i)}}(\xi)$ are orthogonal non-zero vectors in $\Ran(\rho_{p^{(i)}})$ which proves one direction of the result. 
	
	Next we consider the reverse direction. Suppose $\xi = \sum_{i \in I} z_i$ where $z_i$ are orthogonal non-zero vectors in $\Ran(\rho_{p^{(i)}})$ and assume without loss of generality that $p^{(i)} \neq p^{(j)}$ for $i \neq j$. Since $z_i$ is non-zero, $\lim_{n\to\infty}\norm{p^{(i)}_n\xi} = \norm{z_i} > 0$. If there exists $n \in \N$ such that $\norm{p^{(i)}_n(\xi)} = \norm{z_i}$ that would imply $p^{(i)}_n(\xi) = \tau_{p^{(i)}, n}(\xi)$ is contained in a truncation of the ray $p^{(i)}$ which contradicts the assumption $\tau_\nu(\xi) \in \fV^\perp$ for all $\nu \in \Ver$. Therefore, we can conclude $\xi$ is partially contained in each of the rays $p^{(i)}$. Further, it follows from a norm argument that $\phi(w)\xi = 0$ for all $w \in \cup_{i \in I}\{p^{(i)}_n\}_{n \in \N}$ which shows that $\xi \in \fW$ as required. 
	
	Proof of \ref{item fU contained uncountable ray}.~ Suppose such a set of rays exist. Applying Proposition \ref{MCT projection prop} we obtain the orthogonal sum $\xi = \sum_i \rho_{p^{(i)}}(\xi)$. This gives
	\[\norm{\xi}^2 = \sum_i \norm{\rho_{p^{(i)}}(\xi)}^2 = 0\]
	because $\xi$ belongs in the kernel of $\rho_p$ for all rays $p$. This is of course a contradiction and shows that no such countable set of rays can exist. 
	
	Proof of \ref{item phi invariant subspace}.~ It is easy to conclude $\fU$ is $\phi$-invariant by observing $p_nA\xi = p_n'\xi$ where $p' = p\cdot a$ and similarly for $B\xi$. \\
	Now let $\xi \in \fV^p$ and using the presentation of $\fV^p$ from item \ref{item fVp alt def 3}, suppose $\xi$ is a sum of vectors which are contained in a finite set of rays $R \subset [p]$. Then $A\xi$ will be a sum of vectors which are contained in a subset of $R$ consisting of all the rays that begin with a left turn and likewise for $B\xi$. This shows that $\fV^p$ is a $\phi$-invariant subspace and we arrive at the same conclusion for $\fV$ by following a similar reasoning. \\ 
	Finally, let $\xi \in \fW$ be non-zero and by item \ref{item fW alt def 2} suppose $\xi$ is now a sum of vectors contained in a countably infinite set of rays $R$. Again, $A\xi$ will be contained in the subset $R' \subset R$ consisting of all the rays that begin with a left turn. Note $R'$ must either be empty (thus $A\xi = 0$) or be countably infinite because otherwise that would imply $A\xi \in \fV$ which contradicts the fact $A\xi$ is a non-zero vector in $\fV^\perp$. This shows $A\xi \in \fW$ and we can conclude similarly for $B\xi$. This proves $\fW$ is a $\phi$-invariant subspace which completes the proof of the proposition.
\end{proof}

\begin{remark} \label{alt subspace presentation remark}
	\begin{enumerate}[i]
		\item While it is not obvious from Definition \ref{subspace of fH defintiion}, items \ref{item fVp alt def 3}, \ref{item fV alt def}, \ref{item fW alt def 2}, \ref{item fU alt def} show that $\fU, \fV^p, \fV, \fW$ are indeed vector subspaces of $\fH$.  However, note that $\fX^p$ is not necessarily a $\phi$-invariant subspace because $p_1\xi$ will be contained in the ray ${}_1p$ which does not necessarily equal $p$ (this will only hold if $p$ is the ray $\dots aaa$ or $\dots bbb$).  
		\item Informally, items \ref{item fX ran def}, \ref{item fV alt def}, \ref{item fW alt def 1}, \ref{item fU contained uncountable ray} show that $\fX^p$ is the space of all vectors contained in a ray $p$, $\fV$ is the space of all vectors contained in a finite set of rays, $\fW$ is the space of all vectors contained in a countably infinite set of rays while $\fU$ is the space of all vectors contained in an uncountably infinite set of rays. 
	\end{enumerate}
\end{remark}

Having provided some alternate descriptions of the subspace $\fU, \fV, \fW$, we can now easily show that these subspaces are orthogonal to each other.

\begin{proposition} \label{subspace fH orthogonal prop}
	\begin{enumerate}[i]
		\item \label{item fX orthogonal} If $p,q$ are non-identical rays then $\fX^p$ is orthogonal to $\fX^q$.
		\item \label{item fVp orthogonal} If $p,q$ are rays belonging to different equivalence classes then $\fV^p$ is orthogonal to $\fV^q$.
		\item \label{item subspace fH orthogonal} The subspaces $\fU, \fV, \fW$ are orthogonal to each other.
	\end{enumerate}
\end{proposition}

\begin{proof}
	Proof of i.~ Since $\fX^p = \Ran(\rho_p) \cap \fH$, it is sufficient to show that $\Ran(\rho_p), \Ran(\rho_q)$ are orthogonal subspaces if $p \neq q$. If $p \neq q$, then there exists $n \in \N$ such that the $n$th vertices of the two rays are different and thus disjoint. The result can then be concluded by the following:
	\[\Ran(\rho_p) \subset \Ran(\rho_{p,n}) \perp  \Ran(\rho_{q,n}) \supset \Ran(\rho_q).\] 
	Proof of ii.~ This immediately follows from item \ref{item fVp alt def 3} in Proposition \ref{alt subspace presentation prop} and the fact $\Ran(\rho_r), \Ran(\rho_s)$ are orthogonal if the rays $r,s$ are not equal as shown above. \\
	Proof of iii.~ The result immediately follows from items \ref{item fVp alt def 3}, \ref{item fW alt def 2}, \ref{item fU alt def} in Proposition \ref{alt subspace presentation prop}.
\end{proof}

\begin{remark} \label{subspace fH orthogonal remark}
	\begin{enumerate}[i]
	\item Item \ref{item fX orthogonal} shows that if $\xi$ is contained in $p$ then $\{p_i\xi\}_{i=1}^n$ (or $\{p_i\xi\}_{i \in \N^*}$ if $p$ is not eventually periodic) forms an orthogonal set in $\fH$ where $n$ is the length of the period of $p$. Recalling $p_i\xi \in \ker(\phi(\ti p_{i+1}))$ from item i in Remark \ref{subspace of fH remark}, it follows that $p$ must be eventually periodic if $\fX^p$ is a non-trivial subspace and either $\ker(A)$ or $\ker(B)$ is finite-dimensional.
	\item From the above observation, it follows in general most of the subspaces $\fV^p$ will be trivial. In particular, if $A$ and $B$ are injective or $\norm{A}, \norm{B} < 1$ then $\fV^p$ will be trivial for all rays $p$. 
	\item From the proposition we can form the direct sum $\oplus_{p \in \cP} \fV^p$ which is clearly a subspace of $\fV$. However, we do not necessarily have equality as shown by Example \ref{pythag ex 1}.
	\end{enumerate}
\end{remark}

Motivated by the description of vectors in $\fV, \fW$ given by item ii in Remark \ref{alt subspace presentation remark} and by Construction \ref{subspace of scrH from fH construction}, we introduce the following natural terminology.

\begin{definition} \label{atomic definition}
	Let $(A,B)$ be a Pythagorean pair on a Hilbert space $\fH$ with associated representation $(\sigma,\scrH)$ \tb{such that the orthogonal complement $(\fU \oplus \fV \oplus \fW)^\perp$ in $\fH$ is finite-dimensional}. Consider $\sigma_\fU$ and $\sigma_{\fV\oplus\fW}$ the sub-representations generated by $\fU$ and $\fV\oplus\fW$.
	We say that $\sigma_\fU$ and $\sigma_{\fV\oplus\fW}$ are the diffuse and atomic parts of $\sigma$, respectively.
Moreover, we call $(A,B)$ and $\sigma$ diffuse (resp.~atomic) when $\fH=\fU$ (resp.~$\fH=\fV\oplus\fW$).
\end{definition}	

\begin{remark}
\tb{The condition that $(\fU \oplus \fV \oplus \fW)^\perp$ is finite-dimensional is a technical requirement to ensure that $\sigma = \sigma_\fU \oplus \sigma_{\fV \oplus \fW}$ as will be shown in Theorem \ref{complete scrH decomposition theorem}.}
\end{remark}

\subsection{Decomposition of the Jones representation using the monoid action.}

Recall that $\fH$ is acted upon by the free monoid in two generators given by the map 
$$\phi:\Mon(a,b)\to B(\fH),\ a\mapsto A, b\mapsto B.$$

Using Proposition \ref{subrep of sigma from fH proposition} it is obvious that a (orthogonal) decomposition of the monoid action $\phi$ on $\fH$ will induce a decomposition of the Pythagorean representation $\sigma$ on the larger Hilbert space $\scrH$. However, monoid actions do not have an orthogonal decomposition in general. Therefore, a question of particular interest is whether the specific monoid action $\phi$ is decomposable. Examples \ref{pythag ex 1} and \ref{pythag decomp example} provide a negative answer to this question. However, the next result shows that $\phi$ is ``weakly'' decomposable in a certain sense which surprisingly will be sufficient to yield a decomposition of $\scrH$. 

\begin{proposition} \label{pythag invariant decomp prop}
Let $(A,B)$ be a Pythagorean pair of operators acting on $\fH$ and $(\sigma, \scrH)$ be the associated Pythagorean representation. Then the following assertions are true.
	\begin{enumerate}[i]
		\item \label{item complete decompose criteria} Suppose $\fX \subset \fH$ is a $\phi$-invariant subspace \tb{and $\fX^\perp$ is finite-dimensional}. If $\fX^\perp$ does not contain any non-trivial $\phi$-invariant subspaces then $\scrH_\fX = \scrH$ and $\sigma_\fX = \sigma$. Conversely, if $\fX^\perp$ does contain a non-trivial $\phi$-invariant subspace then $\scrH_\fX \neq \scrH$ and $\sigma_\fX \neq \sigma$.
		\item \label{item fH weakly decompose} If $\fX \subset \fH$ is a $\phi$-invariant subspace \tb{and $\fX^\perp$ is finite-dimensional}, then there exists an orthogonal $\phi$-invariant subspace $\fY \subset \fX^\perp$ such that
		\[\scrH = \scrH_\fX \oplus \scrH_\fY, \quad \sigma = \sigma_\fX \oplus \sigma_\fY.\]
		\item \label{item fH completely reducible} 
		If $\fH$ is finite-dimensional then there exists a $\phi$-invariant subspace $\fX \subset \fH$ such that $(\scrH, \sigma) \cong (\scrH_{A\restriction_\fX, B\restriction_\fX}, \sigma_{A\restriction_\fX, B\restriction_\fX})$. In addition, $\fX$ has a decomposition into orthogonal $\phi$-invariant subspaces $\{\fX_i\}_{i \in I}$ such that
		\[\scrH = \oplus_{i \in I} \scrH_{\fX_i} \cong \oplus_{i \in I} \scrH_{A\restriction_{\fX_i}, B\restriction_{\fX_i}}, \quad \sigma = \oplus_{i \in I} \sigma_{\fX_i} \cong \oplus_{i \in I} \sigma_{A\restriction_{\fX_i}, B\restriction_{\fX_i}}\]
		where each subspace $\fX_i$ does not contain any proper $\phi$-invariant subspace.
	\end{enumerate}
\end{proposition}

\begin{proof}
	\tb{
	Proof of \ref{item complete decompose criteria}.~ First, for the converse, let $\fY \subset \fX^\perp$ be a non-trivial $\phi$-invariant subspace. Then $\scrH_\fY \subset \scrH_\fX^\perp$. Since $\scrH_\fY$ is non-trivial it follows $\scrH_\fX \neq \scrH$ and $\sigma_\fX \neq \sigma$.
	\\
	To prove the forward direction, we aim to show that $\fX^\perp \subset \scrH_\fX$.
	\tb{Heuristically we will proceed as follows: given any vector $\xi\in\fX^\perp$ we may consider a tree $t$ and $\xi_t\in\fH_t$, the representative of $\xi$ inside $\scrH$ obtained by labelling the leaves of $t$ by $w\xi$ where $w$ is an appropriate word in $A,B$. Then, when $t$ grows we note that the distance between each component of $\xi_t$ and $\fX$ tends to zero. Moreover, we will show that this convergence is ``global'' among all the leaves in such a way that the distance between the whole vector $\xi_t$ and $\scrH_\fX$ is manifestly small for $t$ large. 
	This will be achieved by making a compactness argument on the closed unit ball $(\fX^\perp)_1$.}\\
	\tb{Assume $\fX^\perp$ is nonzero (the zero case being trivially true).} Let $P$ be the projection of $\fH$ onto $\fX^\perp$ and define the sequence of continuous maps:
	\[\psi_n : \fX^\perp \rightarrow \R,\ \xi \mapsto \sum_{\vert w \vert = n} \norm{P(w\xi)}^2\]
	\tb{where the sum is taken over all words $w$ made of $n$ letters.
	Since $\xi\mapsto \sum_{|w|=n} w\xi$ is an isometry of $\fH$ and $\fX$ is closed under the action of $A,B$, we deduce that }$\psi_n(\xi) = \norm{\xi}^2$ if and only if $w\xi \in \fX^\perp$ for all words $w$ of length smaller or equal to $n$. 
	\tb{Since by assumption the space $\fX^\perp$ is finite-dimensional we have that its closed unit ball $(\fX^\perp)_1$ is compact. Therefore, }each of the maps $\psi_n$ restricted to $(\fX^\perp)_1$ attain a maximum value $M_n$ at some vector $\xi_n \in (\fX^\perp)_1$.
	The remainder of the proof will be separated into individual claims.
	\\
	\textbf{Claim 1:}
	There exists a natural number $N > 0$ such that $M_N < 1$.
	\\
	Suppose that $M_n = 1$ for all $n$. Again, by appealing to the compactness of the closed unit ball, there exists a sub-sequence $(\xi_{m_n})$ which converges to some $\xi \in (\fX^\perp)_1$. Let $w$ be any word of arbitrary length $k$. It then follows that
	\[\norm{w\xi} = \lim_{n \to\infty} \norm{w\xi_{m_n}} = 1\]
	because $\norm{w\xi_{m_n}} = 1$ for all $m_n \geq k$. From the above observation this implies that $w\xi \in \fX^\perp$ for all words $w$. However, this is a contradiction because $\fX^\perp$ does not contain any non-trivial $\phi$-invariant subspaces by assumption. This proves the claim.
	\\
	From the above claim let $N > 0$ be such that $M_N < 1$.
	For the remainder of the proof we fix a unit vector $\xi \in \fX^\perp$. Since $\fX$ is a $\phi$-invariant subspace it is clear that $(\psi_{n}(\xi) : n > 0)$ forms a decreasing sequence which is bounded below by $0$. Hence, the sequence limits to some constant $M_\xi \in [0,1)$. 
	\\
	\textbf{Claim 2.}
	The constant $M_\xi$ is equal to zero.
	\\
	Suppose $M_\xi > 0$. Then fix some $\epsilon > 0$ and let $K > 0$ be such that $\psi_k(\xi) - M _\xi< \epsilon$ for all $k \geq K$. We can then write for any $j > 0$:
	\[\psi_{j+k}(\xi) = \sum_{|v|=j,|u|=k} \norm{P(vu\xi)}^2 = \sum_{|v|=j,|u|=k}  \norm{PvP(u\xi)}^2 = \sum_{|u| = k} \psi_j(P(u\xi))\]
	where the second equality follows by noting that $PvP = Pv$ for all words $v$ because $\fX$ is invariant under $\phi$. Hence we obtain for $k \geq K$:
	\begin{align*}
		\psi_k(\xi) - \psi_{N+k}(\xi) &= \sum_{\vert w \vert = k} \norm{P(w\xi)}^2 - \psi_N(P(w\xi)) \geq \sum_{\vert w \vert = k} \norm{P(w\xi)}^2 - M_N\norm{P(w\xi)}^2 \\
		&= \psi_k(\xi)(1-M_N) \geq M_\xi(1 - M_N) > 0. 
	\end{align*}
	However, $\psi_k(\xi) - \psi_{N+k}(\xi) < \epsilon$ which can be made arbitrarily small. Therefore, this yields a contradiction and we can conclude that $M_\xi = 0$.
	\\
	\textbf{Claim 3.}
	The subspace $\scrH_\fX$ is equal to $\scrH$.
	\\
	Fix again $\epsilon > 0$. From the above there exists $k$ such that $\psi_k(\xi) < \epsilon$. This implies that there exists $\eta_\nu \in \fX$ for each vertex $\nu$ of length $k$ such that
	\[\norm{\xi - \sum_{\vert \nu \vert = k} \tau^*_\nu(\eta_\nu)}^2 = \psi_k(\xi) < \epsilon.\]
	This shows that $\xi$ belongs in the norm closure of $\scrH_\fX$. Since $\scrH_\fX$ is closed, we have $\xi \in \scrH_\fX$. Therefore, we have shown $\fX^\perp \subset \scrH_\fX$. By construction, $\scrH_{\fH} = \scrH$. Thus, we conclude $\scrH_\fX = \scrH$ which proves the claim and completes the proof.
	\\
	Proof of \ref{item fH weakly decompose}.~ The statement is obviously true when $\scrH_\fX = \{0\}$ or $\scrH_\fX = \scrH$. Hence, suppose $\scrH_\fX$ is a proper (closed) subspace of $\scrH$ and thus $\scrH_\fX^\perp$ is also a proper subspace. By the first item, there must exist a non-trivial $\phi$-invariant subspace of $\fX^\perp$. Let $\fY$ to be the largest such subspace of $\fX^\perp$ (take the sum of all $\phi$-invariant subspaces of $\fX^\perp$). Then $(\fX \oplus \fY)^\perp$ does not contain any non-trivial $\phi$-invariant subspaces and we can conclude from the first item.}
	
	Proof of \ref{item fH completely reducible}.~ Existence of such subspaces $\{\fX_i\}_{i \in I}$ follows by iteratively applying the result from item \ref{item fH weakly decompose} which must eventually terminate because $\fH$ is finite-dimensional. Taking $\fX := \oplus_{i \in I}\fX_i$ proves the required result.
\end{proof}

\begin{remark} \label{rem:fH-decomp}
	\tb{Items \ref{item complete decompose criteria} and \ref{item fH weakly decompose} no longer hold true if the requirement that $\fX^\perp$ is finite-dimensional is relaxed. For example, consider $\fH = \ell^2(\N) \oplus \ell^2(\N)$. Let $\{\delta_x\}_{x \in \N}$ (resp. $\{\gamma_y\}_{y \in \N}$) be the usual basis of the first (resp. second) copy of $\ell^2(\N)$ in $\fH$. Let $(c_n : n \geq 0)$ be any sequence of numbers in $(0,1)$ such that the product $\prod_{n=1}^k c_n$ converges to $c > 0$ (for example, $c_n = e^{-1/(n+1)^2}$) and set $d_n = \sqrt{1-c_n^2}$. Define the operators $A,B \in B(\fH)$ by:
	\begin{align*}
		A\delta_{2x} &= \delta_{2x+1},\  B\delta_{2x} = 0,\ A\delta_{2x+1} = \frac{1}{\sqrt{2}}\delta_{2x},\ B\delta_{2x+1} = \frac{1}{\sqrt{2}}\delta_{2x+1}, \\ 
		A\gamma_y &= c_y\gamma_{y+1},\ B\gamma_y = d_y\delta_{2y}.		
	\end{align*}
	It is straightforward to verify that $(A,B)$ indeed defines a Pythagorean pair. Let $\fX := \ell^2(\N) \oplus \{0\} \subset \fH$ and $\fZ := \{0\} \oplus \ell^2(\N) \subset \fH$. It is easy to see that $\fX$ is a $\phi$-invariant subspace of $\fH$ while $\fZ$ does not contain any non-trivial $\phi$-invariant subspaces. However, $\scrH_\fX$ is not equal to $\scrH$. For example, consider $\gamma_0 \in \fZ$. We have \[\tau_{0^k}(\gamma_0) = \prod_{n=0}^{k-1} c_n\gamma_k \in \fZ\] 
	for all $n \geq 0$ and in particular $\rho_\ell(\gamma_0)$ (recall $\ell$ is the ray going down the left side of $t_\infty$) is a non-zero vector which generates a sub-representation orthogonal to $\scrH_\fX$. This demonstrates that the above proposition does not hold when $\fX^\perp$ is infinite-dimensional.
}
\end{remark}

\subsection{Decomposition of the Jones representation using the Pythagorean C*-algebra}
In this subsection we shall assume the Hilbert space $\fH$ is \textit{separable} in order to use direct integral of representations. We will use the Pythagorean C*-algebra $P:=P_2$ (see Definition \ref{universal pythag algebra definition}) and the $*$-representation: 
\[\psi: P_2 \rightarrow B(\fH), a \mapsto A, b\mapsto B\]
for decomposing the Jones action $\sigma:F\act \scrH$.
Moreover, we compare this decomposition with the previous one using the monoid representation: 
$$\phi:\Mon(a,b)\to B(\fH).$$

Since $P$ contains $\Mon(a,b)$ a direct sum decomposition of $\fH$ with respect to the C*-algebra $P$ is coarser than the one from the monoid $\Mon(a,b)$. However, we can use the powerful theory of C*-algebras for decomposing {\it any} representation of $P$ into {\it irreducible} ones by using direct integrals rather than direct sums as stated in the next proposition. 

\begin{proposition} \label{direct integral decomp fH proposition}
Consider the representation $\psi:P\to B(\fH)$ of the C*-algebra $P$.
There exists a standard Borel space $(X,\mu)$ equipped with a $\sigma$-finite measure and a measurable field of representations $( (\psi_x, \fH_x):\ x\in X)$ of $P$ that are irreducible almost everywhere such that $(\psi,\fH)$ is unitary equivalent to their direct integral.\\
Moreover, $(A,B)$ decomposes as a measurable field $((A_x, B_x) := (\psi_x(a), \psi_x(b)):x\in X)$ that are Pythagorean pair on $\fH_x$ for almost every $x\in X$ giving a family of Jones representations $(\sigma_x, \scrH_x)$ whose direct integral is unitary equivalent to $(\sigma,\scrH)$. 
\end{proposition}

Using the C*-algebra $P$ rather than the monoid $\Mon(a,b)$ is only useful when $\fH$ is infinite-dimensional. 
Indeed, up to taking a unitary equivalent Jones representation, in the finite-dimensional case we may decompose $\fH$ into an orthogonal direct sum of irreducible monoid representations which are of course irreducible for $P$, see the last item of Proposition \ref{pythag invariant decomp prop} for the precise statement.
Although, an irreducible representation $\fH$ of the C*-algebra $P$ may contain a sub-representation $\fX$ of the monoid $\Mon(a,b)$ that is reducible and whose restricted Pythagorean pair $(A\restriction_\fX,B\restriction_\fX)$ gives a reducible representation of $P$ as witnessed in the following example.

\begin{example}
Set $\fH := \ell^2(\Z)$ and define $A,B \in B(\fH)$ given by
		\begin{equation*}
			A(\delta_n) = 
			\begin{cases*}
				\delta_{n+1}, \quad &$n > 0$ \\
				\frac{1}{\sqrt{2}}\delta_1, \quad &$n = 0$\\
				0, \quad &$n < 0$	
			\end{cases*}
			,\quad 
			B(\delta_n) = 
			\begin{cases*}
				\delta_{n-1}, \quad &$n < 0$ \\
				\frac{1}{\sqrt{2}}\delta_{-1}, \quad &$n = 0$\\
				0, \quad &$n > 0$.	
			\end{cases*}
		\end{equation*}
	The operators $(A,B)$ form a Pythagorean pair providing a representation $\psi:P\to B(\fH).$
Now we show that any non-zero vector is cyclic for $P$ implying that $\psi$ is irreducible.
Let $\xi = \sum_{k\in K} \alpha_k \delta_{x_k}$ be an arbitrary element in $\fH$ where each of the coefficients $\alpha_k$ are non-zero.
Up to swapping the roles of $A$ and $B$ we may assume that one of the indices $x_k$ is non-negative. Then by applying the operator $A$ we can assume each of the indices are positive. Repeatedly applying $A^*$ we can assume the minimum of the indices $\{x_k\}_{k \in K}$ is equal to $0$. 
Then applying $B^*B$ gives $\lambda \delta_0$ for some non-zero scalar $\lambda$. 
It is easy to conclude from there that $\xi$ is cyclic since clearly $\delta_0$ is cyclic.
	
Now define the following proper closed subspaces of $\fH$:
	\[\fX_1 := \cspan\{\delta_n\}_{n \geq 1}, \quad \fX_2 := \cspan\{\delta_n\}_{n \leq -1}, \quad \fX := \fX_1 \oplus \fX_2.\]
Note that $\fX$ defines a monoid sub-representation of $\fH$ for $\Mon(a,b)$ and so do $\fX_1$ and $\fX_2$.
Moreover, $(A\restriction_\fX,B\restriction_\fX)$ provides a representation of the C*-algebra $P$ on $\fX$ that decomposes as well on $\fX_1$ and $\fX_2$. 
Additionally, the associated Jones representation $\sigma:F\act \scrH$ is also reducible because it contains the sub-representations $\sigma_{\fX_1}, \sigma_{\fX_2}$. More precisely, $\sigma = \sigma_{\fX_1} \oplus \sigma_{\fX_2}$ by item \ref{item complete decompose criteria} in Proposition \ref{pythag invariant decomp prop} since $(\fX_1 \oplus \fX_2)^\perp = \C \delta_0$ does not contain any non-trivial $\phi$-invariant subspaces.
\end{example}

\subsection{Canonical decomposition of Pythagorean representations} \label{pythag decomp section}

By Proposition \ref{alt subspace presentation prop}, the subspaces $\fU, \{\fV^p\}_{p \in \cP}, \fW$ are $\phi$-invariant (i.e.~invariant under the action of $A$ and $B$) and are pair-wise orthogonal by Proposition \ref{subspace fH orthogonal prop}. 
Hence, by Construction \ref{subspace of scrH from fH construction} and Proposition \ref{subrep of sigma from fH proposition} we can define the subspaces $\scrH_\fU, \scrH_{\fV^p},$ and $\scrH_\fW$ of $\scrH$ and their associated pair-wise orthogonal sub-representations $\sigma_\fU, \sigma_p := \sigma_{\fV^p},$ and $\sigma_\fW$ of $\sigma$ where $p$ stands for a ray. 
Observe $(\sigma_p, \scrH_{\fV^p})$ only depends on the class $[p]\in\cP$ of $p$ (i.e.~it is invariant under changing finite prefixes of $p$).

The above allows us to recognise that Pythagorean representations naturally have diffuse and \at{} sub-representations (see Definition \ref{atomic definition}).
In fact, these sub-representations provide a complete decomposition of the Pythagorean representation $\sigma:F\act \scrH$.

\begin{theorem}\label{complete scrH decomposition theorem}
	Let $\scrH$ be the Hilbert space induced by a Pythagorean pair of operators $(A,B)$ acting on a Hilbert space $\fH$. Then we have the following orthogonal decomposition:
	\[\scrH_{\fV} = \oplus_{p \in \cP} \scrH_{\fV^p}, \quad \sigma_\fV = \oplus_{p \in \cP} \sigma_p.\]
	In addition, \tb{if the orthogonal complement $(\fU \oplus \fV \oplus \fW)^\perp$ in $\fH$ is finite-dimensional then} the Pythagorean representation has the orthogonal decomposition:
	\[\scrH = \scrH_{\fU} \oplus\scrH_{\fV} \oplus \scrH_\fW, \quad \sigma = \sigma_\fU \oplus\sigma_\fV \oplus \sigma_\fW.\]
Moreover, if $\fH$ is finite-dimensional, then $\fW$ is trivial.
\end{theorem}

\begin{proof}
\tb{Item \ref{item fV alt def} in Proposition \ref{alt subspace presentation prop} shows that $\fV \subset \oplus_{p \in \cP} \scrH_{\fV^p}$ and thus it immediately follows that $\scrH_\fV = \oplus_{p \in \cP} \scrH_{\fV^p}$.}

Now suppose $\fH$ is finite-dimensional and $\fW$ contains a non-zero vector $\xi$. Then $\xi$ is partially contained in some ray $p$ with $\lim_{n\to\infty}\norm{p_n\xi} = c > 0$. From the proof of Proposition \ref{dichotomy convergence prop}, there exists a subsequence $(p_{m_n}\xi : n \geq 0)$ and $\eta \in \fH$ contained in some ray $q$ (thus $\eta \in \scrH_{\fV}$) such that $\norm{\eta} = c$ and $\lim_{n\to\infty}p_{m_n}\xi = \eta$. However, $\fW$ is invariant under $\phi$ and thus each of the vectors $p_{m_n}\xi$ are orthogonal to $\eta$. This gives
\[c^2 = \langle \eta, \eta \rangle = \langle \lim_{n \to \infty} p_{m_n}\xi, \eta \rangle = \lim_{n\to\infty} \langle p_{m_n}\xi, \eta \rangle = 0.\]
This is of course a contradiction and subsequently $\fW$ must be the trivial subspace. This \tb{proves the last statement of the theorem.}
	
\tb{Finally}, we examine the general decomposition of $\scrH$. Using that $\fU,\fV,\fW$ are pairwise orthogonal and Proposition \ref{pythag invariant decomp prop} it is sufficient to show that $(\fU \oplus \fV \oplus \fW)^\perp \subset \fH$ does not contain any non-trivial $\phi$-invariant subspaces.
	
Suppose such a subspace exists which we shall denote by $\fX$ and let $\xi$ be a non-zero vector in $\fX$. Since $\xi \notin \fU$, there exists a ray $p$ such that $\lim_{n\to\infty} \norm{p_n \xi} = c > 0$. If there exists $k \in \N$ such that $\norm{p_k \xi} = c$ then it follows that $p_k\xi$ is a non-zero vector contained in the ray ${}_kp$ and thus $p_k\xi \in \fV$ (see Definition \ref{subwords of ray definition}). However, this contradicts the assumption: $\fX$ is $\phi$-invariant and orthogonal to $\fV$. Therefore, no such $k$ can exist and it follows $\xi$ is partially contained in the ray $p$.
\tb{Thus, we have a finite-dimensional space $\fX$ that is closed under the actions of $A,B$ and a non-zero vector $\xi\in\fX$ that is partially contained in a ray $p$. We can then apply the argument of above (the one used to prove the last statement of the theorem) and derive a contradiction.}
\end{proof}

\begin{corollary} \label{range rho inside fH corollary}
	Let $\fH$ be a finite-dimensional Hilbert space, $\scrH$ be the Hilbert space induced by a Pythagorean pair of operators $(A,B)$ acting on $\fH$ and $p$ be a periodic ray. Then the range of the projection $\rho_p$ acting on $\scrH$ is contained in the copy of $\fH$ inside $\scrH$. Equivalently, $\Ran(\rho_p) = \fX^p$ for all periodic rays $p$.
\end{corollary}

\begin{proof}
Let $c$ be the prime word such that $p = c^\infty$ and $\nu$ be the vertex such that the path from $\varnothing$ to $\nu$ is given by $c$. Suppose $\fX^p := \Ran(\rho_p) \cap \fH$ is a proper subspace of $\Ran(\rho_p)$ and we shall aim to arrive at a contradiction.
	
Since $\fX^p$ is closed it has a non-trivial orthogonal complement inside $\Ran(\rho_p)$ containing a unit vector $z$.
By Theorem \ref{complete scrH decomposition theorem}, $\scrH = \scrH_\fU \oplus \scrH_\fV$ and Corollary \ref{zero projection corollary} implies that $\rho_p$ is the zero operator when restricted to $\scrH_\fU$. Therefore, up to considering $\fV$ rather than $\fH$ we may assume $\fU = \{0\}, \fH = \fV,$ and $\scrH = \scrH_\fV$.
	
Fix $0 < \epsilon < 1$. By density there exists a unit vector $x = [t_n, \ti\xi] \in \scrK$ such that $\norm{z - x} < \epsilon$. 

\textbf{Claim 1.} There exists a finite set of rays $\{p^{(i)}\}_{i \in I}$ and vectors $\{z_i\}_{i \in I}$ such that $p^{(i)} \neq p^{(j)}$ if $i \neq j$, $z_i \in \Ran(\rho_{p^{(i)}})$ and $x = \sum_{i\in I}z_i$.

We have $x = [t_n, \ti\xi] = \sum_{\nu \in \Leaf(t_n)} \tau^*_\nu(\ti\xi_\nu)$ where recall $\ti\xi_\nu \in \fV$ is the component of $\ti\xi$ corresponding to the leaf $\nu$. 
By item \ref{item fV alt def} in Proposition \ref{alt subspace presentation prop} each of the vectors $\ti \xi_\nu$ which are non-zero can be written as a finite sum 
$$\sum_{j\in I_\nu} \tau^*_{\omega_j}(\xi_j)$$ where each $\xi_j \in \fH$ is contained in some ray and $\{\omega_j\}_{j \in I_\nu}$ is a set of vertices so that $\omega_j$ is disjoint from $\omega_i$ if $j\neq i$ (recall definition of disjoint vertices from Subsection \ref{subsec:F-def}). 
Then from the above we can conclude
\[x = \sum_{\nu \in \Leaf(t_n)} \sum_{j\in I_\nu} \tau^*_\nu(\tau^*_{\omega_j}(\xi_j)) = \sum_{i \in I} \tau^*_{\mu_i}(\xi_i)\]
where 
$$I = \bigcup_{\nu\in\Leaf(t_n)} I_\nu \text { and } \{\mu_i\}_{i \in I} = \{ \nu \cdot \omega_j:\ \nu\in\Leaf(t_n) , j\in I_\nu\}$$
noting that if $(\nu,j)\neq (\nu',j')$, then $\nu\cdot \omega_j$ is disjoint from $\nu'\cdot \omega_{j'}.$
The last equality in the expression for $x$ follows because $\tau^*_\nu \circ \tau^*_{\omega_j} = \tau^*_{\nu \cdot \omega_j}$ (after identifying them with binary sequences and $\nu\cdot \omega_j$ with their concatenation). 
Since each $\xi_i$ is contained in a ray, it follows that each term $\tau^*_{\mu_i}(\xi_i) \in \Ran(\rho_{p^{(i)}})$ for some ray $p^{(i)}$.
Further, since $\mu_i$ is disjoint from $\mu_j$ then $p^{(i)} \neq p^{(j)}$ for $i \neq j$.
Therefore, we have proved the claim after taking $z_i = \tau^*_{\mu_i}(\xi_i)$. 

\textbf{Claim 2.} Using the above notation, we have $\tau_{\mu_i}(x) = \xi_i$ for all $i \in I$ and there exists a unique $m \in I$ such that $p^{(m)} = p$.

From the proof of Claim $1$, the vertex $\mu_i$ is disjoint from $\mu_j$ if $i \neq j$. This implies that $\tau_{\mu_i}(\tau^*_{\mu_j}(\xi_j)) = 0$ if $i \neq j$. The first part of the claim then immediately follows from Claim $1$.

Suppose none of the rays in the set $\{p^{(i)}\}_{i\in I}$ are equal to $p$. Then, since $\Ran(\rho_q), \Ran(\rho_{q'})$ are orthogonal subspaces for two non-identical rays $q,q'$, that would imply $x$ is orthogonal to $z$. However, this would give $\norm{z-x} = \sqrt{2}$ which contradicts the assumption $\norm{z - x}  < \epsilon < 1$. Hence, this shows there exists $m \in I$ such that $p^{(m)} = p$ and uniqueness is clear. This proves the claims.

Denote $\mu := \mu_m$, $\xi := \xi_m$ and let $k$ be the length of $\mu$.

\textbf{Claim 3.} There exists $\zeta \in \fX^p$ such that $\tau_\mu(\zeta) = \xi$.

By Claim $2$, $\tau^*_\mu(\xi)$ is contained in $p$. This gives that $\xi$ is contained in the ray ${}_kp$. Since $p$ is periodic, by considering representatives of $\tau^*_\mu(\xi)$, we can assume without loss of generality that $\mu = \nu^j$ for some $j \in \N^*$ and thus $\xi$ is contained in $p$ (hence $\xi \in \fX^p$).
	
Now item \ref{item fX unitary operator} in Proposition \ref{alt subspace presentation prop} states that $V := \phi(p_n)\restriction_{\fX^p}$ is a unitary operator on $\fX^p$ where $n$ is the length of $\nu$. 
Observe that the partial isometry $\tau_\mu$ stabilises $\fX^p$ and moreover its restriction $\tau_\mu\restriction_{\fX^p}$ coincides with $V^j$ since $\mu=\nu^j$ is the concatenation of $j$ times the sequence associated to the vertex $\nu$.
Therefore, it follows by invertibility of $V$ that there exists $\zeta \in \fX^p$ such that $\tau_\mu(\zeta) = V^j(\zeta) = \xi$ which proves the claim. 

From the above claims we can now arrive at a contradiction. Indeed, observe:
\[\norm{z - \zeta} = \norm{\tau_{\mu}(z) - \tau_\mu(\zeta)} = \norm{\tau_{\mu}(z) - \xi} = \norm{\tau_{\mu}(z) - \tau_{\mu}(x)} \leq \norm{z - x} < \epsilon < 1\]
where the first equality follows because $z, \zeta \in \Ran(\rho_p)$ and $\mu \in \Ver_p$.
However, this contradicts the fact that $z$ is a unit vector orthogonal to $\zeta \in \fX^p$. Therefore, it follows no such vector $z$ can exist and subsequently $\Ran(\rho_p) = \fX^p$ which completes the proof.
\end{proof}

\begin{remark} \label{proj range in fH remark}
	\begin{enumerate}[i]
		\item Continue to let $\fH$ be finite-dimensional and $q$ be a ray that is not eventually periodic. Then the above corollary is vacuously true because it can be shown from item ii in Remark \ref{subspace fH orthogonal remark} that $\rho_q$ is simply the zero operator.
		In addition, if $q$ is eventually periodic but not periodic, then we only have $\fX^q \subset \Ran(\rho_q)$ and equality may not hold. However, let $p$ be a (not necessarily unique) periodic ray belonging to the same equivalence class as $q$. Since $p$ is periodic, there exists $k$ such that ${}_kq = p$. Now observe $z \in \Ran(\rho_p) = \fX^p$ if and only if $\tau^*_{q, k}(z) \in \Ran(\rho_q)$. Therefore, this shows that $\Ran(\rho_q) = \tau^*_{q,k}(\fX^p)$.
		\item The above corollary does not necessarily hold when $\fH$ is infinite-dimensional. Obviously, it will not be true when $\fW$ is non-trivial because there are vectors in the range of the projections that are limit points of $\scrK$ in $\scrH$. However, the corollary may not even hold when $\fH = \fV$. As a simple example to demonstrate this fact, let $\fH = \ell^2(\N)$, $A$ be the isometric right shift operator and $B = 0$. The operators $(A,B)$ is a Pythagorean pair acting on $\fH$ producing a Jones' representation acting on a Hilbert space $\scrH$. Then $\tau^*_0(\delta_0) \in \scrH$ is in the range of $\rho_p$ where $p$ is the (periodic) ray going down the left side of $t_\infty$; however, $\tau^*_0(\delta_0) \notin \fH$.
	\end{enumerate}
\end{remark}

\subsection{Examples of decomposition of Pythagorean representations} \label{examples of pythag rep subsection}

To conclude this section, we provide some instructive examples of the decomposition from Theorem \ref{complete scrH decomposition theorem} applied to particular Pythagorean pairs of operators. We begin with an easy example where $\fH$ is finite-dimensional and thus $\fW$ is trivial.

\begin{example} \label{pythag ex 1}
	Let $\fH = \C^4$ and define a Pythagorean pair $(A,B)$ acting on $\C^4$ by identifying $A$ and $B$ with the below matrices
	\[A = 
	\begin{pmatrix}
		1 & 0 & 0 & \frac{1}{\sqrt{2}} \\
		0 & 0 & \varphi & 0 \\
		0 & 0 & 0 & 0 \\
		0 & 0 & 0 & 0
	\end{pmatrix}
	, \quad B = 
	\begin{pmatrix}
		0 & 0 & 0 & 0 \\
		0 & 0 & 0 & \frac{1}{\sqrt{2}} \\
		0 & \psi & 0 & 0 \\
		0 & 0 & 0 & 0
	\end{pmatrix}
	\]
	where $\varphi, \psi \in \bS$ satisfy $\varphi \psi = 1$. Let $e_i$ be the canonical basis for $\C^4$ for $i = 1,\dots, 4$. It is easy to see that $e_1$ is contained in the ray $\ell := \dots aaa$ (this is the ray going down the left side of $t_\infty$) while $e_2, e_3$ are \textit{eventually contained} in the ray $q:= \dots aba$ (this is the zig-zag ray which begins with a left turn). Since $Ae_4 = 1/\sqrt{2}e_1 \in \fV^\ell$ and $Be_4 = 1/\sqrt{2}e_2 \in \fV^q$ we have $e_4 \in \fV$ and $e_4$ can be considered to be contained in the ray $\ell \cdot a = \ell$ and in the ray $q\cdot b$. Therefore, it follows $\fU = \{0\}$ and we obtain the following decomposition: 
	\[\fV^\ell = \C e_1,\quad \fV^q = \C e_2 \oplus \C e_3,\quad \fH = \fV = \fV^\ell \oplus \fV^q \oplus \C e_4.\]
	Notice in this case $\fV$ is strictly larger than $\oplus_{p \in \cP} \fV^p$. It is easy to verify that $\fV^\ell, \fV^q$ have no proper $\phi$-invariant subspaces; however, $\C e_4$ is not $\phi$-invariant. Therefore $\fH$ is not completely reducible under $\phi$. Though, by item i in Proposition \ref{pythag invariant decomp prop}, we can ``discard'' the subspace $\C e_4$ by instead considering the restriction of $A,B$ on $\fH' := \fV^\ell \oplus \fV^q$ where now $\fH'$ is completely reducible under the monoid action while preserving the equivalence class of the induced Pythagorean representation $\sigma:F\act\scrH$.
	
	The above decomposition in turn induces the below decomposition of $\scrH$ and $\sigma$:
	\[\scrH = \scrH_{\fV^\ell} \oplus \scrH_{\fV^q}, \quad \sigma \cong \sigma_\ell \oplus \sigma_q.\]
	Next we shall examine the sub-representations $\sigma_\ell, \sigma_q$. Starting with $\sigma_\ell$, it is easy to observe $\sigma_\ell$ is equivalent to $(\sigma, \scrH) := (\sigma_{1,0}, \scrH_{1,0})$ which is the Pythagorean representation induced by the pair of constant operators $(1,0)$ acting on $\C$. 
	For the element $[e, 1]$ (where $e$ denotes the trivial tree), it is easy to verify that $\sigma(g)[e, 1] = [e,1]$ for all $g \in F$ and thus $[e,1]$ induces a trivial sub-representation of $\sigma$. Next consider the element $z = \tau^*_1(1)$ (recall vertices in $\Ver$ are identified with finite binary sequences in $0,1$ that are read from left to right) which belongs in $\Ran(\rho_{\ell\cdot b})$. Representatives of $z$ are shown below.
	
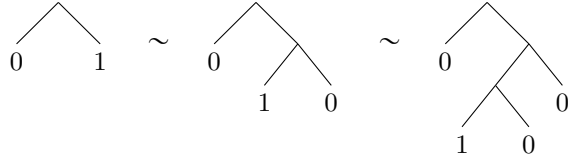
\begin{center}
		\begin{tikzpicture}[baseline=0cm, scale = 1.1]
			\draw (0,0)--(-.5, -.5);
			\draw (0,0)--(.5, -.5);
			
			\node[label={[yshift=-18pt] \footnotesize $0$}] at 	(-.5, -.5) {};
			\node[label={[yshift=-18pt] \footnotesize $1$}] at 	(.5, -.5) {};
			
			\node at (1.2, -.5) {$\sim$};
		\end{tikzpicture}%
		\hspace*{.5em}%
		\begin{tikzpicture}[baseline=0cm, scale = 1.1]
			\draw (0,0)--(-.5, -.5);
			\draw (0,0)--(.5, -.5);
			\draw (.5, -.5)--(.1, -1);
			\draw (.5, -.5)--(.9, -1);
			
			\node[label={[yshift=-18pt] \footnotesize $0$}] at (-.5, -.5) {};
			\node[label={[yshift=-18pt] \footnotesize $1$}] at (.1, -1) {};
			\node[label={[yshift=-18pt] \footnotesize $0$}] at (.9, -1) {};
			
			\node at (1.6, -.5) {$\sim$};
		\end{tikzpicture}%
		\hspace*{.5em}%
		\begin{tikzpicture}[baseline=0cm, scale = 1.1]
			\draw (0,0)--(-.5, -.5);
			\draw (0,0)--(.5, -.5);
			\draw (.5, -.5)--(.1, -1);
			\draw (.5, -.5)--(.9, -1);
			\draw(.1, -1)--(-.3, -1.5);
			\draw(.1, -1)--(.5, -1.5);
			
			\node[label={[yshift=-18pt] \footnotesize $0$}] at (-.5, -.5) {};
			\node[label={[yshift=-18pt] \footnotesize $1$}] at (-.3, -1.5) {};
			\node[label={[yshift=-18pt] \footnotesize $0$}] at (.5, -1.5) {};
			\node[label={[yshift=-18pt] \footnotesize $0$}] at (.9, -1) {};
		\end{tikzpicture}%
	\end{center}
	We wish to compute the associated diagonal matrix coefficient $\phi_z:g\in F\mapsto \langle \sigma(g)z,z\rangle$. Observe that $\tau_{1a}(z) = 1$ for all sequence of zeroes $a \in \{0\}^{\N^*}$ and moreover $\tau_{1b}(z) = 0$ for all sequences $b \in \{0, 1\}^{\N^*}$ containing at least one $1$. By comparing the above statement with the description of the parabolic subgroup $F_{\ell \cdot b}$ using Equation \ref{parabolic subgroup condition eqn} it follows $\sigma(g)z = z$ if and only if $g \in F_{\ell \cdot b}$. Hence, we deduce $\phi_z(g)=1$ if $g\in F_{\ell \cdot b}$ and $0$ otherwise. 
	Therefore, the matrix coefficient associated to $z$ is simply the trivial extension of the matrix coefficient of the trivial representation $1_{F/F_{\ell \cdot b}}$. Hence, by a classical result we can surmise that the sub-representation of $\sigma$ generated by $z$ is equivalent to the quasi-regular representation $\lambda_{F/F_{\ell \cdot b}}$ which is irreducible by Lemma \ref{monomial parabolic irrep lemma}. 
	
	Finally, we claim that $\{[e, 1], z\}$ generates the sub-representation $\sigma_\ell$. Indeed, first note $\tau^*_{0^n}(1) \sim [e, 1]$ for $n \in \N^*$. Then for a vertex $\nu$ in the centre or right side of $t_\infty$ (see Subsection \ref{subsec:F-def} for definition), there exists $g \in F$ such that $\nu \cdot 0$ and $10$ are corresponding vertices (since both vertices $\nu \cdot 0$ and $10$ lie in the centre of $t_\infty$, see Notation \ref{match vertices}). Applying Equation \ref{action rearrange equation} we obtain
	\[\sigma(g)z = \tau^*_{\nu \cdot 0}\tau_{10}(z) = \tau^*_{\nu \cdot 0}(1) = \tau^*_\nu(1).\]
	Since $\{\tau^*_\nu(1)\}_{\nu \in \Ver}$ spans the entire space $\scrH$, it follows $\{[e, 1], z\}$ generates the representation $\sigma_p$. Therefore we can conclude $\sigma_\ell$ is a direct sum of two irreducible sub-representations:
	\[\sigma_\ell \cong 1_F \oplus \lambda_{F/F_{\ell \cdot b}}.\]  
	
	Unsurprisingly, the sub-representation $\sigma_q$ follows a similar analysis; however, the analysis is somewhat simplified because in this case $\sigma_q$ is irreducible. Take $z = [e, e_3]$ and once again let $\phi_z$ denote the diagonal matrix coefficient associated to $z$. Recall $\Ver_q$ denotes the set of vertices $q$ passes through and let $\Ver_{q_e}, \Ver_{q_o} \subset \Ver_q$ be the subset of vertices whose binary sequence ends in $0$ and $1$, respectively. It is then easy to see that
	\begin{align} \label{anti-diagonal tau equation}
		\tau_\nu([e, e_3]) = 
		\begin{cases*}
			0, \quad & $\nu \notin \Ver_q$,\\
			\varphi [e, e_2], \quad & $\nu \in \Ver_{q_e}$,\\
			[e, e_3], \quad & $ \in \Ver_{q_o}$.\\
		\end{cases*}
	\end{align}
	Then in a similar fashion to the above case, by comparing the preceding equation with the description of the parabolic subgroup $F_{q}$ using Equation \ref{parabolic subgroup condition eqn} it follows $\phi_z(g)=1$ if $g\in F_{q}$ and $0$ otherwise. Therefore, the sub-representation of $\sigma_q$ generated by $z$ is equivalent to the quasi-regular representation $\lambda_{F/F_{q}}$ which is again irreducible by Lemma \ref{monomial parabolic irrep lemma}. Further, $\lambda_{F/F_{\ell \cdot b}} \not \cong \lambda_{F/F_{q}}$ by the Mackey-Shoda Criterion (Theorem \ref{Mackey-Shoda criteria}) because $\ell \cdot b$ and $q$ are in different $F$-orbits. 
	
	Lastly, all that remains to be shown is that $z$ is a cyclic vector. Let $\nu$ be any vertex in $t_\infty$. Then by Lemma \ref{match vertices} there exists $g \in F$ such that $\nu\cdot 01$ and $01$ are corresponding vertices. Applying Equation \ref{action rearrange equation} we obtain
	\[\sigma_q(g)z = \tau^*_{\nu \cdot 01}\tau_{01}(z) = \tau^*_{\nu \cdot 01}(e_3) = \tau^*_\nu(e_3).\]
	Then as in the above case we can conclude $\sigma_q \cong \lambda_{F/F_{q}}$.
	
	Combining the above results, we have a complete decomposition of $\sigma$ into irreducible sub-representations:
	\[\sigma \cong 1_F \oplus \lambda_{F/F_{\ell \cdot b}} \oplus \lambda_{F/F_{q}}.\]
	Interestingly, each of the subspaces $\fV^\ell, \fV^q$ induces non-equivalent monomial representation. We shall see in Section \ref{rep from rays section} that this holds more generally. Further, $\fV^\ell$ also induces a trivial representation. Informally, the reason why $\fV^\ell$ induces an additional trivial sub-representation is because the ``support'' of the element $[e,1] \in \scrH_{\fV^\ell}$ is contained in the left side of $t_\infty$. Since every element in $F$ fixes the point $0$ it follows that $[e,1]$ generates the trivial sub-representation. 
\end{example}

Next we shall study a more complicated example where $\fH$ is infinite-dimensional which presents a more involved decomposition of $\fH$ and $\scrH$.

\begin{example}\label{pythag decomp example}
	We begin by defining a partition of $\Ver$:
	\begin{align*}
		\cC_1' &= \{0^m : m \in \N\},\ \cC_1'' = \{0^m10^n1\cdot \omega : m \in \N^* ,n \in \N, \omega \in \Ver\}, \\
		\cC_2 &= \{0^m10^n : m \in \N^*, n \in \N\},\ \cC_3 = \{1 \cdot \omega : \omega \in \Ver\}.
	\end{align*}

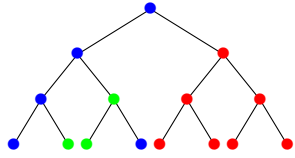
	\begin{center}
		\begin{tikzpicture}[baseline=0cm, scale = 1.2]
			\draw (0,0)--(-.8, -.5);
			\draw (0,0)--(.8, -.5);
			\draw (-.8, -.5)--(-1.2, -1);
			\draw (-.8, -.5)--(-.4, -1);
			\draw (.8, -.5)--(.4, -1);
			\draw (.8, -.5)--(1.2, -1);
			\draw (-1.2, -1)--(-1.5, -1.5);
			\draw (-1.2, -1)--(-.9, -1.5);
			\draw (-.4, -1)--(-.7, -1.5);
			\draw (-.4, -1)--(-.1, -1.5);
			\draw (.4, -1)--(.1, -1.5);
			\draw (.4, -1)--(.7, -1.5);
			\draw (1.2, -1)--(.9, -1.5);
			\draw (1.2, -1)--(1.5, -1.5);

			\node at (0,0) {$\bluebullet$};
			\node at (-.8, -.5) {$\bluebullet$};
			\node at (.8, -.5) {$\redbullet$};
			\node at (-1.2, -1) {$\bluebullet$};
			\node at (-.4, -1) {$\greenbullet$};
			\node at (.4, -1) {$\redbullet$};
			\node at (1.2, -1) {$\redbullet$};			
			\node at (-1.5, -1.5) {$\bluebullet$};
			\node at (-.9, -1.5) {$\greenbullet$};
			\node at (-.7, -1.5) {$\greenbullet$};
			\node at (-.1, -1.5) {$\bluebullet$};
			\node at (.1, -1.5) {$\redbullet$};
			\node at (.7, -1.5) {$\redbullet$};
			\node at (.9, -1.5) {$\redbullet$};
			\node at (1.5, -1.5) {$\redbullet$};
		\end{tikzpicture}%
	\end{center}

	As an example, the tree $t_3$ is drawn above where the blues vertices are in $\cC_1 := \cC_1' \cup \cC_1''$, the green vertices are in $\cC_2$ and the red vertices are in $\cC_3$. \\
	Take $\fH := \ell^2(\Ver)$ and let $\{\delta_\nu : \nu \in \Ver\}$ be the usual orthonormal basis of $\fH$. Choose any sequence $(c_n : n \geq 1)$ of numbers in $(0,1)$ such that their product converges to $1/2$. Define the operators $A,B \in B(\fH)$ by setting
	\begin{equation*}
		A(\delta_\nu) = 
		\begin{cases*}
			\frac{1}{\sqrt{2}} \delta_{\nu 0}, &$\quad \nu \in \cC_1$ \\
			\delta_{\nu 0}, &$\quad \nu \in \cC_2$ \\
			c_k \delta_{\nu 0}, &$\quad \nu \in \cC_3$
		\end{cases*}
		,\quad
		B(\delta_\nu)
		\begin{cases*}
			\frac{1}{\sqrt{2}} \delta_{\nu 1}, &$\quad \nu \in \cC_1$ \\
			0, &$\quad \nu \in \cC_2$ \\
			\sqrt{1-c_k^2} \delta_{\nu 1}, &$\quad \nu \in \cC_3$
		\end{cases*}
	\end{equation*} 
	where $k = \length(\nu)$. It is easy to verify that $(A,B)$ forms a Pythagorean pair. \\
	We shall now decompose $\fH$ into the form described in Theorem \ref{complete scrH decomposition theorem}. If $\nu \in \cC_1''$ then $\norm{q_n\delta_\nu} = 2^{-n/2}$ for all rays $q$ and $n \in \N$. Hence it follows $\delta_\nu \in \fU$. If $\nu \in \cC_2$ then $\delta_{\nu}$ is contained in the ray $\ell := \dots aaa$ and thus $\delta_\nu \in \fV^\ell$. It is clear that if $\nu \in \cC_1' \cup \cC_3$ then $\delta_\nu$ is not contained in a finite set of rays and thus $\delta_\nu \notin \fV$. Further, if $\nu \in \cC_1'$ then $\lim_{n\to\infty}\norm{\ell'_n\delta_\nu} = 1/\sqrt{2}$ where $\ell'$ is the ray given by $\ell \cdot b$. If $\nu \in \cC_3$ then $\lim_{n\to\infty}\norm{\ell_n\delta_\nu} > 0$. This shows that $\delta_\nu \notin \fU$ for $\nu \in \cC_1' \cup \cC_3$. Hence we have
	\[\fU = \cspan\{\delta_\nu : \nu \in \cC_1''\}, \quad \fV = \fV^\ell = \textrm{span}\{\delta_\omega : \omega \in \cC_2\}.\]
	Note, we take the closed linear span for $\fU$ but only take the linear span for $\fV$ because $\fU$ is a closed subspace while $\fV$ is only closed under finite sums.
	Next, we consider $\fW$ which is a subspace of $(\fU \oplus \fV)^\perp = \cspan\{\delta_\nu : \nu \in \cC_1' \cup \cC_3\}$. Let $\nu \in \cC_3$. It is clear that $\delta_\nu$ is partially contained in the ray $\ell$. Similarly, for any vertex $\omega \in \Ver$ the vector $\tau_{\omega \cdot 1}(\delta_\nu) = \lambda \delta_{\nu \cdot \omega \cdot 1}$ is partially contained in the ray $\ell$ where $\lambda$ is some non-zero scalar. Thus $\delta_\nu$ is partially contained in the rays $\ell \cdot w$ where $w$ ranges across all words in $\Mon(a,b)$.  As well, it clear $\tau_\omega(\delta_\nu) \in \fV^\perp$ for all $\omega \in \Ver$ and $\bigcup_{w \in \Mon(a,b)} \Ver_{\ell\cdot w} = \Ver$. This shows that $\delta_\nu \in \fW$ if $\nu \in \cC_3$. If $\nu \in \cC'_1$ then $\delta_\nu \notin \fW$ because $B\delta_\nu \in \fV$ and $\fW$ is a $\phi$-invariant subspace. Therefore, we have 
	\[\fW = \cspan\{\delta_\nu : \nu \in \cC_3\}.\]
	However, note that the above subspaces does not form a complete decomposition $\fH$ since
	\[\fZ := (\fU \oplus \fV \oplus \fW)^\perp = \cspan\{\delta_\nu : \nu \in \cC_1'\}.\]
	Therefore, similarly to Example \ref{pythag ex 1}, $\fH$ is again not completely reducible under the monoid action $\phi$ and $\fZ$ is the ``residual'' subspace which is not $\phi$-invariant. This can be easily verified as $B\delta_\varnothing  = 1/\sqrt{2}\delta_1 \in \fW$ and $B\delta_{0^m} = 1/\sqrt{2}\delta_{0^m1} \in \fV$ for $m \in \N^*$.
	\tb{However, in contrast, the subspace $\fZ$ cannot be immediately discarded because it is infinite-dimensional. Though, we shall show that this is still the case further down below.}

	Furthermore, unlike Example \ref{pythag ex 1}, even after discarding the subspace $\fZ$ we are not able to assume that the monoid action is completely reducible. This is because $\fW$ cannot be decomposed into $\phi$-irreducible components. Indeed, if $\fW'$ is a $\phi$-invariant subspace of $\fW$ then define $\fW''$ to be the closure of $\phi(\Mon(a,b)\backslash \{e\})(\fW')$ which is a $\phi$-invariant closed subspace of $\fW'$ where $e$ is the identity element in $\Mon(a,b)$. Each vector in $\fW'$ can be uniquely expressed in the form $\sum_k \alpha_k \delta_{\nu_k}$ where $\alpha_k$ are non-zero scalars and $\nu_k$ are vertices in $\cC_3$. For each vertex $\nu \in \cC_3$ denote $P_\nu \in B(\fW)$ to be the projection onto $\C \delta_\nu$. Let $V$ be the subset of $\cC_3$ such that $\nu \in V$ if and only if $\fW' \not \subset \ker(P_\nu)$. Set $\nu$ to be a vertex in $V$ with the shortest length (such a vertex must exist but may not be unique) and let $w \in \fW'$ be a non-zero vector which does not belong in the kernel of $P_\nu$. Note for all $v \in \Mon(a,b)$, $\phi(v)\delta_\nu = \lambda\delta_\omega$ where $\lambda$ is a non-zero scalar and $\omega$ is some vertex with length equal to $\length(\nu) + \length(v)$. Therefore, since the vertex $\nu$ has the shortest length in $V$, it follows that $P_\nu(\phi(v)w) = 0$ and $\norm{w - \phi(v)w} \geq \norm{P_\nu w} > 0$ for all $v \in \Mon(a,b)\backslash \{e\}$. Thus $\fW''$ is a proper closed $\phi$-invariant subspace of $\fW'$ from which we can conclude there are no subspaces of $\fW$ which are irreducible under $\phi$. This shows that item \ref{item fH completely reducible} in Proposition \ref{pythag invariant decomp prop} cannot be extended for infinite-dimensional Hilbert spaces.

To better understand the vectors in $\fZ$ we list below some representatives of $[e, \delta_0]$.
	\begin{center}
		\begin{tikzpicture}[baseline=0cm, scale = 1.1]
			\node[label={[yshift=-20pt] \footnotesize $\delta_{0}$}] at 	(0,-.3) {};
			\node at (0.7, -.5) {$\sim$};
		\end{tikzpicture}%
		\hspace*{.5em}%
		\begin{tikzpicture}[baseline=0cm, scale = 1.1]
			\draw (0,0)--(-.5, -.5);
			\draw (0,0)--(.5, -.5);
			
			\node[label={[xshift=-5pt, yshift=-25pt] \footnotesize $\frac{1}{\sqrt{2}}\delta_{00}$}] at 	(-.5, -.5) {};
			\node[label={[yshift=-25pt] \footnotesize $\frac{1}{\sqrt{2}}\delta_{01}$}] at 	(.5, -.5) {};
			
			\node at (1.2, -.5) {$\sim$};
		\end{tikzpicture}%
		\hspace*{.5em}%
		\begin{tikzpicture}[baseline=0cm, scale = 1.1]
			\draw (0,0)--(-.5, -.5);
			\draw (0,0)--(.5, -.5);
			\draw (-.5, -.5)--(-.9, -1);
			\draw (-.5, -.5)--(-.1, -1);
			
			\node[label={[xshift = -5pt, yshift=-25pt] \footnotesize $\frac{1}{2}\delta_{000}$}] at (-.9, -1) {};
			\node[label={[xshift = 3pt, yshift=-25pt] \footnotesize $\frac{1}{2}\delta_{001}$}] at (-.1, -1) {};
			\node[label={[yshift=-25pt] \footnotesize $\frac{1}{\sqrt{2}}\delta_{01}$}] at (.5, -.5) {};
			\node at (1.5, -.5) {$\sim$};
		\end{tikzpicture}%
		\hspace*{.5em}%
		\begin{tikzpicture}[baseline=0cm, scale = 1.1]
			\draw (0,0)--(-.5, -.5);
			\draw (0,0)--(.5, -.5);
			\draw (-.5, -.5)--(-.9, -1);
			\draw (-.5, -.5)--(-.1, -1);
			\draw (-.9, -1)--(-1.3, -1.5);
			\draw (-.9, -1)--(-.5, -1.5);
			
			\node[label={[xshift = -12pt, yshift=-25pt] \footnotesize $\frac{1}{2\sqrt{2}}\delta_{0000}$}] at (-1.3, -1.5) {};
			\node[label={[xshift = 5pt, yshift=-25pt] \footnotesize $\frac{1}{2\sqrt{2}}\delta_{0001}$}] at (-.5, -1.5) {};
			\node[label={[xshift = 3pt, yshift=-25pt] \footnotesize $\frac{1}{2}\delta_{001}$}] at (-.1, -1) {};
			\node[label={[yshift=-25pt] \footnotesize $\frac{1}{\sqrt{2}}\delta_{01}$}] at (.5, -.5) {};
		\end{tikzpicture}%
	\end{center}
	It can been seen that all the components of the representatives of $[e, \delta_0]$ are vectors in $\fV$ except for the first component which tends to $0$ as the length of the first vertex increases. Hence, even though $\delta_0 \in \fZ$, its representatives can almost be described by vectors in $\fV$ and $\delta_0$ can be considered to be almost contained in a finite set of rays. In fact $\delta_0 \in \scrH_\fV$ and more generally, $\delta_{0^m} \in \scrH_\fV$ for $m \in \N^*$ by norm closure. In a similar fashion, it can be observed that $\delta_\varnothing \in \scrH_\fV \oplus \scrH_\fW$.
	\tb{Therefore, $\fZ \subset \scrH_\fV \oplus \scrH_\fW$ and thus we have the complete decomposition of $\scrH$ and $\sigma$ given by $\fU$, $\fV$, $\fW$.}
	
We will not study the sub-representations $\sigma_\ell, \sigma_\fW$ induced by the above decomposition since the infinite-dimensional case requires a more careful treatment which we will instead study in more detail in Section \ref{rep from rays section}.
\end{example}	

Given the decomposition $\scrH_\fV = \oplus_{p \in \cP} \scrH_{\fV^p}$ it is natural to wonder if the subspaces $\fW^p$ can be analogously defined so that $\scrH_\fW = \oplus_{p \in \cP} \scrH_{\fW^p}$. Indeed $\fW^p$ can be defined in a similar fashion to $\fW$ as in Definition \ref{subspace of fH defintiion} but now restricting $\xi$ to be only partially contained in rays belonging to the class $[p]$. This will continue to give a $\phi$-invariant subspace and it can be shown that items \ref{item fW alt def 1}, \ref{item fW alt def 2} in Proposition \ref{alt subspace presentation prop} continue to hold after making the obvious adjustments.
However, in general $\scrH_\fW \neq \oplus_{p \in \cP} \scrH_{\fW^p}$ as we illustrate below.

\begin{example} \label{pythag decomp ex3}
	Again take $\fH = \ell^2(\Ver)$ and let $(c_n : n \geq 1)$ be the same sequence from Example \ref{pythag decomp example}. Call a vertex even (resp. odd) if its associated binary sequence ends with $0$ (resp. $1$). The vertex $\varnothing$ is neither even or odd. Define the operators $A,B \in B(\fH)$ by setting for $\nu \in \Ver \backslash \{\varnothing\}$:
		\begin{equation*}
		A(\delta_\nu) = 
		\begin{cases*}
			c_k\delta_{\nu \cdot 0}, &$\quad \nu = \omega \cdot 0^k$ \\
			\sqrt{1-c_k^2} \delta_{\nu \cdot 0}, &$\quad \nu = \mu \cdot 1^k$
		\end{cases*}
		,\quad
		B(\delta_\nu)
		\begin{cases*}
			\sqrt{1-c_k^2}  \delta_{\nu \cdot 1}, &$\quad \nu = \omega \cdot 0^k$ \\
			c_k\delta_{\nu \cdot 1}, &$\quad \nu = \mu \cdot 1^k$
		\end{cases*}
	\end{equation*} 
	where $k \in \N^*$, $\omega$ is not an even vertex and $\mu$ is not an odd vertex. Set $A(\delta_\varnothing) = 1/\sqrt{2}\delta_0$ and $B(\delta_\varnothing) = 1/\sqrt{2}\delta_1$. It is easy to verify that $(A,B)$ forms a Pythagorean pair.
	
	Consider the vector $\delta_\nu$ and denote the rays $\ell := \dots aaa$, $r := \dots bbb$. It is clear that $\delta_\nu$ is partially contained in the rays $\ell, r$. In addition, $\delta_\nu$ is partially contained in the ray $\ell \cdot v$ where $v \in \Mon(a,b)$ is a word that ends with $b$ and $\delta_\nu$ is also partially contained in the ray $r \cdot w$ where $w \in \Mon(a,b)$ is a word that ends with $a$. Therefore, $\delta_\nu$ is partially contained in the (countably infinite) set of rays which induces a dyadic rational and these rays belong in the equivalence classes $[\ell], [r]$. Note, every dyadic rational in $(0,1)$ is induced by two different rays. For example, the rays $r \cdot 0$ and $\ell \cdot 1$ both induce the dyadic rational $1/2$.  
	
	The above shows that every vector $\delta_\nu$ belongs in the subspace $\fW$ and thus $\fH = \fW$ while $\fU = \fV = \{0\}$. Further, it shows every vector in $\fW$ is partially contained in rays belonging to two different equivalence classes. Thus, we have $\fW^\ell = \fW^r = \{0\}$ and clearly $\scrH = \scrH_\fW \neq \scrH_{\fW^\ell}\oplus \scrH_{\fW^r}$. Hence, this provides an example of the situation described in the discussion preceding this example. 
	
	As an additional note, consider the element $z := \rho_{\ell \cdot b}(\delta_\varnothing)$ which is a limit point of $\scrK$ inside $\scrH$. It can be observed that $\scrX := \cspan(\sigma(F)z)$ consists of vectors which are a sum of vectors in $\Ran(\rho_p)$ where $p$ is a ray in the same equivalence class as $\ell$. Further, if $g,h \in F$ such that $g'(r) \neq h'(r)$ then $\sigma(g)z, \sigma(h)z$ are orthogonal. Therefore, from the above we can deduce that the subspace $\scrX$ generated by $z$ in an infinite-dimensional subspace such that $\scrK \cap \scrX = \{0\}$.
\end{example}



\section{Classification of Atomic Pythagorean Representations} \label{rep from rays section}

In Section \ref{decompose pythag rep section} we provided a general decomposition of Pythagorean representations into diffuse and \at{} parts. 
The goal of this section is to describe and classify the atomic part of a Pythagorean representations.
For pedagogical reasons we will separately consider the finite and infinite-dimensional cases and the decompositions of $\fV$ and $\fW$.

\subsection{Family of one-dimensional representations of $F$, $F_p$ and $\widehat{F_p}$.}

For a ray $p$, define the family $\{\chi_\varphi^p\}_{\varphi \in \bS}$ of one-dimensional representations of $F_p:=\{g\in F:\ g(p)=p\}$ given by
		\[\chi_\varphi^p(g) = \varphi^{\log_2 g'(p)} \textrm{ for all $g \in F_p$}\]
		where recall the definition of the derivative from Subsection \ref{subsec:F-def}.

We consider the following class of monomial representations of $F$:
\[\{\Ind_{F_p}^F \chi_\varphi^p: \varphi \in \bS \textrm{ and } p \textrm{ is a ray}\}.\]
Each of the representations are irreducible by Lemma \ref{monomial parabolic irrep lemma}. 
Note that when $p\neq \ell,r$, then the above representations are infinite-dimensional since $F_p\subset F$ has infinite index and otherwise $\Ind_{F_p}^F \chi_\varphi^p= \chi_\varphi^p$ is one-dimensional. 
Moreover, we have that if $p$ is an eventually periodic ray with period of length $d$, then an element of $F_p$ has a derivative at $p$ equal to a power of $2^d$. 
Therefore, $\chi_\varphi^p$ only depends on the ray $p$ and $\varphi^d$; thus we may choose to have that the principal argument $\textrm{Arg}(\varphi)$ of $\varphi$ belongs to $[0,2\pi/d)$.
Using the Mackey-Shoda criteria we deduce that when $p\neq \ell,r$, then $\Ind_{F_p}^F \chi_\varphi^p$ only depends on the {\it class} of $p$ (i.e.~the ray $p$ up to finite prefixes) and $\varphi^d$.
When $p$ is not eventually periodic then $g\in F_p$ has necessarily slope 1 at $p$ implying that $\chi_\varphi^p=1_{F_p}$ does not depend on $\varphi$ and thus $\Ind_{F_p}^F=\lambda_{F/F_p}$ is the quasi-regular representation.

Consider the subgroup
$$\widehat{F_p}:=\{g \in F : g(p) = p, g'(p) = 1\}$$ and the associated quasi-regular representation $\lambda_{F/\widehat{F_p}}$ which will appear in the decomposition of atomic representations.
When $p$ is eventually periodic this representation is reducible since $\widehat F_p\subset F$ is not self-commensurated as explained in Subsection \ref{general monomial rep section}.
Although, these quasi-regular representations are still easy to classify: they only depends on the class of $p$ when $p\neq\ell,r$.

\begin{definition} \label{monomial rep in Pythag rep def}
	Consider the following sets of monomial representations which form the building blocks of all \at{} Pythagorean representations.
	\begin{align*}
		R_{fin, 1} &= \{ \chi_\varphi^\ell \oplus \Ind_{F_{p \cdot b}}^F\chi_\varphi^{p \cdot b}:\ [p] = [\ell],\ \varphi \in \bS\} \\ 
		&\cup \{ \chi_\varphi^r \oplus \Ind_{F_{p \cdot a}}^F\chi_\varphi^{p \cdot a}:\ [p] = [r],\ \varphi \in \bS, \varphi \neq 1 \},\\
		R_{fin, d} &= \{ \Ind_{F_p}^F\chi_\varphi^p:\ p \text{ periodic with period length } d,\ \varphi \in \bS, 0 \leq \textrm{Arg}(\varphi) < 2\pi/d \}, d\geq 2\\
		R_{inf, 1} &= \{ \lambda_{F/\widehat{F_\ell}} \oplus \lambda_{F/\widehat{F_{p \cdot b}}}:\ [p] = [\ell],\ \varphi \in \bS \} \\
		&\cup \{ \lambda_{F/\widehat{F_r}} \oplus \lambda_{F/\widehat{F_{p \cdot a}}}:\ [p] = [r],\ \varphi \in \bS, \varphi \neq 1 \},\\
		R_{inf, d} &= \{ \lambda_{F/\widehat{F_p}}:\ p \text{ periodic with period length } d,\ \varphi \in \bS, 0 \leq \textrm{Arg}(\varphi) < 2\pi/d \}, d\geq 2\\
		R_{inf, \infty}&= \{ \lambda_{F/F_p}:\ p \text{ not eventually periodic} \}.
	\end{align*}
\end{definition}

\begin{remark} \label{monomial rep in Pythag rep remark}
	\begin{enumerate}[i]
		\item All the representations in the above sets are pair-wise non-equivalent. Further, the representations in the sets $\cup_{d > 1}R_{fin, d}, R_{inf, \infty}$ are irreducible.
		\item The subscript \textit{fin} in the above notations stands for how the representations in those sets appear as sub-representations in $\sigma_\fV$ when $\fV$ is finite-dimensional (note $\fW$ is never finite-dimensional). Similarly, the subscript \textit{inf} stands for how the representations in those sets appear as sub-representations in $\sigma_\fV \oplus \sigma_\fW$ \textit{only if} $\fV \oplus \fW$ is infinite-dimensional. Proving the above will be the focus of the remainder of the section.
	\end{enumerate}
\end{remark}

\subsection{Decomposition of $\sigma_\fV$ for finite-dimensional $\fH$} \label{sigma_p decomp finite subsection}

Before we proceed, we remind the readers that importantly even when $\fH$ is finite-dimensional, the larger Hilbert space $\scrH$ will always be infinite-dimensional as explained in Remark \ref{scrH dimension remark}.
Further, recall the sub-representation $\sigma_\fV$ defined in Proposition \ref{subrep of sigma from fH proposition}.

\begin{theorem} \label{sigma_p decompose finite theorem}
	Let $(A,B)$ be a Pythagorean pair over a {\bf finite-dimensional} Hilbert space $\fH$ and consider the sub-representation $(\sigma_\fV, \scrH_{\fV})$ of $(\sigma_{A,B}, \scrH_{A,B})$.
	\begin{enumerate}
		\item The representation $\sigma_\fV$ decomposes as a finite direct sum of representations appearing in $R_{fin,d}$ for $1\leq d\leq n$. Moreover, if $\sigma_{\fV}\cong \sum_{d,j} \pi_{d,j}$ with $\pi_{d,j}\in R_{fin,d}$ where $1\leq d\leq n$ and $j$ is in some index set $J_d$, then $\sum_{d,j} d  \leq \dim(\fV).$
		\item Conversely, a representation of $R_{fin,d}$ is an \at{} Pythagorean representation for any natural number $d\geq 1$.
	\end{enumerate}
\end{theorem}

\begin{proof}
	Consider $A,B,\fH,\sigma_\fV,\scrH_{\fV}$ as above and assume $\fV\neq\{0\}$, this later case being trivial. By Theorem \ref{complete scrH decomposition theorem}, $\sigma_\fV$ is a finite direct sum of representations of the form $\sigma_p$ where $p$ is a ray. 
	In addition, $\oplus_{p \in \cP} \fV^p \subset \fV$ and thus $\sum_{p \in \cP} \dim(\fV^p) \leq \dim(\fV)$.
	Subsequently, it is suffice to the prove the theorem for $\sigma_p$ where $\fV^p$ is non-trivial. Further, by Remark \ref{subspace fH orthogonal remark}, $p$ must be eventually periodic with length of period smaller or equal to $\dim(\fV^p)$. Since the representation $(\sigma_p, \scrH_{\fV^p})$ only depends on the equivalence class of the ray $p$, we can assume $p$ is periodic. Hence, there is a prime word $c$ with length $d \leq \dim(\fV^p)$ such that $p = c^\infty$ and define $\nu \in \Ver$ to be the vertex such that the path from $\varnothing$ to $\nu$ is given by $c$ (here we have identified $p$ as a sequence in $a,b$ that is read from right to left). Observe that $\nu$ lies in the ray $p$ and 
	$$\tau_{\nu^n}(p_{dm}\xi) = p_{d(m+n)}\xi \text{ for all } m,n \in \N, \xi \in \fH.$$
	
	Recall from Definition \ref{subspace of fH defintiion} that the subspace $\fX^p \subset \fV^p$ denotes the space of all vectors that are contained in the ray $p$ (and is also equal to $\Ran(\rho_p)$ by Corollary \ref{range rho inside fH corollary}). Define $E \in B(\fX^p)$ to be the restriction of the operator $\phi(p_d)$ on $\fX^p$. 
	It is the operator obtained from a period $c$ of the ray $p$. 
	Item \ref{item fX unitary operator} in Proposition \ref{alt subspace presentation prop} shows that this restriction is well-defined and $E$ is a unitary operator of $\fX^p$.
	
	By elementary spectral theory $E$ admits an orthonormal basis of eigenvectors $\{\xi_j\}_{j=1}^k$ for $\fX^p$ with corresponding unit eigenvalues $\{\gamma_j\}_{j=1}^k$ where $k = \dim(\fX^p)$. 
	This decomposition of $\fX^p$ into eigenspaces is the starting point of our proof that we can now outline.
	
	\begin{enumerate}
		\item Show that the set of eigenvectors $\{\xi_j\}_{j=1}^k$ induces orthogonal sub-representations of $\sigma_p$ (here we are identifying $\xi_j$ with its copy $[e, \xi_j]$ inside $\scrH_{\fV^p}$);
		\item Show that the above sub-representations completely decompose $\sigma_p$;
		\item Prove that the diagonal matrix coefficient $F \ni g \mapsto \langle \sigma_p(g)\xi_j, \xi_j \rangle$ associated to each $\xi_j$ is the trivial extension of a diagonal matrix coefficient of a representation of $F_p$.
	\end{enumerate}
			
	In the interest of clarity, we shall divide the main arguments of the proof into separate claims. 
	
	\textbf{Claim 1:} Each of the eigenvectors $\{\xi_j\}_{j=1}^k$ generate sub-representations of $\sigma_p$ which are orthogonal to each other.
	
	Consider two distinct eigenvectors $\xi_i, \xi_j$. We are required to show the closed linear spans of $\sigma_p(F)\xi_i, \sigma_p(F)\xi_j$ are orthogonal. We shall achieve this by performing a series of simplifications. Firstly, by linearity and continuity of the inner-product it is suffice to show $\langle \sigma_p(g)\xi_i, \sigma_p(g')\xi_j \rangle = 0$ for all $g, g' \in F$. By applying Equation \ref{action rearrange equation} and because $\xi_i, \xi_j$ are contained in the ray $p$ this is equivalent to showing $\langle \tau_\mu(\xi_i), \tau_\omega(\xi_j) \rangle = 0$ for all $\mu, \omega \in \Ver$. From Observation \ref{contained vector norm obs}, we only need to consider vertices $\mu, \omega \in \Ver_p$ which reduces the above inner-product to $\langle p_m\xi_i, p_n\xi_j \rangle$ for some $m,n \in \N^*$. Notice if $m-n \notin d\N$ (where recall $d$ is the length of a period $c$ of the ray $p$), then $p_m\xi_i, p_n\xi_j$ are contained in different rays and thus orthogonal by item \ref{item fX orthogonal} in Proposition \ref{subspace fH orthogonal prop}. Hence we only need to consider the case when $p_m = p_{r}E^u, p_n = p_{r}E^v$ for some $0 \leq r < d$ and $u,v\in \N$. Then we can conclude the claim by noting
	\[\langle p_{r}E^u\xi_i, p_{r}E^v\xi_j \rangle = \langle p_{d-r}p_{r}E^u\xi_i, p_{d-r}p_{r}E^v\xi_j \rangle = \langle E^{u+1}\xi_i, E^{v+1}\xi_j \rangle = \langle \gamma_i^{u+1}\xi_i, \gamma_j^{v+1}\xi_j \rangle = 0.\]

The above claim shows that each eigenvector $\xi_j$ generates $d$ orthogonal vectors $p_0\xi_j, \dots, p_{d-1}\xi_j$ in $\fV^p$ which are orthogonal to $p_n\xi_i$ for any $i \neq j$. This implies that $kd \leq \dim(\fV^p)$.

For the remainder of the proof we will consider two separate cases: $p$ is eventually straight or not.
Example \ref{pythag ex 1} illustrates the second case.
	
	\textbf{Claim 2:}
Suppose the ray $p$ is not eventually straight. Then the eigenvectors $\{\xi_j\}_{j=1}^k$ of $E$ generate the representation $(\sigma_p, \scrH_{\fV^p})$.
	
	Denote $\scrX$ to be the closed liner span of $\{\sigma_p(F)\xi_j\}_{j=1}^k = \{\sigma(F)\xi_j\}_{j=1}^k$.
	We are required to show that $\scrX = \scrH_{\fV^p}$. It is obvious that $\scrX \subset \scrH_{\fV^p}$. Hence, we only need to prove the reverse inclusion.
	
	To prove the reverse inclusion, we shall again perform a series of reductions to simplify the problem. 
	Observe $\scrH_{\fV^p}$ is equal to the closed linear span of $$\{\tau^*_\mu(\eta) : \eta \in \fV^p, \mu \in \Ver\}.$$ 
	Thus, it is sufficient to show that $\tau^*_\mu(\eta) \in \scrX$ for all $\eta \in \fV^p$, $\mu \in \Ver$. Fix a $\eta \in \fV^p$ and $\mu \in \Ver$. By item \ref{item fVp alt def 1} in Proposition \ref{alt subspace presentation prop} there exists a finite set of vectors $\{\eta_k\}_{k=1}^m$ which are eventually contained in $p$ and a finite set of vertices $\{\omega_k\}_{k=1}^m$ such that $\eta = \sum_{k=1}^m \tau^*_{\omega_k}(\eta_k)$. This gives	
	\[\tau^*_\mu(\eta) = \sum_{k=1}^m \tau^*_\mu\tau^*_{\omega_k}(\eta_k) = \sum_{k=1}^m \tau^*_{\mu \cdot \omega_k}(\eta_k)\]
	Hence, by linearity we only need to show that each of the terms $\tau^*_{\mu \cdot \omega}(\eta_\omega)$ are in $\scrX$. In particular, this will follow if we can show that $\tau^*_\omega(\zeta) \in \scrX$ for all $\omega \in \Ver$ and $\zeta$ is any vector eventually contained in the ray $p$. This is the statement that we shall prove directly.
	
	Fix a vertex $\omega \in \Ver$ and vector $\zeta$ that is eventually contained in $p$. Since the ray $p$ is periodic, $\zeta$ is contained in a ray $p \cdot w = p \cdot w'$ where $w' = c\cdot w$ and $w$ is some word in $\Mon(a,b)$ (recall $p = c^\infty$). Further, let $\omega' \in \Ver$ to be the child of $\omega$ such that the geodesic path from $\omega$ to $\omega'$ is given by $w'$. Then it follows by considering representatives that $\tau^*_\omega(\zeta) = \tau^*_{\omega'}(w'\zeta)$. Thus, $w'\zeta$ is contained in the ray $p$ and we have $w'\zeta \in \fX^p$. 
	
	Since $\{\xi_j\}_{j=1}^k$ forms an orthogonal basis for $\fX^p$, the vectors $w'\zeta$ and $\tau^*_\omega(\zeta)$ can be expressed in the form
	\[w'\zeta = \sum_{j=1}^k \alpha_j\xi_j, \quad \tau^*_\omega(\zeta) = \sum_{j=1}^k \alpha_j \tau^*_{\omega'}(\xi_j)\]
	where $\alpha_j \in \C$ for $j=1,\dots, k$. Recall each $\xi_j \in \fX^p$ is an eigenvector of $E$ with eigenvalue $\gamma_j \in \bS$ and note $E$ coincides with the restriction of $\tau_\nu$ on $\fX^p$.
	Thus, for each $j = 1, \dots, k$, we have
	\[\xi_j = \rho_{\nu}(\xi_j) = \tau^*_{\nu}(E\xi_j) = \gamma_j\tau^*_{\nu}(\xi_j).\]
	Since the ray $p$ is a not a straight line, the vertices $\omega'$ and $\nu$ lie in the centre of $t_\infty$. Hence, there exists $g := [t,s] \in F$ such that $\omega'$ and $\nu$ are corresponding vertices of $t$ and $s$, respectively. It then follows by Equation \ref{action rearrange equation}:
	\[\sigma(g)(\sum_{j=1}^k \alpha_j \gamma_j^{-1} \xi_j) = \sigma(g)(\sum_{j=1}^k \alpha_j \tau^*_{\nu}(\xi_j)) = \sum_{j=1}^k\alpha_j \tau^*_{\omega'}(\xi_j) = \tau^*_\omega(\zeta).\]
	The above equalities show that $\tau^*_\omega(\zeta) \in \scrX$ from which we can conclude the claim. 
	
	Combining Claims $1$ and $2$ we obtain that the eigenvectors $\{\xi_j\}_{j=1}^k$ completely decompose $\sigma_p$ into orthogonal cyclic sub-representations. The only remaining task for this case is to show that the sub-representation induced by each eigenvector is equivalent to a monomial representation associated to $F_p$.

	To this end, from here on, $\xi$ shall refer to an eigenvector in $\{\xi_j\}_{j=1}^k$ with associated eigenvalue $\gamma$ and denote $\theta_p$ to be the cyclic representation of $F_p$ given by the sub-representation of $\sigma_p\restriction_{F_p}$ (the restriction of $\sigma_p$ to the subgroup $F_p$) generated by $\xi$. Further, for the remainder of the proof set $\varphi$ to be a $d$th root of $\gamma$ (define $\varphi_j$ similarly) and recall $\chi_{\varphi}^p$ is a one-dimensional representation of $F_p$.
	Note, up to unitary equivalence, $\chi_\varphi^p$ does not depend on the choice of $d$th root of $\gamma$.
	
	\textbf{Claim 3:} If the ray $p$ is not eventually straight then the representation $\theta_p$ of $F_p$ generated by $\xi$ is unitarily equivalent to $\chi_{\varphi}^p$.
	
	We shall make use of the description of $F_p$ using tree-diagrams as discussed in Subsection \ref{parabolic desc subsection}.
	Let $g := [t,s] \in F_p$ with corresponding leaves $\omega$ and $\nu^n\omega$ where $\omega \in \Ver_{p}$ and $n \in \N$. First suppose $\omega \in \Leaf(t)$ and $\nu^n\omega \in \Leaf(s)$. We then have by Equation \ref{action rearrange equation}:
	\begin{equation} \label{chi eq1}
		\sigma_p(g)\xi = \tau^*_\omega \tau_{\nu^n\omega}(\xi) = \rho_\omega \tau_{\nu^n}(\xi) = E^n(\xi) = \gamma^{n} \xi.
	\end{equation}
	Alternatively, suppose $\nu^n\omega \in \Leaf(t)$ and $\omega \in \Leaf(s)$. Similarly, we then have
	\begin{equation} \label{chi eq2}
		\sigma_p(g)\xi = \tau^*_{\nu^n\omega}\tau_\omega(\xi) = \tau^*_{\nu^n}(\xi) = \gamma^{-n}\xi.
	\end{equation}
	
	Equations \ref{chi eq1} and \ref{chi eq2} show $\theta_p$ is a one-dimensional representation and the claim follows from these two equations. Indeed, in the case of Equation \ref{chi eq1} we have $g'(p) = 2^{dn}$ and thus
	\[\theta_p(g)\xi = \sigma_p(g)\xi = \gamma^n\xi = \varphi^{\log_2(2^{dn})}\xi = \varphi^{\log_2g'(p)}\xi = \chi_{\varphi}^p(g)\xi.\]
	In the case of Equation \ref{chi eq2} we have $g'(p) = 2^{-dn}$ and thus
	\[\theta_p(g)\xi = \sigma_p(g)\xi = \gamma^{-n}\xi = \varphi^{\log_2(2^{-dn})}\xi = \varphi^{\log_2g'(p)}\xi = \chi_{\varphi}^p(g)\xi.\]	 
	This proves the claim since $\xi$ is a cyclic vector for $\theta_p$.
	
	Now consider $g := [t,s] \in F$ and set $\Leaf(t) = \{\nu_i\}_{i \in I}$ and $\Leaf(s) = \{\omega_i\}_{i \in I}$. Let $\nu_k, \omega_l$ be the leaves of $t$ and $s$, respectively, which the ray $p$ passes through. Denote $\phi_\xi : F \ni g \mapsto \langle \sigma_p(g)\xi, \xi\rangle$ to be the diagonal matrix coefficient associated to the vector $\xi$. 
	
	\textbf{Claim 4:} If $g \notin F_p$, then $\phi_\xi(g) = 0$.
	
	By the discussion in Subsection \ref{parabolic desc subsection}, $g \notin F_p$ if and only if $k \neq l$ (equivalently, $\nu_k$ and $\omega_l$ are not corresponding leaves of $g$) or $k = l$ and Equation \ref{parabolic subgroup condition eqn} is not satisfied. First suppose $k \neq l$. 	
	Then by Equation \ref{action rearrange equation}:
	\begin{align*}
		\phi_\xi(g) &= \langle \sigma_p(g)\xi, \xi \rangle 
		= \sum_{i \in I} \langle \tau_{\omega_i}(\xi), \tau_{\nu_i}(\xi) \rangle \\
		&= \sum_{i \neq k,l} \langle \tau_{\omega_i}(\xi), \tau_{\nu_i}(\xi) \rangle + \langle \tau_{\omega_k}(\xi), \tau_{\nu_k}(\xi) \rangle + \langle \tau_{\omega_l}(\xi), \tau_{\nu_l}(\xi) \rangle.  
	\end{align*}
	From Observation \ref{contained vector norm obs}, $\tau_{\nu_i}(\xi) = 0$ for $i \neq k$ since $\nu_i \notin \Ver_{p}$. Similarly $\tau_{\omega_j}(\xi) = 0$ for $j \neq l$. Thus, each of the terms in the above equation is $0$ and we obtain $\phi_\xi(g) = 0$. 
	
	Then suppose $k = l$ and Equation \ref{parabolic subgroup condition eqn} is not satisfied for $m = \length(\nu_k)$ and $n = \length(\omega_l)$. Subsequently, $\nu_k, \omega_l$ are corresponding leaves and $m-n \notin d\N$. Then by a similar reasoning as before we have
	\begin{equation*}
		\phi_\xi(g) = \sum_{i \neq k} \langle \tau_{\omega_i}(\xi), \tau_{\nu_i}(\xi) \rangle + \langle \tau_{\nu_k}(\xi), \tau_{\omega_k}(\xi) \rangle = \langle p_n\xi, p_m\xi \rangle.
	\end{equation*}
	Since $m-n \notin d\N$ and the length of a period of $p$ is $d$ it follows that $p_n\xi, p_m\xi$ are vectors contained in different rays. Therefore $\langle p_n\xi, p_m\xi \rangle = 0$ by item \ref{item fX orthogonal} in Proposition \ref{subspace fH orthogonal prop} and thus $\phi_\xi(g) = 0$ which proves the claim.
		
	\textbf{Claim 5:} If the ray $p$ is not eventually straight then the sub-representation of $\sigma_p$ generated by $\xi$ is unitarily equivalent to $\Ind_{F_p}^F \chi_{\varphi}^p$.
	
	Let $\phi'_\xi$ be the diagonal matrix coefficient of $\chi_{\varphi}^p$ associated with $\xi$. 
	From Claim $3$, $\chi_\varphi^p \cong \theta_p$ which gives $\phi_\xi(g) = \phi'_\xi(g)$ for all $g \in F_p$. Hence, Claim $4$ shows that $\phi_\xi$ is the trivial extension of $\phi'_\xi$. Therefore, the cyclic sub-representation of $\sigma_p$ generated by $\xi$ is equivalent to the induced representation $\Ind_{F_p}^F \chi_{\varphi}^p$ which proves the claim.

	From the above claims we have shown that if $p$ is not eventually straight then
	\[\sigma_p \cong \oplus_{j=1}^k \Ind_{F_p}^F \chi_{\varphi_j}^p.\]
	As observed earlier we can assume that $0 \leq \textrm{Arg}(\varphi_j) < 2\pi/d$. Hence,  as required, each of the sub-representations $\Ind_{F_p}^F \chi_{\varphi_j}^p$ belongs in the set $R_{fin, d}$ and we have $\sum_{d,j} d = kd \leq \dim(\fV^p)$ where the index set $J_d$ is $\{1, \dots, k\}$.
	
	{\bf We shall now treat the case when $p$ is eventually straight.} There are precisely two equivalence classes of rays which are eventually straight; one which consists of all rays in the form $\ell \cdot w$ and another one which consist of all rays in the form $r \cdot w$ where $w$ is any word in $\Mon(a,b)$. Thus, without loss of generality, we shall assume $p$ is either equal to the periodic ray $\ell$ or $r$.
	In this case the analysis will be similar to the initial case; however, now the eigenvectors $\{\xi_j\}_{j=1}^k$ (thus the subspace $\fX^p = \Ran(\rho_p)$) will no longer generate the entire representation $\sigma_p$ as an additional one-dimensional representation will also need to be considered in the decomposition of $\sigma_p$ (compare this to the analysis of $\sigma_\ell$ in Example \ref{pythag ex 1}).
	
	\textbf{Claim 6.}
	If $p = \ell$ (resp. $p = r$) then the elements $\{\xi_j, \tau^*_1(\xi_j)\}_{j=1}^k$ (resp. $\{\xi_j, \tau^*_0(\xi_j)\}_{j=1}^k$) generate sub-representations of $\sigma_p$ which are orthogonal to each other.
	
	First let $p = \ell$. 
	From Claim $1$ we know the sub-representations generated by $\{\xi_j\}_{j=1}^k$ are orthogonal and slightly modifying the same proof shows that the sub-representations generated by $\{\tau^*_1(\xi_j)\}_{j=1}^k$ are also orthogonal to each other. Thus, we only need to show that the sub-representations generated by $\xi_i$ and $\tau^*_1(\xi_j)$ are orthogonal for $1 \leq i,j \leq k$.	
	It is then sufficient to show $\langle \sigma_p(g)\xi_i, \tau^*_1(\xi_j) \rangle = 0$ for all $g \in F$. This is indeed true because $\sigma_p(g)\xi_i = \gamma_i^n\xi_i$ for some $n \in \Z$ and belongs in $\Ran(\rho_\ell)$ while $\tau^*_1(\xi_j)$ belongs in $\Ran(\rho_{\ell \cdot b})$. Then we can conclude by the proof of item \ref{item fX orthogonal} in Proposition \ref{subspace fH orthogonal prop}.

	The case when $p = r$ follows the same proof after interchanging $a$ with $b$ and $0$ with $1$. This completes the proof of the claim.
		
	When $p$ is not eventually straight, Claim $3$ shows that the space $\fX^p$ (which is equal to $\Ran(\rho_p)$) generates the representation $\sigma_p$.
	When $p$ is a straight line then this is no longer true. Rather, we shall now show that 
	when $p = \ell$ (resp. $p = r$) then the two subspaces $\fX^p$ and $\tau^*_1(\fX^p) = \Ran(\rho_{p\cdot b})$ (resp. $\fX^p$ and $\tau^*_0(\fX^p) = \Ran(\rho_{p \cdot a})$) generate $\sigma_p$.
	Note, the equality between the subspaces holds from item i in Remark \ref{proj range in fH remark}.
	
	\textbf{Claim 7.}
	If $p = \ell$ (resp. $p = r$) then the elements $\{\xi_j, \tau^*_1(\xi_j)\}_{j=1}^k$ (resp. $\{\xi_j, \tau^*_0(\xi_j)\}_{j=1}^k$) generate the representation $(\sigma_p, \scrH_{\fV^p})$.
	
	First suppose $p = \ell$.  In this case $c = a, d=1$, $\nu = 0$ and $E = A\restriction_{\fX^p}$.
	Denote $\scrX$ to be the closed liner span of $\{\sigma(F)\xi_j, \sigma(F)\tau^*_1(\xi_j)\}_{j=1}^k$. By performing the same series of reductions as done in the proof of Claim $2$, it is suffice to show that if $\omega$ is any vertex in $\Ver \backslash \{\varnothing\}$ and $\zeta$ is any vector that is eventually contained in the ray $p$ then $\tau^*_\omega(\zeta) \in \scrX$. Let $w$ be a word in $\Mon(a,b)$ such that $\zeta$ is contained in the ray $p \cdot w$ and as before set $w' = c \cdot w$.
	
	Suppose $\zeta \notin \fX^p$. Then the word $w'$ must contain the letter $b$. As in Claim $2$, define $\omega' \in \Ver$ to be the child of $\omega$ such that the geodesic path from $\omega$ to $\omega'$ is given by $w'$. Since $w'$ contains the letters $a$ and $b$, the vertex $\omega'$ must lie in the centre of $t_\infty$. Further, $\tau^*_\omega(\zeta) = \tau^*_{\omega'}(w'\zeta)$ where $w'\zeta \in \fX^p$. Then since $\{\xi_j\}_{j =1}^k$ forms an orthonormal basis for $\fX^p$ we have
	\[w'\zeta = \sum_{j=1}^k \alpha_j\xi_j, \quad \tau^*_\omega(\zeta) = \sum_{j=1}^k \alpha_j \tau^*_{\omega'}(\xi_j)\]
	where $\alpha_j \in \C$ for $j=1,\dots, k$. As well we have
	\[\tau^*_1(\xi_j) = \gamma_j\tau^*_{10}(\xi_j).\]
	Observe that the vertices $\omega'$ and $10$ lie in the centre of $t_\infty$. Then by continuing this argument in the same manner as in Claim $2$, but now replacing $\xi_j$ with $\tau^*_1(\xi_j)$, we can conclude that $\tau^*_\omega(\zeta) \in \scrX$ as required.
	
	Now suppose $\zeta \in \fX^p$. If $\omega$ does not lie on the left side of $t_\infty$ then this will imply that $\omega'$ lies in the centre of $t_\infty$ and subsequently we can follow the same argument presented in the above paragraph. Hence, suppose $\omega$ lies on the left side of $t_\infty$. In this case, it turns out $\tau^*_\omega(\zeta)$ does not belong to the subspace generated by $\{\sigma(F)\tau^*_1(\xi_j)\}_{j=1}^k$ but instead belongs to the subspace generated by $\{\sigma(F)\xi_j\}_{j=1}^k$. We shall show this in the following. 
	
	Since $\zeta \in \fX^p$, in a similar fashion as above we can write
	\[\zeta = \sum_{j=1}^{k} \alpha_j'\xi_j, \quad \tau^*_\omega(\zeta) = \sum_{j=1}^{k} \alpha_j' \tau^*_\omega(\xi_j).\]
	Since $\omega, 0$ lie on the left side of $t_\infty$, there exist a $g := [t, s] \in F$ such that $\omega, 0$ are corresponding vertices of $t,s$, respectively. We then have
	\[\sigma(g)(\sum_{j=1}^{k} \alpha_j' \gamma_j^{-1}\xi_j) = \sigma(g)(\sum_{j=1}^{k} \alpha_j' \tau^*_0(\xi_j)) = \sum_{j=1}^{k} \alpha_j' \tau^*_\omega(\xi_j) = \tau^*_\omega(\zeta).\] 
	This proves the required result for when $p = \ell$.
	
	When $p = r$ the proof follows similarly by interchanging left with right and $0$ with $1$. This finishes the proof of the claim.
	
	The only remaining task for this case is to show that the sub-representation generated by each of the elements in the above claim are equivalent to one of the representations specified in the theorem. 

	Once again, let $\xi$ refer to an eigenvector in $\{\xi_j\}_{j=1}^k$ with associated eigenvalue $\gamma$. Now, denote $\theta_p$ to be the cyclic representation of $F_p$ given by the sub-representation of $\sigma_p\restriction_{F_p}$ generated by $\tau^*_1(\xi)$ (resp. $\tau^*_0(\xi)$) if $p = \ell$ (resp. $p = r$). Further, recall $\chi_{\gamma}^p$ is a one-dimensional representation of $F$.
	
	\textbf{Claim 8:}
	If $p = \ell$ (resp. $p = r$) then the sub-representation $\sigma_p$ generated by $\tau^*_1(\xi)$ (resp. $\tau^*_0(\xi)$) is unitarily equivalent to $\Ind_{F_{p}}^F\chi_\gamma^p$. Further, the sub-representation of $\sigma_p$ generated by $\xi$ is unitarily equivalent to $\chi_\gamma^p$. 
	
	Suppose $p = \ell$. We shall first show that the sub-representation of $\sigma_p$ generated by $\xi$ is equivalent to the representation $\chi_{\gamma}$ of $F$. 
	
	Let $g := [t, s] \in F$ and let $\mu, \omega$ be the first vertex of $t,s$ with lengths $m,n$, respectively. Then by Equation \ref{action rearrange equation} we obtain
	\[\sigma(g)\xi = \tau^*_\mu\tau_\omega(\xi) = \gamma^{n-m}\xi = \gamma^{\log_2g'(\ell)}\xi = \chi_{\gamma}^p(g)\xi.\]
	This proves the claim for $\xi$. 
	
	Next, we shall show that the representation $\theta_p$ of $F_p$ is equivalent to $\chi_\gamma^p$. The proof follows a similar calculation for when $p$ is not eventually periodic. Indeed, Equations \ref{chi eq1} and \ref{chi eq2} from the proof of Claim $3$ continue to hold after replacing $\omega$, $\nu^n\omega$, $\xi$ with $1\cdot \omega$, $1 \cdot \nu^n\omega$, $\tau^*_1(\xi)$, respectively. Subsequently, the equalities at the end of the proof of Claim $3$ remain true after replacing $\xi$ with $\tau^*_1(\xi)$ and it follows that $\theta_p \cong \chi_\gamma^p$. 

	Now let $\phi_\xi$ be the diagonal matrix coefficient of $\sigma_p$ associated to $\tau^*_1(\xi)$. By following a similar reasoning as in Claim $4$ we can show that $\phi_\xi(g) = 0$ if $g \notin F_p$. Then applying the proof of Claim $5$ to $\phi_\xi$ proves that $\phi_\xi$ is the trivial extension of the matrix coefficient of a cyclic vector associated to the representation $\chi_\gamma^p$ of $F_p$. Therefore, the sub-representation of $\sigma_p$ generated by $\tau^*_1(\xi)$ is equivalent to $\Ind_{F_{p \cdot b}}^F \chi_\gamma^{p \cdot b}$. This proves the claim for when $p = \ell$.	
	
	The remaining case is when $p = r$. This however follows from the above case when $p = \ell$ by interchanging left with right, $1$ with $0$, $b$ with $a$ and by instead taking $\mu, \omega$ to be the last vertex of $t,s$, respectively. This completes the proof of the claim.
	
	From the above claims we have shown that if $p$ is eventually straight then:
	\[\sigma_p \cong \oplus_{i=1}^k \chi_{\gamma_i}^q \bigoplus \oplus_{j=1}^k \Ind_{F_{p \cdot x}}^F \chi_{\gamma_j}^{p \cdot x}\]
	where $q = \ell, x = b$ if $[p] = [\ell]$ and $q = r, x = a$ if $[p] = [r]$. In particular, the above representation is contained in the set $R_{fin, 1}$ and again have $\sum_{1,j} 1 = k \leq \dim(\fV^k)$.\\
	Thus, we have shown that $\sigma_\fV$ can be decomposed into representations arising from the set $\cup_{d=1}^n R_{fin, d}$ which proves the first part of the theorem. Conversely, from the above proof it is instructive that we can construct an \at{} Pythagorean pair such that its associated Pythagorean representation is equivalent to any representation in $\cup_{d \geq 1} R_{fin, d}$.
\end{proof}

\begin{remark}\label{product rep remark}
	\begin{enumerate}[i]
		\item Note, by applying Proposition \ref{pythag invariant decomp prop} we could have additionally assumed that $\fV^p$ is irreducible under the monoid action $\phi$. Under this additional assumption, it is easy to see this would imply that $\fX^p$ is one-dimensional and thus $E$ only has a single eigenvector. This would have led to a shorter proof of the above theorem as Claim $1$ would no longer be required and Claim $2$ could be simplified. However, we have decided to forget this simplification because it does not carry to when $\fH$ is infinite-dimensional and the proof of the decomposition of $\sigma_p$ in this case (Theorem \ref{sigma_p decompose infinite theorem}) will require Claims $1$ and $2$.
		\item Interestingly, the proof in the above theorem shows that the sub-representation $\sigma_p \cong \oplus_{j \in J} \Ind_{F_p}^F \chi_{\varphi_j}^p$ only depends on the eigenvalues of the operator $E$ (for ease of notation, we will only refer to the case when $p$ is not eventually straight; however, the following also applies to when $p$ is eventually straight). This implies that the behaviour of $\sigma_p$ is not determined by the individual operators $\phi(x_i)$ but only on the product of the operators $\phi(p_d) = \phi(x_d) \dots \phi(x_2)\phi(x_1)$ restricted to $\fX^p$ (where we take $x_i$ to be the $i$th letter of $p$ and as before, we assume $p$ is periodic).  
\end{enumerate}	
\end{remark}

\subsection{Decomposition of $\sigma_\fV$ for infinite-dimensional $\fH$}

We now consider the decomposition of $\sigma_\fV$ when $\fH$ is a separable infinite-dimensional Hilbert space. 
As in the finite-dimensional case we can do it ray by ray. 
However, now we may have aperiodic rays and isometries rather than unitaries for the map $E$. Moreover, we will use direct integral rather than direct sums to decompose the actions of $E$.
Finally, note that the range of the projection $\rho_p$ is not necessarily contained inside $\fH$.

\begin{theorem} \label{sigma_p decompose infinite theorem}
	Let $(A,B)$ be a Pythagorean pair over a separable infinite-dimensional Hilbert space $\fH$ and consider the sub-representation $(\sigma_\fV, \scrH_{\fV})$ of $(\sigma_{A,B}, \scrH_{A,B})$. 
	
	\begin{enumerate}
	\item The representation $\sigma_\fV$ can be decomposed as a direct integral of representations belonging to the set $\cup_{d \geq 1} R_{fin, d}$ and a countable direct sum of representations belonging to the set $(\cup_{d \geq 1} R_{inf, d}) \cup R_{inf, \infty}.$
	\item Conversely, any representation as described above is an \at{} Pythagorean representation.
	\end{enumerate}
\end{theorem}

\begin{proof}
	As in the proof of the Theorem \ref{sigma_p decompose finite theorem} it is suffice to only consider the sub-representation $\sigma_p$ where $p$ is a ray and $\fV^p$ is non-trivial.
	We examine separately the aperiodic and eventually periodic case and split the proof in a number of claims.
	
	\textbf{Claim 1.} If the ray $p$ is not eventually periodic and $\xi \in \fX^p$ then the sub-representation $\sigma_p$ generated by $\xi$ is equivalent to $\lambda_{F/F_p}$.
	
	Once again, let $\theta_p$ denote the sub-representation of $\sigma_p \restriction_{F_p}$ generated by $\xi$. 
	From Equation \ref{irrational stab group equation}, if $p$ is not eventually periodic and $g := [t,s] \in F_p$ then the leaves of $t$ and $s$ whose sdi contains $p$ must coincide and thus $g$ fixes a sdi containing $p$.
	From there, it is easy to deduce that $\sigma(g)\xi=\xi$ for all $g\in F_p$ and thus $\theta_p$ is a copy of the trivial representation of $F_p$. Then applying the same arguments used in Claims $4$ and $5$ in the proof of Theorem \ref{sigma_p decompose finite theorem} shows that the sub-representation of $\sigma_p$ generated by $\xi$ is equivalent to $\Ind_{F_p}^F 1_{F_p} = \lambda_{F/F_p}$.
		
	\textbf{Claim 2.} If the ray $p$ is not eventually periodic, then $\sigma_p$ is unitary equivalent to a countable direct sum of $\lambda_{F/F_{p}} \in R_{inf, \infty}$.
	
	Let $\{\xi_k\}_{k \in K_0}$ be an orthonormal basis for $\fX^p$ and let $x_i$ denote the $i$th letter of $p$ (here we identify $p$ with an infinite sequence in $a,b$). The map $\phi(x_1)$ is an isometry from $\fX^p$ to $\fX^{_1p}$ and thus the set of vectors $\{x_1\xi_k\}_{k \in K_0}$ forms an orthogonal set in $\fX^{_1p}$. Let $\{\xi_k\}_{k \in K_1}$ be an orthonormal basis for the orthogonal complement of $x_1(\fX^p)$ inside $\fX^{_1p}$. Iteratively, in a similar fashion define the sets of vectors $\{\xi_k\}_{K_i}$ for $i \in \N$ and take $K = \cup_{i \in \N} K_i$. For $k \in K_i$, define $\eta_k := \tau^*_{p,i}(\xi_k)$ where recall the definition of $\tau^*_{p,i}$ from Notation \ref{tau notation}. Observe $\{\eta_k\}_{k \in K}$ forms an orthogonal set in $\scrH_{\fV^p}$.
	
	By Claim $1$, each of the vectors $\eta_k$ generate a sub-representation $\lambda_{F/F_{p^{(k)}}}$ contained in $\sigma_p$ (where if $k \in K_i$ then $p^{(k)} = {}_ip$ which is not an eventually periodic ray and is equivalent to $p$). 	
	Since the vectors $\{\eta_k\}_{k \in K}$ are pair-wise orthogonal and the restriction of the maps $\phi(x_{i+1})$ on $\fX^{{}_ip}$ are isometries, it is easy to deduce that each of the sub-representations $\lambda_{F/F_{p^{(k)}}}$ are orthogonal to each other.
	
	Next, we wish to show that the vectors $\{\eta_k\}_{k \in K}$ generate $\sigma_p$. Denote $\scrX$ to be the subspace generated by these vectors. We are required to prove $\scrX = \scrH_{\fV^p}$. The forward inclusion is clear, thus we only need to prove the reverse inclusion.
	
	By performing a similar reduction as in Claim $2$ in the proof of Theorem \ref{sigma_p decompose finite theorem} it is sufficient to show that $\tau^*_\omega(\zeta) \in \scrX$ for all $\omega \in \Ver$ and for all vectors $\zeta$ eventually contained in $p$. Fix such a vertex $\omega$ and vector $\zeta$ contained in a ray $q$ belonging to the equivalence class of $p$. There exists $m,n \in \N$ such that $_mp = {}_nq$. Note $_{m+i}p = {}_{n+i}q$ for all $i \in \N$. Hence, since $q$ is not eventually periodic, $m$ and $n$ can be taken large enough such that the word $q_n$ contains at least one $a$ and $b$. Then the vector $q_n\zeta \in \fX^{_mp}$. Define the vertex $\omega'$ to be the child of $\omega$ such that the geodesic path from $\omega$ to $\omega'$ is given by $q_n$. Further, define $\nu_i$ to be the vertex such that the geodesic path from the root node to $\nu_i$ is given by $({}_ip)_{m-i}$ for $0 \leq i \leq m$. From the above construction $\{\tau_{\nu_i}(\eta_k)\}_{0 \leq i \leq m, k \in K_i}$ is an orthonormal basis for $\fX^{_mp}$. Then we can write
	\[\tau^*_\omega(\zeta) = \tau^*_{\omega'}(q_n\zeta) = \sum_{i=0}^m \sum_{k \in K_i} \alpha_{i,k} \tau^*_{\omega'} \tau_{\nu_i}(\eta_{k})\]
	where $\alpha_{i,k} \in \C$. Hence, we only need to show that each term $\tau^*_{\omega'}\tau_{\nu_i}(\eta_k) \in \scrX$. Since $\tau_{\nu_i}(\eta_k)$ is contained in ${}_mp$, observe that for any vertex $\mu$ in the ray ${}_mp$ we have
	\[\tau_{\nu_i}(\eta_k) = \rho_\mu \tau_{\nu_i}(\eta_k) = \tau^*_\mu\tau_\mu \tau_{\nu_i}(\eta_k) = \tau^*_{\mu}\tau_{\nu_i\cdot \mu}(\eta_k).\]
	Since $_mp$ is not eventually periodic, we can take $\mu$ to be in the centre of $t_\infty$. Thus we have
	\[\tau^*_{\omega'}\tau_{\nu_i}(\eta_k) = \tau^*_{\omega'}\tau^*_\mu\tau_{\nu_i\cdot \mu}(\eta_k) = \tau^*_{\omega' \cdot \mu}\tau_{\nu_i\cdot \mu}(\eta_k)\]
	where $\omega' \cdot \mu$ and $\nu_i \cdot \mu$ lie in the centre of $t_\infty$. Hence there exists $g := [t, s] \in F$ such that $\omega' \cdot \mu$ and $\nu_i \cdot \mu$ are corresponding vertices of $t$ and $s$, respectively. Then since 
	\[\tau^*_{\omega' \cdot \mu}\tau_{\nu_i\cdot \mu}(\eta_k) = \sigma_p(g)(\eta_k) \in \scrX\]
	it follows that $\scrH_{\fV^p} \subset \scrX$.
	
	Therefore, we have shown $\sigma_p \cong \oplus_{k \in K} \lambda_{F/F_{p^{(k)}}}$ where each of the rays $p^{(k)}$ belong in the equivalence class $[p]$. This latter property ensures $\lambda_{F/F_{p^{(k)}}} \cong \lambda_{F/F_p}$ which completes the proof of the claim.
	
	Now we consider the case when $p$ is eventually periodic and recall we can further assume $p$ is periodic. There is a prime word $c$ with length $d$ such that $p = c^\infty$. Let $\nu \in \Ver_p$ be the vertex such that the path from $\varnothing$ to $\nu$ is given by $c$. As in the previous proof, we will separately consider when $p$ is a straight line or not a straight line.
	
	\textbf{Claim 3.} 
	If the ray $p$ is periodic and not a straight line then $\sigma_p$ is a direct integral of monomial representations.
	
	Consider the (possibly infinite-dimensional) space $\fX^p$ and define the isometry $E$ to be the restriction of $\phi(p_d)$ on $\fX^p$ as in the proof of Theorem \ref{sigma_p decompose finite theorem}. 
	We now apply the Wold decomposition to the isometry $E$ of $\fX^p$ (see Chapter $1$ in \cite{NFBK}).
	There exist (possibly trivial) Hilbert spaces $\fK$ and $\fX$ such that up to unitary conjugacy $(\fX^p,E)$ is the space
	$$\left(\fK\otimes \ell^2(\N) \right)\oplus \fX$$
	and the operator
	$$\left( \id_\fK\otimes S\right) \oplus U$$
	where $S$ is the unilateral shift operator and $U$ is a unitary operator.
	Then arguing as in the proof of Claims $1$ and $2$, it can be seen $\fK\otimes \ell^2(\N)$ provides the representation $\id_\fK\otimes \lambda_{F/\widehat{F_p}}$
	(recall $\widehat{F_p}$ from Definition \ref{subgroup of Fp definition}). 
	Hence, up to some direct sum of $\lambda_{F/\widehat{F_p}} \in \cup_{d > 1} R_{inf, d}$, we can assume that $E$ is a unitary operator acting on some subspace $\fX \subset \fX^p$.

	By the spectral theorem for unitary operators, the spectrum of $E$ is contained in the unit circle $\bS$ and up to unitary conjugacy $(\fX, E)$ is the direct integral
	\[\int_{\bS}^{\oplus} \fX_z d\mu(z)\]
	and the multiplicative operator
	\[\int_{\bS}^{\oplus} zd\mu(z).\]
	Define $\ti \chi_z^p$ to be a representation of $F_p$ on $\fX_z$ equal to a (necessary countable) direct sum of $\chi_z^p$ where the direct sum is indexed by an orthonormal basis of $\fX_z$. It is easy to observe $\{\ti \chi_z^p\}_{z \in \bS}$ forms a measurable field of representations over $(\bS, \mu)$ because $\ti \chi_z^p$ simply acts by scalar multiplication. Hence we can define
	\[\beta^p := \int_{\bS}^{\oplus} \ti \chi_{z^{1/d}}^p d\mu(z)\]
	to be a representation of $F_p$ on $\fX$ where $d$ is the length of a period of $p$ and $z^{1/d}$ is the $d$-th root of $z$ with minimal principal argument. By a similar calculation as in Claim $3$ in the proof of Theorem \ref{sigma_p decompose finite theorem}, we can show that $\beta^p$ is equivalent to a sub-representation of $\sigma_p\restriction_{F_p}$. Let $g := [t,s] \in F_p$ and let $\nu, \omega \in \Ver_p$ be corresponding vertices of $t,s$ with lengths $dm, dn$, respectively, for some $m,n \in \N^*$. Then for all $\xi \in \fX, z \in \bS$:
	\begin{align*}
		(\beta^p(g)\xi)(z) &:= \tilde \chi_{z^{1/d}}^p(g)(\xi(z)) \\
		&= z^{d^{-1}\log_2(g'(p))}\xi(z) \\
		&= z^{n-m}\xi(z) = (E^{n-m}\xi)(z).
	\end{align*}
	Thus, $\beta^p(g)\xi = E^{n-m}\xi$ which gives:
	\begin{align*}
		\beta^p(g)\xi= E^{n-m}\xi = \tau^*_\nu \tau_\omega(\xi) = \sigma_p(g)\xi	
	\end{align*}
	where note we have identified $\xi$ with its copy $[e, \xi]$ inside $\scrH_{\fV^p}$.
	Then by following a similar argument as in Claims $4$ and $5$ in the proof of Theorem \ref{sigma_p decompose finite theorem} it follows that $\Ind_{F_p}^F \beta^p$ is equivalent to the sub-representation of $\sigma_p$ generated by the copy of $\fX$ inside $\scrH_{\fV^p}$. Therefore we obtain by applying Proposition \ref{direct integral induced prop}:
	\[\int_{\bS}^{\oplus} \Ind_{F_p}^F \ti \chi_{z^{1/d}}^p d\mu(z) \subset \sigma_p \textrm{ where } \Ind_{F_{p}}^F \ti \chi_{z^{1/d}}^p \cong \oplus_{i=1}^{\dim(\fX_z)} \Ind_{F_p}^F \chi_{z^{1/d}}^p.\] 
	Hence, all it remains is to show that the sub-representation of $\sigma_p$ generated by $\fX$ is in fact equal to the entire representation $\sigma_p$. However, this was shown in the proof of Claim $2$ in Theorem \ref{sigma_p decompose finite theorem} and a similar reasoning can be applied to complete the proof of the claim.
	
	The final remaining case is when $p$ is a straight line. Suppose $p = \ell$. Then as in the proofs of Claims $6$ and $7$ in Theorem \ref{sigma_p decompose finite theorem}, we can decompose $\sigma_p$ as a direct sum of the two orthogonal sub-representation generated by $\fX^p$ and $\tau^*_1(\fX^p)$. Adjusting the proofs of Claims $2$ and $3$ accordingly, we can obtain the following decomposition:
	\[\sigma_p \cong \left(\int_{\bS}^{\oplus} \ti \chi_{z}^\ell d\mu'(z)\right) \bigoplus \left(\oplus_{i \in I} \lambda_{F/\widehat{F_\ell}} \right)\bigoplus \left(\int_{\bS}^{\oplus} \Ind_{F_{\ell \cdot b}}^F \ti \chi_{z}^{\ell \cdot b} d\mu(z) \right) \bigoplus \left(\oplus_{j \in J} \lambda_{F/\widehat{F_{\ell \cdot b}}}\right)\]
	where $\ti \chi_z^\ell \cong \oplus_{b\in \mathcal B} \chi_z$ with $\mathcal B$ an orthonormal basis of the Hilbert space $\fX_z$, the first two sub-representations are generated by $\fX^p$, the last two sub-representations are generated by $\tau^*_1(\fX^p)$. The same reasoning yields a similar decomposition for $\sigma_p$ when $p = r$ after replacing $\ell$ with $r$, $1$ with $0$ and $b$ with $a$. This proves the first statement of the theorem.
	
	The above construction shows that the converse also holds.
\end{proof}

\subsection{Decomposition of $\sigma_\fW$}

For the subspace $\fW$ we have that $\Ran(\rho_p) \cap \fW$ is trivial for all rays $p$. 
Rather, it can be seen from item \ref{item fW alt def 1} in Proposition \ref{alt subspace presentation prop} that the space $\Ran(\rho_p)$ instead solely consists of limit points of $\scrK$ in $\scrH$ which arise from vectors in $\fW$. Therefore, we shall find that to decompose $\sigma_\fW$ we must look beyond the space $\fH$ and must instead work in the larger Hilbert space $\scrH$. However, we can still continue to apply the same techniques and strategy developed in previous proofs to decompose $\sigma_\fW$. 
Further, we do not have the advantage of studying a single equivalence class of rays at a time (this is demonstrated in Example \ref{pythag decomp ex3}).
We shall continue to use the notation defined in the previous subsections and recall $\fW$ is only non-trivial when $\fH$ is infinite-dimensional.

\begin{theorem}\label{sigma_fW decompose theorem}
	Let $(A,B)$ be a Pythagorean pair over a separable infinite-dimensional Hilbert space $\fH$ and consider the sub-representation $(\sigma_\fW, \scrH_{\fV})$ of $(\sigma_{A,B}, \scrH_{A,B})$. 
	\begin{enumerate}
		\item The representation $\sigma_\fW$ can be decomposed as a direct integral of representations belonging to the set $\cup_{d \geq 1} R_{fin, d}$ and a countable direct sum of representations belonging to the set $(\cup_{d \geq 1} R_{inf, d}) \cup R_{inf, \infty}.$
		\item Conversely, any representation as described above is a Pythagorean representation.
	\end{enumerate}
\end{theorem}	
	
\begin{proof}
	Consider $A,B,\fH, (\sigma, \scrH) := (\sigma_\fW, \scrH_\fW)$ as above. For a ray $p$, define $\scrX^p := \Ran(\rho_p) \subset \scrH$.
	The subspace $\scrX^p$ will play the role of the subspace $\fX^p = \Ran(\rho_p) \cap \fH$ of vectors contained in $p$ used in the decomposition of $\sigma_p$. Thus, we will see how the techniques applied to the decomposition of $\sigma_p$ can also be applied to the decomposition of $\sigma_\fW$ after making some minor adjustments.

	Note, if $z \in \scrX^p$ for some ray then $\tau^*_\nu \tau_\omega(z) \in \scrX^q$ for some ray $q$ for all $\nu, \omega \in \Ver$. In fact, $\tau^*_\nu \tau_\omega$ is an isomorphism between $\scrX^p$ and $\scrX^q$ if ${}_np = {}_mq$ where $\length(\nu) = m$, $\length(\omega) = n$, $\nu \in \Ver_q$, $\omega \in \Ver_p$. Thus, if $p \sim q$ then $\scrX^p$ is non-trivial if and only if $\scrX^q$ is non-trivial. 
	Define $P$ to be a subset of representatives of rays that are pairwise non-equivalent so that they are either periodic or not eventually periodic, and give rise to non-trivial $\scrX^p$. 
	In addition, if the straight ray $\ell$ is in $P$ then also include $\ell \cdot b$ in $P$. Similarly, if $r$ is in $P$ then include $r \cdot a$ in $P$ (recall $\ell = (\dots aaa)$ and $r = (\dots bbb)$). Observe $\scrX^{\ell \cdot b} = \tau^*_1(\scrX^\ell),$ $\scrX^{r \cdot a} = \tau^*_0(\scrX^r)$ and these two subspaces behave similarly to the vectors $\tau^*_1(\xi_j), \tau^*_0(\xi_j)$ in Claim $6$ in the proof of Theorem \ref{sigma_p decompose finite theorem}.
	
	{\bf First:} we shall show that the sub-representations of $\sigma$ generated by $\scrX^p$ and $\scrX^q$ are orthogonal when $p,q \in P, p \neq q$. 
	
	It is sufficient to show that given $y \in \scrX^p, z \in \scrX^q$ and $g\in F$ we have $\langle \sigma(g)y, z \rangle = 0$. Begin by additionally assuming $p \not \sim q$ (meaning that we exclude the cases where $p$ or $q$ is equal to $\ell$ or $r$). It is easy to see $\sigma(g)y \in \scrX^{p'}$ for some ray $p'$ which belongs in the same equivalence class as $p$ and thus distinct from the ray $q$. Therefore $\sigma(g)y$ is orthogonal to $z$ because $\rho_h$ and $\rho_{h'}$ are orthogonal projections if $h \neq h'$. Next, assume $p = \ell$ (in the case that $\ell \in P$) and relax the assumption $p \not \sim q$. Then $\sigma(g)y$ will also belong in $\scrX^p$ as a consequence of all elements in $F$ fixing the point $0$. Thus, we again have $\langle \sigma(g)y, z \rangle = 0$. The same reasoning also applies when $p = r$. Therefore, $\{\scrX^p\}_{p \in P}$ induces orthogonal sub-representations of $\sigma$.
		
	{\bf Second:} as in the previous proofs, we shall consider different types of rays and determine the representation induced by vectors in $\scrH$ contained in the range of the projections associated to those rays. In that purpose we fix a vector $z\in \scrX^p$ for a given $p \in P$.
		
	Suppose $p \in P$ is not an eventually periodic ray. We have $\langle \tau_{p, m}(z), \tau_{p, n}(z) \rangle = 0$ for $m \neq n$ because $\tau_{p, m}(z) \in \scrX^{_mp}, \tau_{p, n} \in \scrX^{_np}$ and ${}_mp \neq {}_np$. In addition, $\tau_\nu(z) = 0$ for all $\nu \notin \Ver_p$. Then it follows by the same reasoning as in Claim $1$ in the proof of Theorem \ref{sigma_p decompose infinite theorem} that the sub-representation of $\sigma$ generated by $z$ is equivalent to $\lambda_{F/F_p}$.
		
	Suppose now that $p \in P$ is a periodic ray. Define $c, d, \nu$ as in the proof of Theorem \ref{sigma_p decompose finite theorem}. Denote $E$ to be the isometry formed by the restriction of $\tau_\nu$ on $\scrX^p$. This is well-defined because if $z \in \scrH$ is in the range of $\rho_p$ then $\tau_\nu(z)$ is in the range of $\rho_{{}_dp} = \rho_p$. Observe $E$ plays the same role as $E$ in the previous proofs by noting $\phi(p_d)\xi = \tau_\nu(\xi)$. However, now $E$ is defined on a subspace of $\scrH$ rather than $\fX^p \subset \fH$. Regardless, by the same reasoning as in Claim $3$ (and the discussion following) in the proof of Theorem \ref{sigma_p decompose infinite theorem} and using the same notation as before, we have the sub-representation of $\sigma$ generated by $\scrX^p$ consists of a direct sum of quasi-regular representations $\lambda_{F/\widehat{F_{p_m}}}$ in direct sum of a  direct integral:
	$$\begin{cases} \int_{\bS}^{\oplus} \Ind_{F_p}^F \ti \chi_{z}^p d\mu(z), &\text{ if $p$ is not a straight line}\\
	\int_{\bS}^{\oplus} \ti \chi_{z}^p d\mu'(z), &\text{ otherwise} 
	\end{cases}.$$
	
	Taking the direct sum of the above representations as $p$ runs over $P$ gives the decomposition stated in the theorem. Therefore, to prove the first statement in the theorem, all that remains to be shown is that the subspaces $\{\scrX^p\}_{p \in P}$ generates the entire space $\scrH$ (where here we assume that $\fH=\fW$). Define $\scrX := \cspan\{\sigma(F)(\scrX^p)\}_{p \in P} \subset \scrH$. We are required to show $\scrH \subset \scrX$. By density, it is sufficient to show $\scrK \subset \scrX$.
	
	Fix a non-zero vector $z := [t, \xi] \in \scrK$. We have that $z = \sum_{\omega \in \Leaf(t)} \tau^*_{\omega} (\xi_\omega)$ where $\xi_\omega \in \fW$ is the component of $\xi$ corresponding to the leaf $\omega$. By item \ref{item fW alt def 1} in Proposition \ref{alt subspace presentation prop} each $\xi_\omega$ is a countably infinite sum of non-zero orthogonal vectors in $\scrH$ where each vector belongs in the range of a projection associated to some ray. If $x \in \scrH$ is such a vector it is easy to see $\tau^*_\omega(x)$ also belongs in the range of a projection associated to a ray. Therefore, we can write
	\[z = \sum_{i \in \N} x_i\]
	where each $x_i \in \scrH$ is non-zero and belongs in $\Ran(\rho_{p^{(i)}})$ for a ray $p^{(i)}$ and $p^{(i)} \neq p^{(j)}$ if $i \neq j$. If $p^{(i)}$ is a straight line then $p^{(i)} \in P$ and clearly $x_i \in \scrX$. Hence, suppose $p := p^{(i)}$ is not a straight line. Then there is ray $q \in P$ which is not a straight line such that $p \sim q$ and thus ${}_mp = {}_nq$ for $m,n \in \N$ (if $p$ is eventually straight then $q$ is either $\ell \cdot b$ or $r \cdot a$). Let $\mu$ (resp. $\omega$) be the vertex in $p$ (resp. $q$) that has length $m$ (resp. $n$). Since $p,q$ are not straight lines, we can take $m,n$ large enough so that the vertices $\mu, \omega$ lie in the centre of $t_\infty$. Then there exists $g \in F$ such that $\mu, \omega$ are corresponding vertices. It follows $\sigma(g)x_i$ belongs in $\Ran(\rho_q)$ and thus $x_i \in \sigma(F)(\scrX^q) \subset \scrX$.
	
	Define $z_k = \sum_{i=0}^k x_i$ for $k \in \N$ noting that $z_k \in \scrX$. Further,
	\[\norm{z - z_k}^2 = \norm{\sum_{i\geq k+1} x_i}^2 = \sum_{i \geq k+1} \norm{x_i}^2.\]
	Since $\sum_{i \in \N} \norm{x_i}^2 = \norm{z}^2 < \infty$ it follows the above sum converges to $0$ as $n$ tends to infinity which proves that $z \in \scrX$ as $\scrX$ is closed. This concludes the proof of the first part of the theorem and as before, the converse easily follows from the above construction.
\end{proof}

\begin{example}\label{ex:int-atomic}
	Recall, the discussion following Proposition \ref{subrep of sigma from fH proposition} explains how Pythagorean representations are closed under taking direct sums and direct integrals. Further, from the theorems in this section, it is clear that the class of \textit{\at{}} Pythagorean representations are closed under taking countable direct sums. However, interestingly the direct integral of atomic representations may not necessary be \at{}. \\
	Indeed, let $m$ be any diffuse measure of full support on the Cantor space endowed with its standard Borel $\sigma$-algebra. As well, consider any measurable field of \at{} Pythagorean representations $\{(\sigma_p, \scrH_p)\}_{p \in \cC}$ (for example, take $(\sigma_p, \scrH_p) = (\lambda_{F/F_p}, \ell^2(F/F_p))$). Then the direct integral
	\[(\sigma, \scrH) := \left(\int_{p \in \cC}^\oplus \sigma_p dm(p), \int_{p \in \cC}^\oplus \scrH_p dm(p)\right)\]
	is again a Pythagorean representation. Further, it can be verified that
	\[\rho_\nu\left(\int_{p \in \cC} \xi(p)dm(p)\right) = \int_{I_\nu} \xi(p)dm(p)\]
	for all $\nu \in \Ver, \xi \in \scrH$.
	Thus, from the assumption that $m$ is a diffuse measure, we deduce that $q_n(\xi) \rightarrow 0$ as $n \rightarrow \infty$ for all rays $q$ and all vectors $\xi \in \scrH$. Therefore, $\sigma$ is in fact a \textit{diffuse} Pythagorean representation. \\
	Intuitively, the above conclusion can be deduced from the description of vectors in $\fV, \fW$ as being contained in countably many rays as described in Remark \ref{alt subspace presentation remark}. This conclusion is why the statements in Theorems \ref{sigma_p decompose infinite theorem} and \ref{sigma_fW decompose theorem} specify only a \textit{countable direct sum} of representations in $(\cup_{d \geq 1} R_{inf, d}) \cup R_{inf, \infty}$.
\end{example}	
		
\tb{
\begin{remark} \label{rem:sigma_fW-decompose}
	As shown in Remark \ref{rem:fH-decomp}, when $(\fU \oplus \fV \oplus \fW)^\perp \cap \fH$ is infinite-dimensional then we no longer have $\sigma = \sigma_\fU \oplus \sigma_\fV \oplus \sigma_\fW$. However, we can still obtain a decomposition of $\sigma$ into ``diffuse'' and ``atomic'' parts. 
	Indeed, define $\scrH_d \subset \scrH$ to be the subspace containing all vectors $z \in \scrH$ such that $\rho_p(z) = 0$ for all rays $p$, and define $\scrH_a \subset \scrH$ be the subspace containing all vectors $z \in \scrH$ such there exists a countable set of rays $\{p^{(i)}\}_{i\in I}$ satisfying $z = \sum_{i \in I} \rho_{p^{(i)}}(z)$ (this definition is analogous to the definition in \cite{DHJ-atomic15}, see the end of Section \ref{sec:preliminaries} for a comparison between this study and the current study). Observe that $\scrH_\fU \subset \scrH_d$ (resp. $\scrH_{\fV \oplus \fW} \subset \scrH_a$) and the two subspaces coincide when $\fZ$ is finite-dimensional. Furthermore, the subspaces $\scrH_d$, $\scrH_a$ are orthogonal since vectors of $\scrH_d$ are in $\ker(\rho_p)$ where else vectors of $\scrH_a$ are in the span of $\Ran(\rho_p)$ for rays $p$.
	\\
	It is easy to verify that $\scrH_d$, $\scrH_a$ are closed under $\tau_\nu$, $\tau_\nu^*$ for all vertices $\nu$ and thus form sub-representations $\sigma_d$, $\sigma_a$ of $\sigma$. In particular, we can consider the restriction of $\tau_0$, $\tau_1$ to $\scrH_d$ which forms a P-pair $(\tau_0\restriction_{\scrH_d}, \tau_1\restriction_{\scrH_d})$. Analogously to Proposition \ref{subrep of sigma from fH proposition}, it is clear that the resulting P-representation is equivalent to $\sigma_d$ and thus $\sigma_d$ is Pythagorean. The same argument shows that $\sigma_a$ is Pythagorean.
	\\
	Considering these enlarged subspaces, we now have $\sigma = \sigma_d \oplus \sigma_a$. Indeed, consider $z \in \scrH$. Let $P$ be the set of all rays $p$ such that $\rho_p(z) \neq 0$. Note $P$ must be countable via a norm argument since $\{\rho_p\}_{p \in \cP}$ are pair-wise orthogonal projections. Set $z_a := \sum_{p\in P}\rho_p(z)$ and $z_d := z - z_a$ from which it becomes immediate that $z \in \scrH_d \oplus \scrH_a$.
	\\ 
	Finally, by construction it is obvious that $\sigma_d$ is diffuse. Moreover, we can apply the exact same proof of Theorem \ref{sigma_fW decompose theorem} to show that the theorem also holds for $\sigma_a$. This justifies $\sigma_d$ (resp. $\sigma_a$) being termed as the diffuse (resp. atomic) part of $\sigma$. 
	\\
	Therefore, from Theorem \ref{complete scrH decomposition theorem} and from the above, any Pythagorean representation can be decomposed into diffuse and atomic parts ($\scrH_d$ and $\scrH_a$, respectively) which have greatly contrasting structure. When $\fZ$ is finite-dimensional, the diffuse and atomic parts are induced from $\phi$-invariant subspaces of $\fH$, $\fU$ and $\fV \oplus \fW$, respectively. This is shown in Theorem \ref{complete scrH decomposition theorem}. When $\fZ$ is infinite-dimensional then the diffuse and atomic parts are no longer induced from subspaces of $\fH$ and rather must be directly defined as subspaces of $\scrH$ as done above. 
\end{remark}
}

\begin{proof}[Proof of Corollary \ref{cor:B}]
Consider a Pythagorean pair $(A,B)$ acting on $\fH$ and its associated representation $\sigma:F\act\scrH$.
Using the main results of our previous articles we only need to prove the two reverse implications that is: $\sigma$ weak-mixing (resp.~Ind-mixing) implies $\lim_nA^n\xi=\lim_nB^n\xi=0$ (resp.~$\lim_n p_n\xi=0$ for all ray $p$) for all vector $\xi\in\fH$ \cite{Brothier-Wijesena22}.
Assume there exists $\xi\in\fH$ so that $\lim_nA^n\xi\neq 0$. 
The proof of Theorem \ref{sigma_p decompose finite theorem} implies that the one-dimensional representation $\chi_\ell^\varphi$ of $F$ is contained in $\sigma$ for a certain $\varphi\in S^1$. Hence, $\sigma$ is not weakly mixing. 
A similar proof works by swapping $A$ by $B$ and the ray $\ell=\dots 00$ by $r=\dots 11$.
This proves the first statement.

Assume now that there exists $\xi\in\fH$ and a ray $p$ so that $\lim_n p_n\xi \neq 0$.
This implies that $\sigma$ contains either $\Ind_{F_p}^F\chi^p_\varphi$ for a certain $\varphi\in S^1$ or the quasi-regular representation $\lambda_{F/\widehat F_p}$. In both cases the representation $\sigma$ is not Ind-mixing with finishes the proof. 
\end{proof}


\newcommand{\etalchar}[1]{$^{#1}$}

\end{document}